\documentclass[a4paper,twoside,11pt,reqno]{amsart}

\usepackage[margin=1.1in]{geometry}
\usepackage{amsmath} 
\usepackage{amssymb} 
\usepackage{amsopn}
\usepackage{amsthm} 
\usepackage{amsfonts} 
\usepackage{mathrsfs}
\usepackage{stmaryrd}
\usepackage{accents} 
\usepackage{colonequals}
\usepackage{tikz-cd}
\usepackage{enumitem}
\usetikzlibrary{fit,matrix}
\usepackage{cite}
\usepackage{verbatim}
\usepackage[OT2,T1]{fontenc}
\usepackage[pdfpagelabels,pdftex]{hyperref}
\usepackage{setspace}
\usepackage{amssymb}
\usepackage{euscript}
\usepackage{cite}
\usepackage[all]{xy}
\usepackage{tabularx}

\theoremstyle{plain}
\usepackage{amsfonts}
\usepackage{graphicx}
\usepackage{amsmath}

\hypersetup{
  pdftitle={},
  pdfauthor={},
  pdfsubject={},
  pdfkeywords={},
  colorlinks=false,
  breaklinks=true,
  bookmarksopen=true,
  bookmarksnumbered=true,
  pdfpagemode=UseOutlines,
  plainpages=false}

\DeclareMathOperator{\GL}{GL}

\DeclareMathOperator\Gal{Gal}

\DeclareMathOperator\red{red}
\DeclareMathOperator\diag{diag}
\DeclareMathOperator\Nm{Nm}
\DeclareMathOperator\pr{pr}
\DeclareMathOperator\Fr{Fr}
\DeclareMathOperator\Hom{Hom}
\DeclareMathOperator\Tr{Tr}

\DeclareMathOperator\Ind{Ind}

\DeclareMathOperator\val{val}

\DeclareMathOperator\Ad{Ad}
\DeclareMathOperator\Res{Res}

\newcommand\from{\colon}
\newcommand\bG{\mathbb G}
\newcommand\bU{\mathbb U}
\newcommand\QQ{\mathbb Q}

\newcommand\bR{\mathbb R}
\newcommand\bT{\mathbb T}

\newcommand\bW{\mathbb W}
\newcommand\FF{\mathbb F}
\newcommand\cA{\mathcal A}
\newcommand\cI{\mathcal I}
\newcommand\cJ{\mathcal J}
\newcommand\bZ{\mathbb Z}
\newcommand\sT{\mathscr T}
\newcommand\F{\mathbb F_{q^n}}

\newcommand\Loc{\mathcal L}
\newcommand\cT{\mathcal T}
\newcommand\sB{\mathscr{B}}
\newcommand\sC{\mathscr{C}}
\newcommand\sD{\mathscr{D}}
\newcommand\sE{\mathscr{E}}
\newcommand\sS{\mathscr{S}}

\newcommand\bK{\mathbb{K}}
\newcommand\bL{\mathbb{L}}

\newcommand\cF{\mathcal F}
\newcommand\bN{\mathbb{N}}
\newcommand\cox{{\,{\rm cox}}}
\newcommand\xp{{\,{\rm sp}}}
\DeclareMathOperator\Ver{V}
\DeclareMathOperator\Perf{Perf}
\DeclareMathOperator\Spec{Spec}

\DeclareMathOperator\lcm{lcm}

\newcommand\cO{\mathcal O}
\newcommand\sL{\mathscr L}
\newcommand\bA{\mathbb A}
\newcommand\bP{\mathbb P}

\newcommand\trans{\intercal}

\DeclareMathOperator\Span{span}

\newcommand\Z{X_h^1}
\newcommand\Y{Y_h^1}

\makeatletter
\newtheorem*{rep@theorem}{\rep@title}
\newcommand{\newreptheorem}[2]{%
\newenvironment{rep#1}[1]{%
 \def\rep@title{#2 \ref{##1}}%
 \begin{rep@theorem}}%
 {\end{rep@theorem}}}
\makeatother

\newtheorem{thm}{Theorem}[subsection]
\newtheorem*{thm*}{Theorem}

\newtheorem*{cor*}{Corollary}
\newtheorem{lm}[thm]{Lemma}

\newtheorem{theorem}[thm]{Theorem}
\newtheorem{lemma}[thm]{Lemma}
\newtheorem{corollary}[thm]{Corollary}
\newtheorem{proposition}[thm]{Proposition}
\newtheorem{conjecture}[thm]{Conjecture}

\theoremstyle{definition}\newtheorem{example}[thm]{Example}

\theoremstyle{remark}

\newreptheorem{thm}{Theorem}
\newreptheorem{prop}{Proposition}
\newreptheorem{cor}{Corollary}

\theoremstyle{definition}

\newtheorem{definition}[thm]{Definition}

\theoremstyle{remark}

\newtheorem{remark}[thm]{Remark}

\newenvironment{pro*}[1][Proof]{{\it{#1:}} }{}

\newcounter{absatzcounter}[section]
\numberwithin{equation}{section}

\title{The Drinfeld stratification for $\GL_n$}
\author{Charlotte Chan and Alexander B.\ Ivanov}
\address{Department of Mathematics \\
MIT \\
Cambridge, MA 02139, USA}
\email{charchan@mit.edu}
\address{Mathematisches Institut \\ Universit\"at Bonn \\ Endenicher Allee 60 \\ 53115 Bonn, Germany}
\email{ivanov@math.uni-bonn.de}

\begin{document}

\maketitle

\begin{abstract}
We define a stratification of Deligne--Lusztig varieties and their parahoric analogues which we call the Drinfeld stratification. In the setting of inner forms of $\GL_n$, we study the cohomology of these strata and give a complete description of the unique closed stratum. We state precise conjectures on the representation-theoretic behavior of the stratification. We expect this stratification to play a central role in the investigation of geometric constructions of representations of $p$-adic groups.
\end{abstract}

\tableofcontents

\section{Introduction}

Like the classical upper half plane, its nonarchimedean analogue---the Drinfeld upper half plane---appears naturally in a wide range of number theoretic, representation theoretic, and algebro-geometric contexts. For finite fields, the $\ell$-adic \'etale cohomology of the Drinfeld upper half plane $\bP^1(\overline \FF_q) \smallsetminus \bP^1(\FF_q)$ with coefficients in nontrivial rank-1 local systems, is known to realize the cuspidal irreducible representations of $\GL_2(\FF_q)$. One can generalize this to $\GL_n(\FF_q)$ by projectivizing the complement of all rational sub-vector spaces of $V = \overline \FF_q^{\oplus n}$. This is the Drinfeld upper half space for $\FF_q$. In this paper, we consider a stratification of the Drinfeld upper half space induced by ``intermediate'' Drinfeld upper half spaces of smaller dimension sitting inside $\bP(V)$. 

In earlier work \cite{CI_ADLV}, we proved that for inner forms of $\GL_n$, Lusztig's \textit{loop Deligne--Lusztig set} \cite{Lusztig_79} is closely related to a finite-ring analogue of the Drinfeld upper half space. This allowed us to endow this set with a scheme structure (a statement which is still conjectural for any group outside $\GL_n$) and define its cohomology. Under a regularity condition, we prove in \cite{CI_ADLV} that the cohomology of loop Deligne--Lusztig varieties for inner forms of $\GL_n$ realize certain irreducible supercuspidal representation and describe these within the context of the local Langlands and Jacquet--Langlands correspondences. In \cite{CI_loopGLn}, we are able to relax this regularity condition to something quite general by using highly nontrivial input obtained by studying the cohomology of a stratification---the \textit{Drinfeld stratification}---which comes from the aforementioned stratification of the Drinfeld upper half space. 

In \cite{CI_MPDL}, we studied a class of varieties $X_h$ associated to parahoric subgroups of a(ny) connected reductive group $G$ which splits over an unramified extension. We define a stratification in this general context as well; the Drinfeld stratification is indexed by twisted Levi subgroups of $G$. The purpose of this paper is to initiate the study of these strata and, in due course, supply the necessary input for the final step in \cite{CI_loopGLn}. 

We focus on the setting of inner forms of $\GL_n$ and prove the first foundational representation-theoretic traits of the cohomology of the Drinfeld stratification: irreducibility (Theorem \ref{t:inner_prod}) and a special character formula (Proposition \ref{t:very_reg}). Using Theorem \ref{t:inner_prod}, in Section \ref{s:single_degree} we prove that the torus eigenspaces in the cohomology of the unique closed Drinfeld stratum is supported in a single (possibly non-middle) degree. Furthermore, this stratum is a \textit{maximal variety} in the sense of Boyarchenko--Weinstein \cite{BoyarchenkoW_16}: the number of rational points on the closed Drinfeld stratum attains its Weil--Deligne bound. Our analysis relies on techniques developed in \cite{Chan_siDL} in the special case of division algebras and gives some context for what we expect to be the role of maximal varieties in these Deligne--Lusztig varieties for $p$-adic groups. 

In practice, it is sometimes only possible to work directly with the Drinfeld stratification of the parahoric Deligne--Lusztig varieties $X_h$ instead of with the entire $X_h$. In this paper, for example, the maximality of the closed stratum allows us to give an exact formula (Corollary \ref{c:dimension}) for the formal degree of the associated representation of the $p$-adic group. We prove a comparison theorem in \cite{CI_loopGLn} relating the Euler characteristic of this stratum to that of $X_h$. This formal degree input, by comparison with Corwin--Moy--Sally \cite{CorwinMS_90}, allows us to obtain a geometric supercuspidality result in \cite{CI_loopGLn}.

We finish the paper with a precise formulation of some conjectures (Conjecture \ref{c:single_degree} and \ref{c:Xh}) which describe what we expect to be the shape of the cohomology of the Drinfeld stratification and its relation to the cohomology of loop Deligne--Lusztig varieties. In the Appendix, we present an analysis of the fibers of the natural projection maps $X_h \to X_{h-1}$; we believe this could be a possible approach to proving Conjecture \ref{c:Xh} and may be of independent interest. It would be interesting to see if the Drinfeld stratification plays a role in connections to orbits in finite Lie algebras, \`a la work of Chen \cite{Chen_2019}.

\subsection*{Acknowledgements}

We would like to thank Masao Oi and Michael Rapoport for enlightening conversations. The first author was partially supported by the DFG via the Leibniz Prize of Peter Scholze and an NSF Postdoctoral Research Fellowship, Award No.\ 1802905. The second author was supported by the DFG via the Leibniz Preis of Peter Scholze.

\section{Notation}

Let $k$ be a nonarchimedean local field with residue field $\FF_q$ and let $\breve k$ denote the completion of the maximal unramified extension of $k$. We write $\cO_{\breve k}$ and $\cO_k$ for the rings of integers of $\breve k$ and $k$, respectively. For any positive integer $m$, let $k_m$ denote the degree-$m$ unramified extension of $k$. Additionally, we define $L \colonequals k_n$.  

\textit{With the exception of Section \ref{s:drinfeld}}, in this paper, $G$ will be an inner form of $\GL_n$ defined over $k$. Let $\sigma \in \Gal(\breve k/k)$ denote the Frobenius which induces the $q$th power automorphism on the residue field $\overline \FF_q$. Abusing notation, we also let $\sigma$ denote the map $\GL_n(\breve k) \to \GL_n(\breve k)$ by applying $\sigma$ to each matrix entry. The inner forms of $\GL_n$ are indexed by an integer $0 \leq \kappa \leq n-1$; fix such an integer. Throughout the paper, we write $\kappa/n = k_0/n_0$, $(k_0,n_0) = 1$, and $\kappa = k_0 n'$. We define an element $b_\cox$ with $\val\det(b_\cox) = \kappa$ and set $G = J_{b_\cox}$ (the $\sigma$-stabilizer of $b_\cox$) with $k$-rational structure induced by the Frobenius
\begin{equation*}
F \from \GL_n(\breve k) \to \GL_n(\breve k), \qquad g \mapsto b_\cox \sigma(g) b_\cox^{-1}.
\end{equation*}
Note that $G \cong \GL_{n'}(D_{k_0/n_0})$, where $D_{k_0/n_0}$ denotes the division algebra over $k$ of dimension $n_0^2$ with Hasse invariant $k_0/n_0$. Let $T$ denote the set of diagonal matrices in $G$. Let $x$ be the unique point in the intersection $\cA(T) \cap \sB(\GL_n, \breve k)^F$. Note that $T(k) \cong L^\times$.

If $k$ has characteristic $p$, we let $\bW(A) = A[\![\pi]\!]$ for any $\FF_q$-algebra $A$ and write $[a_i]_{i \geq 0}$ to denote the element $\sum_{i \geq 0} a_i \pi^i \in \bW(A)$. If $k$ has characteristic zero, we let $\bW = W_{\cO_k} \times_{\Spec \cO_k} \Spec \FF_q$, where $W_{\cO_k}$ is the $\cO_k$-ring scheme of $\cO_k$-Witt vectors \cite[Section 1.2]{FarguesFontaine_book}. Following the notation of \textit{op.\ cit.\ } we write the elements of $\bW(A)$ as $[a_i]_{i \geq 0}$ where  $a_i \in A$. We may now talk about $\bW$ uniformly, regardless of the characteristic of $k$. As usual, we have the Frobenius and Verschiebung morphisms
\begin{align*}
&\sigma \from \bW \to \bW, \qquad [a_i]_{i \geq 0} \mapsto [a_i^q]_{i \geq 0}, \\
&V \from \bW \to \bW, \qquad [a_i]_{i \geq 0} \mapsto [0,a_0,a_1,\ldots].
\end{align*}
For any $h \in \bZ_{\geq 0}$, let $\bW_h = \bW/V^h \bW$ denote the corresponding truncated ring scheme.

For the benefit of the reader, we present a summary of the relationship between the various schemes appearing in this paper. Let $r \mid n'$ and let $h$ be a positive integer. We have 
\begin{equation*}
\begin{tikzcd}
S_h^{(r)} \ar[hook]{r} \ar{d} & S_h \ar{d} \\
X_h^{(r)} \ar[hook]{r} & X_h
\end{tikzcd}
\end{equation*}
where the vertical maps are quotients by an affine space. We have $X_h = X_h(b_{\cox}, b_{\cox})$ and when $X_h(b,w) \cong X_h$, then $X_h(b,w)^{(r)} \cong X_h^{(r)}$ by definition. We have
\begin{equation*}
X_h^{(r)} = \bigsqcup_{g \in \bG_h(\FF_q)/(\bL_h^{(r)}(\FF_q) \bG_h^1(\FF_q))} g \cdot (X_h \cap \bL_h^{(r)} \bG_h^1).
\end{equation*}
The $r$th Drinfeld stratum is 
\begin{equation*}
X_{h,r} = X_h^{(r)} \smallsetminus \bigcup_{r \mid s \mid n', \, r < s} X_h^{(s)}
\end{equation*}
and the closure of $X_{h,r}$ in $X_h$ is $X_h^{(r)}$. The unique closed Drinfeld stratum is the setting $r = n'$; in this case,
\begin{equation*}
X_{h,n'} = X_h^{(n')} = \bigsqcup_{g \in \bG_h(\FF_q)} g \cdot X_h^1, \qquad \text{where $X_h^1 = X_h \cap \bG_h^1$}.
\end{equation*}

For any positive integer $m$, we let $[l]_m$ denote the unique representative of $l\bZ/m\bZ$ in the set $\{1, \ldots, m\}$.

\section{The Drinfeld stratification}\label{s:drinfeld}

In this section only, we let $G$ be any reductive group over $k$ which splits over $\breve k$. Let $F$ denote a Frobenius associated to the $k$-rational structure on $G$. Fix a $k$-rational, $\breve k$-split maximal torus $T \subset G$, let $x \in \cA(T) \cap \sB(G, \breve k)^F$, and let $G_{x,0}$ be the attached parahoric model. Pick a $\breve k$-rational Borel subgroup $B \subset G_{\breve k}$ containing $T$ and let $U$ be the unipotent radical of $B$. Let $h \geq 1$ be an integer. There is a smooth affine group scheme $\bG_h$ over $\FF_q$ such that 
\begin{equation*}
\bG_h(\FF_q) = G_{x,0}(\cO_k)/G_{x,(h-1)+}(\cO_k), \quad\bG_h(\overline \FF_q) = G_{x,0}(\cO)/G_{x,(h-1)+}(\cO)
\end{equation*}
(see \cite[Section 2.5, 2.6]{CI_MPDL} for more details). The subgroups $T, U$ have associated subgroup schemes $\bT_h$, $\bU_h$ of $\bG_h$ such that 
\begin{equation*}
\bT_h(\FF_q) = (T(k) \cap G_{x,0}(\cO_k))/G_{x,(h-1)+}(\cO_k), \quad \bT_h(\overline \FF_q) = (T(\breve k) \cap G_{x,0}(\cO))/G_{x,(h-1)+}(\cO),
\end{equation*}
and $\bU_h(\overline \FF_q) = (U(\breve k) \cap G_{x,0}(\cO))/G_{x,(h-1)+}(\cO)$ (note here that $\bU_h$ is defined over $\overline \FF_q$ but may not be defined over $\FF_q$ as $U$ may not be $k$-rational).

\subsection{The schemes $S_h$ and $X_h$}

The central object of study is $X_h$:

\begin{definition}\label{def:Xh}
Define $\overline \FF_q$-scheme
\begin{equation*}
X_h \colonequals \{x \in \bG_h : x^{-1}F(x) \in \bU_h\}/(\bU_h \cap F^{-1}(\bU_h)).
\end{equation*}
$X_h$ comes with a natural action of $\bG_h(\FF_q) \times \bT_h(\FF_q)$ by left- and right-multiplication:
\begin{equation*}
(g,t) \cdot x = gxt, \qquad \text{for $(g,t) \in \bG_h(\FF_q) \times \bT_h(\FF_q),$ $x \in X_h$.}
\end{equation*}
\end{definition}

In some contexts, it will be more convenient to study $S_h$:

\begin{definition}
Define $\overline \FF_q$-scheme
\begin{equation*}
S_h \colonequals \{x \in \bG_h : x^{-1}F(x) \in \bU_h\}.
\end{equation*}
So, $S_h$ is the closed subscheme of $\bG_h$ obtained by pulling back $\bU_h$ along the (finite \'etale) Lang map $\bG_h \to \bG_h, \, g \mapsto g^{-1}F(g)$. Note that $S_h$ comes with the same natural action of $\bG_h(\FF_q) \times \bT_h(\FF_q)$ as $X_h$.
\end{definition}

Observe that since $\bU_h \cap F^{-1}(\bU_h)$ is an affine space, the cohomology of $X_h$ and $S_h$ differs only by a shift, and in particular, for any $\theta \from \bT_h(\FF_q) \to \overline \QQ_\ell^\times$, we have
\begin{equation*}
H_c^*(X_h, \overline \QQ_\ell)[\theta] = H_c^*(S_h, \overline \QQ_\ell)[\theta]
\end{equation*}
as elements of the Grothendieck group of $\bG_h(\FF_q)$.

\subsection{The scheme $X_h(b,w)$}\label{s:Xhbw}

In this section, we further assume that $G$ is quasisplit over $k$ and $B \subset G$ is $k$-rational. In this section, we write $\sigma = F$ for our $q$-Frobenius associated to the $k$-rational structure on $G$.

\begin{definition}
Let $b,w \in G(\breve k)$. Assume that $b,w$ normalize both subgroups $G_{x,0}(\cO_{\breve k})$, $G_{x,(h-1)+}(\cO_{\breve k})$ of $G(\breve k)$, and additionally assume that $w$ normalizes $T(\breve k)$. Define the $\overline \FF_q$-scheme
\begin{equation*}
X_h(b,w) \colonequals \{x \in \bG_h : x^{-1} b \sigma(x) \in \bU_h w \bU_h\}/\bU_h,
\end{equation*}
where the condition $x^{-1} b \sigma(x) \in \bU_h w \bU_h$ means the following: For any lift $\widetilde x \in G$ of $x \in \bG_h$, the element $\widetilde x^{-1} b \sigma(\widetilde x)$ is an element of $(U \cap G_{x,0}) w (U \cap G_{x,0}) G_{x,(h-1)+} \subset G$. More precisely, $X_h(b,w) = S_h(b,w)/\bU_h$, where $S_h(b,w)$ is the reduced $\overline \FF_q$-subscheme of $\bG_h$ such that $S_h(b,w)(\overline \FF_q)$ is equal to the image of $\{x \in G_{x,0}(\cO_{\breve k}) : x^{-1} b \sigma(x) \in (U(\breve k) \cap G_{x,0}(\cO_{\breve k})) w (U(\breve k) \cap G_{x,0}(\cO_{\breve k})) G_{x,(h-1)+}(\cO_{\breve k})\}$ in $\bG_h(\overline \FF_q)$. 
Note that $X_h(b,w)$ comes with a natural action by left- and right-multiplication of $G_h(b)$ and $T_h(w)$, where $G_h(b) \subset \bG_h(\overline \FF_q)$ is the image of $\{g \in G_{x,0}(\cO_{\breve k}) : b \sigma(g) b^{-1} = g\}$ and $T_h(w) \subset \bT_h(\overline \FF_q)$ is the image of $\{t \in T(\breve k) \cap G_{x,0}(\cO_{\breve k}) : w \sigma(t) w^{-1} = t\}$.
\end{definition}

The next lemma is a one-line computation; we record it for easy reference.

\begin{lemma}\label{l:change b}
Let $\gamma \in G_{x,0}(\cO_{\breve k})$. Then we have an isomorphism
\begin{equation*}
X_h(b,w) \to X_h(\gamma^{-1} b \sigma(\gamma),w), \qquad x \mapsto \overline \gamma x,
\end{equation*}
where $\overline \gamma$ is the image of $\gamma$ in the quotient $\bG_h(\overline \FF_q)$.
\end{lemma}

\begin{lemma}\label{l:Xbw to Xh}
Consider the morphism $F \from (\bG_h)_{\overline \FF_q} \to (\bG_h)_{\overline \FF_q}$ given by $g \mapsto b \sigma(g) b^{-1}$. If $w G_{x,0} b^{-1} = G_{x,0}$ and $F(\bU_h) = w \bU_h b^{-1}$, then 
\begin{equation*}
X_h(b,w) = X_h,
\end{equation*}
where $X_h$ is the $\overline \FF_q$-scheme in Definition \ref{def:Xh} associated to the group scheme $(\bG_h)_{\overline \FF_q}$ endowed with the $\FF_q$-rational structure associated to the $q$-Frobenius $F$.
\end{lemma}

\begin{proof}
We have
\begin{align*}
X_h(b,w) 
&= \{x \in \bG_h : x^{-1} F(x) \in \bU_h w \bU_h b^{-1}\}/\bU_h \\
&= \{x \in \bG_h : x^{-1} F(x) \in \bU_h F(\bU_h)\}/\bU_h \\
&= \{x \in \bG_h : x^{-1} F(x) \in \bU_h\}/(\bU_h \cap F^{-1} \bU_h) = X_h. \qedhere
\end{align*}
\end{proof}

\subsection{The Drinfeld stratification for $S_h$}\label{s:gen_Drin_S}

Let $L$ be a $k$-rational twisted Levi subgroup of $G$ and assume that $L$ contains $T$. Recall a $k$-rational subgroup $L \subset G$ is a \textit{twisted Levi} if $L_{\overline k}$ is a Levi subgroup of $G_{\overline k}$. Note also that the condition that $L$ contains $T$ forces $L$ to be split over $\breve k$. Following \cite[Section 2.6]{CI_MPDL}, the schematic closure $L_{x}$ in $G_{x,0}$ is a closed subgroup scheme defined over $\cO_k$. Applying the ``positive loop'' functor to $L_{x}$, for $h \in \bZ_{> 0}$, we can define a $\overline \FF_q$-scheme $\bL_h$ such that $\bL_h(\overline \FF_q)$ is the image of $L_x(\cO_{\breve k})$ in $\bG_h(\overline \FF_q)$.

\begin{definition}[Drinfeld stratification for $S_h$]\label{d:drinfeld Sh}
Define
\begin{equation*}
S_h^{(L)} \colonequals \{x \in \bG_h : x^{-1}F(x) \in (\bL_h \cap \bU_h) \bU_h^1\},
\end{equation*}
where $(\bL_h \cap \bU_h) \bU_h^1 \subset \bU_h$ is the subgroup generated by $\bL_h \cap \bU_h$ and $\bU_h^1$ (which is normalized by $\bL_h \cap \bU_h$). Note that the subscheme $S_h^{(L)}$ of $S_h$ is closed and stable under the action of $\bG_h(\FF_q) \times \bT_h(\FF_q)$.
\end{definition}

\begin{definition}[Drinfeld stratification for $X_h$, $X_h(b,w)$]\label{d:drinfeld Sh}
Define $X_h^{(L)}$ to be the image of $S_h^{(L)}$ under the surjection $S_h \to X_h$. Recall that for any $\gamma \in G_{x,0}(\cO_{\breve k})$, we have $X_h(b,w) \cong X_h(\gamma^{-1} b \sigma(\gamma),w)$ via $x \mapsto \overline\gamma x$. If $F(\bU_h) = w \bU_h b^{-1}$, then $X_h = X_h(b,w)$; in this setting, let $X_h(\gamma^{-1} b \sigma(\gamma),w)^{(L)}$ denote the image of $X_h^{(L)}$.
\end{definition}

Another subscheme of $S_h$ which we may associate to the twisted Levi subgroup $L \subset G$ is the intersection
\begin{align*}
S_h \cap \bL_h \bG_h^1 
&= \{x \in \bL_h \bG_h^1 : x^{-1}F(x) \in \bU_h\} \\
&= \{x \in \bL_h \bG_h^1 : x^{-1}F(x) \in (\bL_h \cap \bU_h) \bU_h^1\},
\end{align*}
where $\bL_h \bG_h^1$ denotes the subgroup scheme of $\bG_h$ generated by $\bL_h$ and $\bG_h^1$ (which is normalized by $\bL_h$). Note that $S_h \cap \bL_h \bG_h^1$ is stable under the action of $\bL_h(\FF_q) \bG_h^1(\FF_q) \times \bT_h(\FF_q)$.

\begin{lemma}\label{l:ShP}
Let $L$ be a $k$-rational twisted Levi subgroup of $G$ containing $T$. Then
\begin{equation*}
S_h^{(L)} = \bigsqcup_{\gamma \in \bG_h(\FF_q)/(\bL_{h}(\FF_q)\bG_h^1(\FF_q))} \gamma \cdot (S_h \cap \bL_{h} \bG_h^1).
\end{equation*}
\end{lemma}

\begin{proof}
Pick any $u \in \bU_{h}(\overline \FF_q) \bU_h^1(\overline \FF_q)$. By surjectivity of the Lang map, there exists $x \in \bL_{h}(\overline \FF_q) \bG_h^1(\overline \FF_q)$ and $y \in \bG_h(\overline \FF_q)$ such that $x^{-1} F(x) = u$ and $y^{-1} F(y) = u$. Then 
\begin{equation*}
(xy^{-1})^{-1} F(xy^{-1}) = y x^{-1} F(x) = F(y)^{-1} = y u F(y)^{-1} = y u u^{-1} y^{-1} = 1.
\end{equation*}
Therefore $xy^{-1} \in \bG_h(\FF_q)$. The assertion now follows from the fact that the stabilizer of $S_h \cap \bL_{h} \bG_h^1$ in $\bG_h(\FF_q) \times \bT_h(\FF_q)$ is $\bL_{h}(\FF_q) \bG_h^1(\FF_q) \times \bT_h(\FF_q)$.
\end{proof}

By Lemma \ref{l:ShP}, we see:

\begin{lemma}\label{l:geom parabolic induction}
If $L$ is a twisted Levi subgroup of $G$ containing $T$, then for any character $\theta \from \bT_h(\FF_q) \to \overline \QQ_\ell^\times$ and for all $i \geq 0$,
\begin{equation*}
H_c^i(S_h^{(L)}, \overline \QQ_\ell)[\theta] \cong \Ind_{\bL_{h}(\FF_q) \bG_h^1(\FF_q)}^{\bG_h(\FF_q)}\big(H_c^i(S_h \cap \bL_{h} \bG_h^1, \overline \QQ_\ell)[\theta]\big).
\end{equation*}
\end{lemma}

\section{The case of $\GL_n$}

In this paper, study the varieties introduced in Section \ref{s:drinfeld} in the special case when $G$ is an inner form of $\GL_n$. We emphasize that these varieties $S_h, X_h, X_h(b,w)$---at least \textit{a priori}---depend on a choice of Borel subgroup containing the torus at hand. From now until the end of the paper, we work with the varieties associated with the Borel subgroup explicitly chosen in Section \ref{s:explicit}. We explicate the Drinfeld stratification for $S_h$, $X_h$, and certain $X_h(b,w)$, and give a description in terms of Drinfeld upper half-spaces and $\sL_h \subset \bW_h^{\oplus n}$, a finite-ring analogue of an isocrystal.

Let $\sigma \in \Gal(\breve k/k)$ denote the $q$-Frobenius induces $x \mapsto x^q$ on the residue field $\overline \FF_q$. Abusing notation, also let
\begin{equation*}
\sigma \from \GL_n(\breve k) \to \GL_n(\breve k), \qquad (M_{i,j})_{i,j=1,\ldots,n} \mapsto (\sigma(M_{i,j}))_{i,j=1,\ldots,n}.
\end{equation*}
For $b \in \GL_n(\breve k)$, let $J_b$ be the $\sigma$-stabilizer of $b$: for any $k$-algebra $R$, 
\begin{equation*}
J_b(R) \colonequals \{g \in \GL_n(R \otimes_k \breve k) : g^{-1} b \sigma(g) = b\}.
\end{equation*}
$J_b$ is an inner form of the centralizer of the Newton point of $b$ (which is a Levi subgroup of $\GL_n$), and we may consider
\begin{equation*}
\GL_n(\breve k) \to \GL_n(\breve k), \qquad g \mapsto b \sigma(g) b^{-1}
\end{equation*}
to be an associated $q$-Frobenius for the $k$-rational structure on $J_b$. If $b$ is \textit{basic} (i.e.\ the Newton point of $b$ is central), then $J_b$ is an inner form of $\GL_n$ and every inner form arises in this way. If $\kappa = \kappa_{\GL_n}(b) \colonequals \val(\det(b))$, then then $J_b(k) \cong \GL_{n'}(D_{k_0/n_0})$ where $\kappa/n = k_0/n_0$, $(k_0,n_0) = 1$, and $\kappa = k_0 n'$. Note that the isomorphism class of $J_b$ only depends on the $\sigma$-conjugacy class $[b] \colonequals \{g^{-1} b \sigma(g) : g \in \GL_n(\breve k)\}$. 

Fix an integer $0 \leq \kappa \leq n-1$. In the next sections, we will focus on representatives $b$ revolving around the \textit{Coxeter representative} (Def \ref{d:coxeter}) and give explicit descriptions of the varieties $X_h$, $X_h(b,w)$, and their Drinfeld stratifications $\{X_h^{(r)}\}$, $\{X_h(b,w)^{(r)}\}$, where $r$ runs over the divisors of $n'$. The $X_h^{(r)}, X_h(b,w)^{(r)}$ are closed subvarieties of $X_h, X_h(b,w)$; we call the $r$th Drinfeld stratum 
\begin{equation}\label{e:rth stratum}
X_h^{(r)} \smallsetminus \Big(\bigcup_{\substack{r < r' \leq n' \\ r \mid r' \mid n'}} X_h^{(r')}\Big), \qquad
X_h(b,w)^{(r)} \smallsetminus \Big(\bigcup_{\substack{r < r' \leq n' \\ r \mid r' \mid n'}} X_h(b,w)^{(r')}\Big)
\end{equation}
so that the closure of the $r$th Drinfeld stratum is $X_h^{(r)}$, $X_h(b,w)^{(r)}$.

\subsection{Explicit parahoric subgroups of $G$}

Set
\begin{equation*}
b_0 \colonequals \left(\begin{matrix} 0 & 1 \\ 1_{n-1} & 0 \end{matrix}\right), \quad \text{and} \quad t_{\kappa,n} \colonequals
\begin{cases}
\diag(\underbrace{1, \ldots, 1}_{n-\kappa}, \underbrace{\varpi, \ldots, \varpi}_\kappa) & \text{if $(\kappa,n)=1$,} \\
\diag(\underbrace{t_{k_0,n_0}, \ldots, t_{k_0,n_0}}_{n'}) & \text{otherwise.}
\end{cases}
\end{equation*}
Fix an integer $e_{\kappa,n}$ such that $(e_{\kappa,n},n) = 1$ and $e_{\kappa,n} \equiv k_0$ mod $n_0$. If $\kappa$ divides $n$ (i.e.\ $k_0 = 1$), we always take $e_{\kappa,n} = 1$.

\begin{definition}\label{d:coxeter}
The \textit{Coxeter-type representative} attached to $\kappa$ is $b_{\cox} \colonequals b_0^{e_{\kappa,n}} \cdot t_{\kappa,n}$.
\end{definition}

Define $G \colonequals J_{b_{\cox}}$ with Frobenius 
\begin{equation*}
F \from \GL_n(\breve k) \to \GL_n(\breve k), \qquad g \mapsto b_{\cox} \sigma(g) b_{\cox}^{-1}
\end{equation*}
and define $T$ to be the set of diagonal matrices in $G$. Observe that $T$ is $F$-stable and that $T(k) \cong L^\times$. Since $T$ is elliptic, the intersection $\cA(T) \cap \sB(G,\breve k)^F$ consists of a single point $x$ and $G_{x,0}$ consists of invertible matrices $(A_{i,j})_{1 \leq i,j \leq n}$ where
\begin{equation*}
A_{i,j} \in \begin{cases}
\bW & \text{if $[i]_{n_0} \geq [j]_{n_0}$}, \\
V\bW & \text{if $[i]_{n_0} < [j]_{n_0}$}.
\end{cases}
\end{equation*}

For technical reasons, we will need to write down the relationship between the Coxeter element $b_0^{e_{\kappa,n}}$ and the Coxeter element $b_0$. Define $\gamma$ to be the unique permutation matrix which a) fixes the first elementary column vector and b) has the property that
\begin{equation}\label{e:gamma}
\gamma b_0^{e_{\kappa,n}} \gamma^{-1} = b_0.
\end{equation}
Note that one can express $\gamma$ explicitly as well: it corresponds to the permutation of $\{1, \ldots, n\}$ given by
\begin{equation*}
i \mapsto [(i-1) e_{\kappa,n} + 1]_n.
\end{equation*}

\subsection{An explicit description of $X_h$}\label{s:explicit}

The choices in this section are the same as those from \cite[Section 7.7]{CI_ADLV}. In the setting of division algebras, these choices also appear in \cite{Chan_DLII, Chan_siDL}.

Let $U_{\rm up}, U_{\rm low} \subset G_{\breve k}$ denote the subgroups of unipotent upper- and lower-triangular matrices. Define
\begin{equation}\label{e:U}
U \colonequals \gamma^{-1} U_{\rm low} \gamma, \qquad U^- \colonequals \gamma^{-1} U_{\rm up} \gamma.
\end{equation}
Let $\bU_h, \bU_h^-$ be the associate subgroup schemes of $\bG_h$. By \cite[Lemma 7.12]{CI_ADLV}, we have an isomorphism of $\overline \FF_q$-schemes
\begin{equation}\label{e:section}
(\bU_h \cap F \bU_h^-) \times (\bU_h \cap F^{-1} \bU_h) \to \bU_h, \qquad (g,x) \mapsto x^{-1} g F(x).
\end{equation}
We will need a refinement of this isomorphism later (see Lemma \ref{l:U_{h,r}}). Define 
\begin{equation*}
\sL_{h} \colonequals \big(\bW_h \oplus (\bW_{h-1})^{\oplus n_0-1}\big)^{\oplus n'}.
\end{equation*}
Write $t_{\kappa,n} = \diag\{t_1, \ldots, t_n\}$. Viewing any $v \in \sL_h$ as a column vector, consider the associated matrix 
\begin{align}\label{e:lambda}
\lambda(v) &\colonequals \left(v_1 \, \Big| \, v_2 \, \Big| \, v_3 \, \Big| \, \cdots \, \Big| \, v_n\right), \\
\label{e:lambda i}
\text{where } v_{[ie_{\kappa,n}+1]_n} &\colonequals 
\varpi^{-\lfloor i k_0/n_0 \rfloor} \cdot (b\sigma)^i(v)
\text{ for $0 \leq i \leq n-1$}.
\end{align}

\begin{lemma}\label{l:description}
We have
\begin{align*}
X_h &= \{x \in \bG_h : x^{-1} F(x) \in \bU_h \cap F(\bU_h^-)\} \\
&= \{\lambda(v) \in \bG_h : \text{$v \in \sL_h$ and $\sigma(\det \lambda(v)) = \det \lambda(v)$}\}. 
\end{align*}
\end{lemma}

\begin{proof}
The first equality holds by \eqref{e:section}. The second equality is an explicit computation: in the division algebra setting, see \cite[Equation (2.2)]{Lusztig_79}, \cite[Lemma 4.4]{Boyarchenko_12}, \cite[Section 2.1]{Chan_siDL}; in the present setting of arbitrary inner forms of $\GL_n$, see \cite[Section 6]{CI_ADLV}. We give an exposition of these works here.

By direct computation, $\bU_h \cap F(\bU_h^-)$ is the subgroup of $\bG_h$ consisting of unipotent lower-triangular matrices whose entries outside the first column vanish:
\begin{equation*}
\bU_h \cap F(\bU_h^-) = \left\{\left(\begin{smallmatrix} 1 & & & \\ * & 1 & & \\ \vdots &  & \ddots &  \\ * &  &  & 1 \end{smallmatrix}\right)\right\}.
\end{equation*} 
Suppose that $x \in \bG_h$ is such that $x^{-1} F(x) \in \bU_h \cap F(\bU_h^-)$ and let $x_i$ denote the $i$th column of $x$. Then recalling that $b = b_0^{e_{\kappa,n}} t_{\kappa,n}$ and writing $t_{\kappa,n} = \diag\{t_1, \ldots, t_n\}$, we have
\begin{align*}
F(x) &= \left(
b\sigma(x_1) \, \Big| \, b\sigma(x_2) \, \Big | \, \cdots \, \Big| b\sigma(x_n)
\right) b^{-1} \\
&= \left(
t_{[1-e_{\kappa,n}]}^{-1}b\sigma(x_{[1-e_{\kappa,n}]_n}) \, \Big| \, t_{[2-e_{\kappa,n}]}^{-1}b\sigma(x_{[2-e_{\kappa,n}]_n}) \, \Big | \, \cdots \, \Big| t_{[n-e_{\kappa,n}]}^{-1}b\sigma(x_{[n-e_{\kappa,n}]_n})
\right).
\end{align*}
On the other hand, we have
\begin{equation*}
x (\bU_h \cap F(\bU_h^-)) = \left(* \,\Big|\, x_2 \,\Big|\, x_3 \,\Big|\, \cdots \,\Big|\, x_n\right).
\end{equation*}
Comparing columns, we see that each $x_i$ is uniquely determined by $x_1$ and that we have 
\begin{align}\label{e:x1 determines}
x_{[(n-1)e_{\kappa,n}+1]_n} 
&= t_{[(n-2)e_{\kappa,n}+1]_n}^{-1}  b\sigma(x_{[(n-2)e_{\kappa,n}+1]_n}) \\
&= t_{[(n-2)e_{\kappa,n}+1]_n}^{-1}t_{[(n-3)e_{\kappa,n}+1]_n}^{-1} b\sigma(b\sigma(x_{[(n-3)e_{\kappa,n}+1]_n}))  \\
&= t_{[(n-2)e_{\kappa,n}+1]_n}^{-1}t_{[(n-3)e_{\kappa,n}+1]_n}^{-1} \cdots t_{1}^{-1}  (b\sigma)^{n-1}(x_1).
\end{align}
Using Lemma \ref{l:ti contribution}, we now see that $x = \lambda(x_1)$, and finally, the condition $\sigma(\det \lambda(x)) = \det \lambda(x)$ comes from observation that $x^{-1}F(x)$ must have determinant $1$.
\end{proof}

\begin{lemma}\label{l:ti contribution}
For $1 \leq i \leq n-1$,
\begin{equation*}
\prod_{j=0}^{i-1} t_{[je_{\kappa,n}+1]_n} = \varpi^{\lfloor ik_0/n_0 \rfloor}.
\end{equation*}
\end{lemma}

\begin{proof}
We prove this by induction on $i$. If $i = 1$, then by definition we have $t_1 = 1$, so this proves the base case. Now assume that the lemma holds for $i$. We would like to prove that it holds for $i+1$. This means we need to prove two assertions:
\begin{enumerate}[label=(\alph*)]
\item
If $\lfloor (i+1)k_0/n_0 \rfloor > \lfloor ik_0/n_0 \rfloor$, then $t_{[ie_{\kappa,n}+1]_n} = \varpi$.
\item
If $\lfloor (i+1)k_0/n_0 \rfloor = \lfloor ik_0/n_0 \rfloor$, then $t_{[ie_{\kappa,n}+1]_n} = \varpi$.
\end{enumerate}
The arguments are very similar. For (a): Observe that $\lfloor (i+1)k_0/n_0 \rfloor > \lfloor ik_0/n_0 \rfloor$ if and only if $n_0 > [i e_{\kappa,n}]_{n_0} \geq n_0 - k_0$ since $e_{\kappa,n} \equiv k_0$ mod $n_0$. But this happens if and only if $[i e_{\kappa,n} + 1]_{n_0} > n_0 - k_0$, which means $t_{[i e_{\kappa,n}+1]_n} = \varpi$ by definition. For (b): Observe that $\lfloor (i+1)k_0/n_0 \rfloor = \lfloor ik_0/n_0 \rfloor$ if and only if $[i e_{\kappa,n}]_{n_0} = n_0$ or $[i e_{\kappa,n}]_{n_0} < n_0 - k_0$. But this happens if and only if $[i e_{\kappa,n}+1]_{n_0} \leq n_0 - k_0$, which means that $t_{[i e_{\kappa,n}+1]_n} = 1$ by definition.
\end{proof}

\subsection{The Drinfeld stratification of $X_h$}

For any divisor $r \mid n'$, define $L^{(r)}$ to be the twisted Levi subgroup of $G$ consisting of matrices $(A_{i,j})_{1 \leq i,j \leq n}$ such that $A_{i,j} = 0$ unless $i - j \equiv 0$ modulo $rn_0$. Note that $L^{(r)} \cong \Res_{k_{\frac{n}{r}}/k}(\GL_r)$ and that every $k$-rational twisted Levi subgroup of $G$ containing $T$ is conjugate to $L^{(r)}$ for some $r \mid n'$. Let $\bL_h^{(r)}$ denote subgroup of $\bG_h$ associated to $L^{(r)}$ and define 
\begin{equation*}
\bU_{h,r} \colonequals \bL_h^{(r)} \bU_h^1 \cap \bU_h, \qquad
\bU_{h,r}^- \colonequals \bL_h^{(r)} \bU_h^{-,1} \cap \bU_h^-.
\end{equation*}

\begin{lemma}\label{l:U_{h,r}}
The isomorphism of $\overline \FF_q$-schemes \eqref{e:section}
\begin{equation*}
(\bU_h \cap F\bU_h^-) \times (\bU_h \cap F^{-1}\bU_h) \to \bU_h, \qquad (g,x) \mapsto x^{-1} g F(x)
\end{equation*}
restricts to an isomorphism
\begin{equation*}
(\bU_{h,r} \cap F\bU_{h,r}^-) \times (\bU_{h,r} \cap F^{-1} \bU_{h,r}) \to \bU_{h,r}. 
\end{equation*}
\end{lemma}

\begin{proof}
This lemma is a refinement of \cite[Lemma 7.12]{CI_ADLV}. Recall that $\gamma \bU_h \gamma^{-1}$ and $\gamma \bU_h^- \gamma^{-1}$ are the subgroups consisting of unipotent lower- and upper-triangular matrices in $\bG_h$. Recall also that $F(g) = b_0^{e_{\kappa,n}} t_{\kappa,n} \sigma(g) t_{\kappa,n}^{-1} b_0^{e_{\kappa,n}}$. Conjugating \eqref{e:section}, which is proven in \textit{op.\ cit.}, we have
\begin{equation*}
(\gamma \bU_h \gamma^{-1} \cap F_0(\gamma \bU_h^- \gamma^{-1})) \times (\gamma \bU_h \gamma^{-1} \cap F_0^{-1}(\gamma \bU_h \gamma^{-1})) \to \gamma \bU_h \gamma^{-1},
\end{equation*}
where $F_0(g) = (b_0 \gamma t_{\kappa,n} \gamma^{-1}) \sigma(g) (b_0 \gamma t_{\kappa,n} \gamma^{-1})^{-1}$. Since $\gamma L^{(r)} \gamma^{-1} = L^{(r)}$, to prove the lemma, it suffices to show that if $(g,x) \in (\gamma \bU_h \gamma^{-1} \cap F_0(\gamma \bU_h^- \gamma^{-1})) \times (\gamma \bU_h \gamma^{-1} \cap F_0^{-1}(\gamma \bU_h \gamma^{-1}))$ is such that $A = x^{-1} g F(x) \in \gamma \bU_{h,r} \gamma^{-1}$, then 
\begin{equation}\label{e:x,g gamma}
(g,x) \in (\gamma \bU_{h,r} \gamma^{-1} \cap F_0(\gamma\bU_{h,r}^- \gamma^{-1})) \times (\gamma \bU_{h,r} \gamma^{-1} \cap F_0^{-1}(\gamma \bU_{h,r} \gamma^{-1})).
\end{equation}

Keeping the same notation as in \cite[Lemma 7.12]{CI_ADLV}, write
\begin{equation*}
x =
\left(\begin{matrix}
1 & 0 & 0 & \cdots & \cdots & 0 \\
b_{21} & 1 & 0 & \cdots & \cdots & 0 \\
b_{31} & b_{32} & 1 & \ddots & & \vdots \\
\vdots & & \ddots & \ddots & 0 & \vdots \\
b_{n-1,1} & b_{n-1,2} & \cdots & b_{n-1,n-2} & 1 & 0 \\
0 & \cdots & \cdots & 0 & 0 & 1
\end{matrix}\right), \qquad
g =
\left(\begin{matrix}
1 & 0 & 0 & \cdots & 0 \\
c_1 & 1 & 0 & \cdots & 0 \\
c_2 & 0 & 1 & \ddots & \vdots \\
\vdots & \vdots & \ddots & \ddots & 0 \\
c_{n-1} & 0 & \cdots & 0 & 1
\end{matrix}\right).
\end{equation*} 
Let $\gamma t_{\kappa,n} \gamma^{-1} = \diag(s_1, s_2, \ldots, s_n)$ so that we have
\begin{equation*}
F_0(x) = 
\left(\begin{matrix}
1 & 0 & 0 & 0 & \cdots & 0 \\
0 & 1 & 0 & 0 & \cdots & 0 \\
0 & \sigma(b_{21}) s_2/s_1 & 1 & 0 &  & 0 \\
0 & \sigma(b_{31}) s_3/s_1 & \sigma(b_{32}) s_3/s_2 & 1 & \ddots &\vdots \\
\vdots & \vdots & \ddots & \ddots & 1 & 0 \\
0 & \sigma(b_{n-1,1}) s_{n-1}/s_1 & \sigma(b_{n-1,2}) s_{n-1}/s_2 & \cdots & \sigma(b_{n-1,n-2}) s_{n-1}/s_{n-2} & 1
\end{matrix}\right).
\end{equation*}
As in \cite[Lemma 7.12]{CI_ADLV}, we see that the $(i,j)$th entry of $g F_0(x)$ is
\begin{equation}\label{e:RHS ij}
(g F_0(x))_{i,j} = 
\begin{cases}
1 & \text{if $i = j$,} \\
0 & \text{if $i < j$,} \\
c_{i-1} & \text{if $i > j = 1$,} \\
\sigma(b_{i-1,j-1}) s_{i-1}/s_{j-1} & \text{if $i > j > 1$.}
\end{cases}.
\end{equation}
We also compute the $(i,j)$th entry of $xA$ when $A = (a_{i,j})_{i,j} \in \gamma\bU_h\gamma^{-1}$:
\begin{equation}\label{e:LHS ij}
(xA)_{i,j} = 
\begin{cases}
1 & \text{if $i = j$,} \\
0 & \text{if $i < j$,} \\
b_{ij} + \sum_{k=j+1}^{i-1} b_{ik} a_{kj} + a_{ij} & \text{if $j < i \leq n-1$,} \\
a_{nj} & \text{if $j < i = n$.} 
\end{cases}
\end{equation}
We now have $n^2$ equations given by \eqref{e:RHS ij} $=$ \eqref{e:LHS ij}, viewed as equations in the variables $b_{i,j}$ and $c_i$. Let $\overline b_{i,j}$, $\overline c_i$, $\overline a_{i,j}$ denote the images of $b_{i,j}$, $c_i$, $a_{i,j}$ in $\bW_1$. In particular, we have the following: 
\begin{align}\label{e:i=n}
\overline b_{n-1,j-1} = 0 \qquad \Longleftrightarrow \qquad \overline a_{n,j} = 0,
\end{align}
and for $1 < j < i < n$,
\begin{equation}\label{e:i<n}
\overline b_{i-1,j-1} = 0 \qquad \Longleftrightarrow \qquad \overline b_{i,j} + \sum_{k=j+1}^{i-1} \overline b_{i,k} \overline a_{k,j} + \overline a_{i,j} = 0.
\end{equation}
Assume now that $A \in \gamma\bU_{h,r}\gamma^{-1} = \gamma(\bL_h^{(r)} \bU_h^1 \cap \bU_h)\gamma^{-1}$. Then $\overline a_{i,j} = 0$ if $rn_0 \nmid i-j$. From \eqref{e:i=n} we see that $\overline b_{n-1,j-1} = 0$ if $rn_0 \nmid n-j = (n-1)-(j-1)$. We now proceed by (decreasing) induction on $i$. If $i,j$ are such that $1 < j < i < n$ and $rn_0 \nmid i-j$, then necessarily either $rn_0 \nmid i-k$ or $rn_0 \nmid k-j$, and therefore each term in the sum on the right-hand side of \eqref{e:i<n} is zero, and so $\overline b_{i-1,j-1} = 0$. 

We have therefore shown that $x \in \gamma(\bL_h^{(r)} \bU_h^1 \cap \bU_h)\gamma^{-1} \cap F^{-1}(\gamma \bU_h\gamma^{-1})$. In particular, $F(x) \in \gamma\bU_h\gamma^{-1}$. Since $\bL_h^{(r)}$ is $F$-stable, we have that $F(\overline x) \in \bL_1^{(r)}$ and therefore $F(x) \in \gamma(\bU_h \cap \bL_h^{(r)} \bU_h^1)\gamma^{-1}$. Hence $x \in \gamma(\bU_{h,r} \cap F^{-1}\bU_{h,r}) \gamma^{-1}$.

Now since $\overline A, \overline x \in \bL_1^{(r)}$, we must have $\overline g \in \bL_1^{(r)}$. Since $g \in \gamma\bU_h\gamma^{-1}$, we must have $g \in \gamma(\bL_h^{(r)} \bU_h^1 \cap \bU_h)\gamma^{-1} = \gamma\bU_{h,r}\gamma^{-1}$, and since $g \in F(\gamma \bU_h^- \gamma^{-1})$, we must have $g \in F(\gamma(\bL_h^{(r)} \bU_h^{-,1} \cap \bU_h^-)\gamma^{-1})$. Hence $g \in \gamma\bU_{h,r}\gamma^{-1} \cap F(\gamma\bU_{h,r}^-\gamma^{-1})$. This establishes \eqref{e:x,g gamma} and finishes the proof of the lemma.
\end{proof}

\begin{definition}[Drinfeld stratification for $X_h$]\label{d:drinfeld GLn}
For each divisor $r \mid n'$, we define
\begin{align*}
S_h^{(r)} &\colonequals \{x \in \bG_h : x^{-1} F(x) \in \bU_{h,r}\}, \\
X_h^{(r)} &\colonequals \{x \in \bG_h : x^{-1} F(x) \in \bU_{h,r}\}/(\bU_{h,r} \cap F^{-1}\bU_{h,r}) \\
&= \{x \in \bG_h : x^{-1} F(x) \in \bU_{h,r} \cap F\bU_{h,r}^-\},
\end{align*}
where the second equality in $X_h^{(r)}$ holds by Lemma \ref{s:gen_Drin_S}.
\end{definition}

Note that $S_h^{(r)}$ is the variety $S_h^{(L)}$ defined in Section \ref{s:gen_Drin_S} in the special case that $G$ is an inner form of $\GL_n$, the twisted Levi $L$ is $L^{(r)}$, and $U$ is the unipotent radical of the Borel subgroup specified in Section \ref{s:explicit}. By Lemma \ref{l:U_{h,r}}, we can change the quotient in the definition of $X_h^{(r)}$ from $\bU_{h,r} \cap F^{-1} \bU_{h,r}$ to $\bU_h \cap F^{-1} \bU_h$ so that
\begin{equation*}
X_{h}^{(r)} = \{x \in \bG_h : x^{-1} F(x) \in \bU_{h,r}\}/(\bU_h \cap F^{-1} \bU_h) \subset X_h.
\end{equation*}
Hence we have the picture:
\begin{equation*}
\begin{tikzcd}
S_h^{(r)} \ar[hook]{r} \ar[two heads]{d} & S_h \ar[two heads]{d} \\
X_h^{(r)} \ar[hook]{r} & X_h
\end{tikzcd}
\end{equation*}

\subsection{The Drinfeld stratification for the Drinfeld upper half-space}

Consider the twisted Frobenius $b_\cox \sigma \from \breve k^{\oplus n} \to \breve k^{\oplus n}$. Then $G(k)$ is equal to the subgroup consisting of all elements of $\GL_n(\breve k)$ which commute with $b_\cox \sigma$. Now consider the subquotient of $\breve k^{\oplus n}$ given by
\begin{equation*}
\sL_h \colonequals \big(\bW_h(\overline \FF_q) \oplus (\Ver\bW_{h-1}(\overline \FF_q))^{\oplus n_0 - 1}\big)^{\oplus n'} \subset \bW_h(\overline \FF_q)^{\oplus n}
\end{equation*}
and write $\sL = \varprojlim_h \sL_h$. The action of $G(k)$ on $\breve k^{\oplus n}$ restricts to an action of $G_{x,0}(\cO_k)$ on $\sL$ which induces an action of $\bG_h(\FF_q)$ on $\sL_h$. 

Now consider the $n'$-dimensional $\overline \FF_q$-vector space $V \colonequals \sL_1 \subset \overline \FF_q^{\oplus n}$. The morphism $\varpi^{-k_0}(b_\cox \sigma)^{n_0}$ is a Frobenius automorphism of $V$ and defines a $\FF_{q^{n_0}}$-rational structure on $V$. Observe that $\bG_1(\FF_q)$ is isomorphic to the subgroup of $\GL(V)$ consisting of elements which commute with $\varpi^{-k_0}(b_\cox \sigma)^{n_0}$. For any divisor $r \mid n'$ and any $\FF_{q^{n_0r}}$-rational subspace $W$ of $V$, consider
\begin{equation*}
\Omega_{W, q^{n_0r}} \colonequals \{[x] \in \bP(V) : \text{$W$ is the smallest $\FF_{q^{n_0r}}$-rational subspace of $V$ containing $x$}\}.
\end{equation*}
Note that $\Omega_{W,q^{n_0r}} \subset \bP(V)$ is isomorphic to the Drinfeld upper half-space for $W$ with respect to $\FF_{q^{n_0r}}$. For any divisor $r \mid n'$, define
\begin{equation*}
\sS_r \colonequals \bigcup_{W} \Omega_{W, q^{n_0r}},
\end{equation*}
where the union ranges over all $\FF_{q^{n_0r}}$-rational subspaces $W$ of dimension $n'/r$ in $V$. The following lemma records some easy facts.

\begin{lemma}
We have
\begin{enumerate}[label=(\roman*)]
\item
$\sS_1 = \Omega_{V, q^{n_0}}$ and $\sS_{n'} = \bP(V)(\FF_{q^n})$.
\item
If $r \mid r' \mid n'$ and $W$ is a $\FF_{q^{n_0r}}$-rational subspace of $V$, then $\Omega_{W, q^{n_0r'}} \subseteq \Omega_{W, q^{n_0r}}$.
\item
If $r \mid r' \mid n'$, then $\sS_1 \cap \sS_{r'} \subseteq \sS_1 \cap \sS_r$.
\end{enumerate}
\end{lemma}

Note that $\sS_1$ is the classical Deligne--Lusztig variety for $\bG_1(\FF_q) \cong \GL_{n'}(\FF_{q^{n_0}})$ with respect to the nonsplit maximal torus $\bT_1(\FF_q) \cong \FF_{q^n}^\times$ \cite[Section 2.2]{DeligneL_76} and the variety $X_h$ when $h=1$ is a $\FF_{q^n}^\times$-cover of $\sS_1$. Hence for any $h \geq 1$, we have a map
\begin{equation*}
X_h \to X_1 \to \sS_1.
\end{equation*}

\begin{lemma}\label{l:drinfeld preimage}
For any divisor $r \mid n'$, the variety $X_h^{(r)}$ is the preimage of $\sS_1 \cap \sS_r$ under the composition map $X_h \to X_1 \to \sS_1$.
\end{lemma}

\begin{proof}
To prove this, we use the explicit description of $X_h$ coming from Lemma \ref{l:description}:
\begin{equation*}
X_h = \{\lambda(v) \in \bG_h : \text{$v \in \sL_h$ and $\sigma(\det \lambda(v)) = \det \lambda(v)$}\}.
\end{equation*}
By Definition \ref{d:drinfeld GLn}, if $v \in \sL_h$ is such that $\lambda(v) \in X_h^{(r)}$, then $\lambda(v)^{-1} F(\lambda(v)) \in \bU_{h,r} \cap F \bU_{h,r}^-$, which is equivalent to
\begin{equation*}\label{e:A lin comb}
F(\lambda(v)) = \lambda(v) A, \qquad \text{for some $A \in \bU_{h,r} \cap F \bU_{h,r}^-$}.
\end{equation*}
Note that $A = (a_{i,j})_{1 \leq i,j \leq n}$ has the property that
\begin{align*}
a_{i,i} &= 1, && \text{for $i = 1, \ldots, n$,} \\
a_{i,1} &\in \bW_h, && \text{if $i \equiv 1$ mod $r n_0$}, \\
a_{i,1} &\in V \bW_{h-1} \subset \bW_h, && \text{if $i \not\equiv 1$ mod $r n_0$}, \\
a_{i,j} &= 0 && \text{otherwise.}
\end{align*}
The first column of $F(\lambda(v))$ is the vector $\sigma^n(v)$. Therefore \eqref{e:A lin comb} implies that
\begin{equation*}
\sigma^n(v) = \sum_{i=1}^n a_{i,1} \lambda(v)_i = v + \sum_{i=2}^n a_{i,1} \lambda(v)_i,
\end{equation*}
where $\lambda(v)_i$ denotes the $i$th column of $\lambda(v)$. Recall from \eqref{e:lambda i} that $\lambda(v)_{[ie_{\kappa,n}+1]_n} = \prod_{j=0}^{i-1} t_{[je_{\kappa,n}+1]}^{-1} \cdot (b\sigma)^i(v)$. If $[i e_{\kappa,n}+1]_n \equiv 1$ modulo $rn_0$, then $i \equiv 0$ modulo $rn_0$. Therefore, if $\mathfrak v$ denotes the image of $v$ in $\sL_1$, we have (using \eqref{e:lambda i}),
\begin{equation*}
\sigma^n(\mathfrak v) \in \Span\{\mathfrak v, \varpi^{-rk_0}(b\sigma)^{rn_0}(\mathfrak v), \varpi^{-2rk_0}(b\sigma)^{2rn_0}(\mathfrak v), \ldots, \varpi^{-(n'-1)rk_0}(b\sigma)^{(n'-1)rn_0}(\mathfrak v)\}.
\end{equation*}
Since $\lambda(v) \in \bG_h$, necessarily $\mathfrak v, \varpi^{-rk_0}(b\sigma)^{rn_0}(\mathfrak v), \ldots, \varpi^{-(n'-1)rk_0} (b\sigma)^{(n'-1)rn_0}(\mathfrak v)$ are linearly independent and therefore span a $n'/r$-dimensional subspace of $\sL_1$. This exactly means that $\mathfrak v \in \sS_1 \cap \sS_r$, so the proof is complete.
\end{proof}

\begin{remark}
By Lemma \ref{l:drinfeld preimage}, we see that for $\GL_n$ and its inner forms, the Drinfeld stratification of $X_h$ is induced by considering intermediate Drinfeld upper half-spaces of smaller dimension embedding in $\bP^{n'}_{\FF_{q^{n_0}}}$. %\hfill $\Diamond$
\end{remark}

\subsection{The Drinfeld stratification of $X_h(b,w)$}\label{s:drinfeld b,w}

In this section, we consider the varieties $X_h(b,w)$ in the special case
\begin{equation*}
\text{$b = g_0 b_{\cox} \sigma(g_0)^{-1}$ for some $g_0 \in G_{x,0}(\cO_{\breve k})$}, \qquad \text{and} \qquad w = b_{\cox}.
\end{equation*}
For any such $b$, recall from Lemmas \ref{l:change b} and \ref{l:Xbw to Xh} that
\begin{equation}\label{e:b bcox}
X_h = X_h(b_{\cox}, b_{\cox}) \cong X_h(b, b_{\cox}), 
\end{equation}
where the second isomorphism is given by $x \mapsto \overline g_0 x$, where $\overline g_0$ is the image of $g_0$ in $\bG_h(\overline \FF_q)$. Therefore the Drinfeld stratification $\{X_h^{(r)}\}$ of $X_h$ gives rise to a stratification $\{X_h(b,b_\cox)^{(r)}\}$ for $X_h(b,b_{\cox})$. The proof of Lemma \ref{l:g0 independence} shows that if $\sigma^n(\overline g_0) = \overline g_0$, then the Drinfeld stratification of $X_h(b, b_\cox)$ does not depend on the choice of $g_0$.

\begin{definition}
Let $b = g_0 b_{\cox} \sigma(g_0)^{-1} \in G(\breve k)$ for some $g_0 \in G_{x,0}(\cO_{\breve k})$. To each $v \in \sL_h$, define
\begin{align*}
g_b(v) &\colonequals \left(v_1 \, \Big| \, v_2 \, \Big| \, v_3 \, \Big| \, \cdots \, \Big| \, v_n\right) \\
\text{where } v_i &\colonequals \varpi^{\lfloor (i-1)k_0/n_0 \rfloor} \cdot (b\sigma)^{i-1}(v) \text{ for $1 \leq i \leq n-1$,}
\end{align*}
where we abuse notation by writing $\varpi^{\lfloor (i-1)k_0/n_0 \rfloor} \cdot (b\sigma)^{i-1}$ for the map $\sL_h \to \sL_h$ which takes $v$ to the image $\varpi^{\lfloor (i-1)k_0/n_0 \rfloor} \cdot (b\sigma)^{i-1}(\widetilde v)$ in the subquotient $\sL_h$ of $\breve k^{\oplus n}$, where $\widetilde v$ is any lift of $v$ in $\sL \subset \breve k^{\oplus n}$.
\end{definition}

\begin{lemma}
If $b = g_0 b_{\cox} \sigma(g_0)^{-1}$ for some $g_0 \in G_{x,0}(\cO_{\breve k})$, then 
\begin{equation*}
X_h(b,b_{\cox}) \cong \{v \in \sL_h :  \text{$\sigma(\det g_b(v)) =  \frac{\det b_\cox}{\det b} \cdot \det g_b(v) \in \bW_h^\times$}\}.
\end{equation*}
\end{lemma}

\begin{proof}
First note that one can obtain $g_{b_\cox}(v)$ from $\lambda(v)$ by permuting columns. In particular, 
\begin{equation*}
X_h(b_\cox,b_\cox) = X_h \cong \{v \in \sL_h : \sigma(\det g_{b_{\cox}}(v)) = \det g_{b_{\cox}}(v) \in \bW_h^\times\}.
\end{equation*}
Since $X_h(b_\cox, b_\cox) \cong X_h(b, b_\cox)$ is given by $x \mapsto \overline g_0 x$ where $\overline g_0$ denotes the image of $g_0$ in $\bG_h(\overline \FF_q)$, we have that $X_h(b,b_\cox)$ is isomorphic to the set of $\overline g_0 \cdot g_{b_\cox}(v)$ where $v \in \sL_h$ satisfies the above criterion. By direct computation,
\begin{equation*}
\overline g_0 \cdot g_{b_\cox}(v) = g_b(\overline g_0 \cdot v),
\end{equation*}
and hence if $\sigma(\det g_{b_\cox}(v)) = \det g_{b_\cox}(v)$, then
\begin{align*}
\sigma(\det g_b(\overline g_0 \cdot v)) 
&= \sigma(\overline \det g_0) \cdot \sigma(\det g_{b_\cox}(v)) = \sigma(\det \overline g_0) \cdot \det g_{b_\cox}(v) \\
&= \frac{\sigma(\det \overline g_0)}{\det \overline g_0} \cdot \det g_b(\overline g_0 \cdot v) =  \frac{\det b_\cox}{\det b} \cdot \det g_b(\overline g_0 \cdot v). \qedhere
\end{align*}
\end{proof}

\begin{lemma}\label{l:g0 independence}
Let $b = g_0 b_\cox \sigma(g_0)^{-1}$ for some $g_0 \in G_{x,0}(\cO_{\breve k})$ and assume that the image $\overline g_0 \in \bG_h(\overline \FF_q)$ of $g_0$ has the property that $\sigma^n(\overline g_0) = \overline g_0$. Let $r \mid n'$ be any divisor. For $v \in \sL_h$, let $\mathfrak v$ denote its image in $\sL_1$. Then
\begin{equation*}
X_h(b, b_\cox)^{(r)} \cong 
\left\{v \in \sL_h : 
\begin{gathered}
\sigma(\det g_{b_\xp}(v)) = \frac{\det b_\cox}{\det b} \cdot \det g_{b}(v) \in \bW_h^\times \\
\sigma^n(\mathfrak v) \in \Span\{\varpi^{-ik_0r} (b\sigma)^{irn_0}(\mathfrak v): 0 \leq i \leq n'-1\}\end{gathered}
\right\}.
\end{equation*}
In particular, the Drinfeld stratification of $X_h(b, b_\cox)$ does not depend on the choice of $g_0$.
\end{lemma}

%\begin{lemma}\label{l:drinfeld special}
%Let $r \mid n'$ be any divisor. For $v \in \sL_h$, let $\mathfrak v$ denote its image in $\sL_1$. Then
%\begin{equation*}
%X_h(b_\xp, b_\cox)^{(r)} \cong 
%\left\{v \in \sL_h : 
%\begin{gathered}
%\sigma(\det g_{b_\xp}(v)) = \frac{\det b_\cox}{\det b_\xp} \cdot \det g_{b_\xp}(v) \in \bW_h^\times \\
%\sigma^n(\mathfrak v) \in \Span\{\varpi^{-ik_0r} (b_\xp\sigma)^{irn_0}(\mathfrak v): 0 \leq i \leq n'-1\}\end{gathered}
%\right\}.
%\end{equation*}
%\end{lemma}
%
\begin{proof}
Recall that 
\begin{equation*}
X_h(b_\cox, b_\cox)^{(r)} \cong
\left\{v \in \sL_h : 
\begin{gathered}
\sigma(\det g_{b_\cox}(v)) = \det g_{b_\cox}(v) \in \bW_h^\times \\
\sigma^n(\mathfrak v) \in \Span\{\varpi^{-ik_0r} (b_\cox\sigma)^{irn_0}(\mathfrak v): 0 \leq i \leq n'-1\}\end{gathered}
\right\}.
\end{equation*}
By definition, every element in $X_h(b, b_\cox)^{(r)}$ is of the form $\overline g_0 g_{b_\cox}(v)$ for some $v \in \sL_h$ satisfying the above criteria. Since $\overline g_0 g_{b_\cox}(v) = g_{b}(\overline g_0 v)$ and since $\sigma^n(\overline g_0) = \overline g_0$, we have
\begin{equation*}
\overline g_0 \sigma^n(\mathfrak v) \in \Span\{\overline g_0 \varpi^{-ik_0r}(b_\cox \sigma)^{irn_0}(\mathfrak v): 0 \leq i \leq n'-1\}.
\end{equation*}
But now $\overline g_0 \varpi^{-ik_0r}(b_\cox \sigma)^{irn_0}(\mathfrak v) = \varpi^{-ik_0r}(b \sigma)^{irn_0}(\mathfrak v)$ and therefore the desired conclusion follows.
\end{proof}

\begin{remark}
In Appendix \ref{s:fibers}, we will work directly with a particular $b$ called the \textit{special representative} in \cite{CI_ADLV} (see Definition \ref{d:special} of the present paper). The special representative satisfies the hypotheses of Lemma \ref{l:g0 independence}. %\hfill $\Diamond$
\end{remark}

\section{Torus eigenspaces in the cohomology}

We prove an irreducibility result for torus eigenspaces in the alternating sum of the cohomology of $X_h \cap \bL_h^{(r)}\bG_h^1$.

\subsection{Howe factorizations}\label{s:howe_fact}

Let $\sT_{n,h}$ denote the set of characters $\theta \from \bW_h^\times(\F) \to \overline \QQ_\ell^\times$. Recall that if $h \geq 2$, we have natural surjections $\pr \from \bW_h^\times \to \bW_{h-1}^\times$ and injections $\bG_a \to \bW_h^\times$ given by $x \mapsto [1, 0, \ldots, 0, x]$. Furthermore, for any subfield $F \subset L$, the norm map $L^\times \to F^\times$ induces a map $\Nm \from \bW_h^\times(k_L) \to \bW_h^\times(k_F)$. These maps induce 
\begin{align*}
\pr^* \from \sT_{n,h'} &\to \sT_{n,h}, && \text{for $h' < h$}, \\
\Nm^* \from \sT_{m,h} &\to \sT_{n,h}, && \text{for $m \mid n$}.
\end{align*}
First consider the setting $h \geq 2$. By pulling back along $\bG_a \to \bW_h^\times, x \mapsto [1, 0, \ldots, 0, x]$, we may restrict characters of $\bW_h^\times(\F)$ to characters of $\F$. We say that $\theta \in \sT_{n,h}$ is \textit{primitive} if $\theta|_{\F}$ has trivial stabilizer in $\Gal(\F/\FF_q)$. If $h = 1$, then $\theta \in \sT_{n,h}$ is a character $\theta \from \FF_{q^n}^\times \to \overline \QQ_\ell^\times$, and we say it is \textit{primitive} if $\theta$ has trivial stabilizer in $\Gal(\F/\FF_q)$. For any $h \geq 1$, we write $\sT_{n,h}^0 \subset \sT_{n,h}$ to denote the subset of primitive characters.

We can decompose $\theta \in \sT_{n,h}$ into primitive components in the sense of Howe \cite[Corollary after Lemma 11]{Howe_77}.

\begin{definition}
A \textit{Howe factorization} of a character $\theta \in \sT_{n,h}$ is a decomposition
\begin{equation*}
\theta = \prod_{i=1}^d \theta_i, \qquad \text{where $\theta_i = \pr^* \Nm^* \theta_i^0$ and $\theta_i^0 \in \sT_{m_i, h_i}^0$},
\end{equation*}
such that $m_i < m_{i+1}$, $m_i \mid m_{i+1}$, and $h_i > h_{i+1}$. It is automatic that $m_i \leq n$ and $h \geq h_i$. For any integer $0 \leq t \leq d$, set $\theta_0$ to be the trivial character and define
\begin{equation*}
\theta_{\geq t} \colonequals \prod_{i=t}^d \theta_i \in \sT_{n,h_t}.
\end{equation*}
\end{definition}

Observe that the choice of $\theta_i$ in a Howe factorization $\theta = \prod_{i=1}^r \theta_i$ is not unique, but the $m_i$ and $h_i$ only depend on $\theta$. Hence the Howe factorization attaches to each character $\theta \in \sT_{n,h}$ a pair of well-defined sequences
\begin{align*}
&1 \equalscolon m_0 \leq m_1 < m_2 < \cdots < m_d \leq m_{d+1} \colonequals n \\
&h \equalscolon h_0 \geq h_1 > h_2 > \cdots > h_d \geq h_{d+1} \colonequals 1
\end{align*}
satisfying the divisibility $m_i \mid m_{i+1}$ for $0 \leq i \leq d$.

\begin{example}
We give some examples of the sequences associated to characters $\theta \in \sT_{n,h}$.
\begin{enumerate}[label=(\alph*)]
\item
If $\theta$ is the trivial character, then $d = 1$ and the associated sequences are 
\begin{equation*}
\{m_0, m_1, m_2\} = \{1, 1, n\}, \qquad \{h_0, h_1, h_2\} = \{h, 1, 1\},
\end{equation*}
where we note that $\sT_{1,1} = \sT_{1,1}^0$ since any character of $\FF_q^\times$ has trivial $\Gal(\FF_q/\FF_q)$-stabilizer.

\item
Say $h \geq h'$. We say that $\theta$ is a primitive character of level $h' \geq 2$ if $\theta|_{U_L^{h'}} = 1$ and $\theta|_{U_L^{h'-1}/U_L^{h'}}$ has trivial $\Gal(\FF_{q^n}/\FF_q)$-stabilizer. Then $d = 1$ and the associated sequences are 
\begin{equation*}
\{m_0, m_1, m_2\} = \{1, n, n\}, \qquad \{h_0, h_1, h_2\} = \{h, h', 1\}. 
\end{equation*}
In the division algebra setting, this case is studied in \cite{Chan_DLI,Chan_DLII}. For arbitrary inner forms of $\GL_n$ over $K$, we considered \textit{minimal admissible} $\theta$, which are exactly the characters $\theta \in \sT_{n,h}$ which are either primitive or have $d=2$ with associated sequences
\begin{equation*}
\{m_0, m_1, m_2, m_3\} = \{1, 1, n, n\}, \qquad \{h_0, h_1, h_2, h_3\} = \{h, h_1, h_2, 1\}.
\end{equation*}
This is a very slight generalization over the primitive case.

\item
Say $h \geq 2$. If $\theta|_{U_L^2} = 1$ and the stabilizer of $\theta|_{U_L^1/U_L^2}$ in $\Gal(\F/\FF_q)$ is $\Gal(\F/\FF_{q^m})$, then $d = 1$ and the associated sequences are 
\begin{equation*}
\{m_0, m_1, m_2\} = \{1, m, n\}, \qquad \{h_0, h_1, h_2\} = \{h, 2, 1\}. 
\end{equation*}
In the division algebra setting, the case $h = 2$ is studied in \cite{Boyarchenko_12,BoyarchenkoW_16}.

\item
Say $h \geq 1$. If $\theta|_{U_L^1} = 1$ and the stabilizer of $\theta \from \FF_{q^n}^\times \to \overline \QQ_\ell^\times$ is $\Gal(\F/\FF_{q^m})$, then $d = 1$ and the associated sequences are
\begin{equation*}
\{m_0, m_1, m_2\} = \{1, m, n\}, \qquad \{h_0, h_1, h_2\} = \{h, 1, 1\}.
\end{equation*}
This is the so-called ``depth zero'' case.
\end{enumerate}
\end{example}

\subsection{Irreducibility}\label{s:irred}

Recall that the intersection $X_h \cap \bL_h^{(r)}\bG_h^1$ has an action by the subgroup $\bL_h^{(r)}(\FF_q)\bG_h^1(\FF_q) \times \bT_h(\FF_q) \subset \bG_h(\FF_q) \times \bT_h(\FF_q)$. In this section, we study the irreducibility of the virtual $\bL_h^{(r)}(\FF_q)\bG_h^1(\FF_q)$-representation $H_c^*(X_h \cap \bL_h^{(r)}\bG_h^1)[\theta]$, where $\theta \from \bT_h(\FF_q) \to \overline \QQ_\ell^\times$ is arbitrary. 

We follow a technique of Lusztig which has appeared in the literature in many incarnations, the closest analogues being  \cite{Lusztig_04,Stasinski_09,CI_MPDL}. In these works, the strategy is to translate the problem of calculating an inner product between two representations to calculating the cohomology of a third variety $\Sigma$. This is done by first writing $\Sigma = \Sigma' \sqcup \Sigma''$, proving the cohomology of $\Sigma''$ gives the expected outcome, and then putting a lot of work into showing that the cohomology of $\Sigma'$ does not contribute. In the three works cited, one can only prove the vanishing of (certain eigenspaces of) the Euler characteristic of $\Sigma'$ under a strong \textit{regularity} condition on the characters $\theta, \theta'$. The key new idea here is adapted from \cite[Section 3.2]{CI_loopGLn}, which allows us to relax this regularity assumption by working directly with $\Sigma$ throughout the proof. We give only a sketch of the proof of Theorem \ref{t:inner_prod} here, as the proof of \cite[Theorem 3.1]{CI_loopGLn} is very similar.

\begin{theorem}\label{t:inner_prod}
Let $\theta, \theta' \from \bT_h(\FF_q) \to \overline \QQ_\ell^\times$ be any two characters. Then
\begin{equation*}
\Big\langle H_c^*(X_h \cap \bL_h^{(r)}\bG_h^1)[\theta], H_c^*(X_h \cap \bL_h^{(r)}\bG_h^1)[\theta'] \Big\rangle_{\bL_h^{(r)}(\FF_q) \bG_h^1(\FF_q)} = \#\{w \in W_{\bL_h^{(r)}}^F : \theta' = \theta \circ \Ad(w)\},
\end{equation*}
where $W_{\bL_h^{(r)}}^F = N_{\bL_h^{(r)}(\FF_q)}(\bT_h(\FF_q))/\bT_h(\FF_q)$.
\end{theorem}

Since $W_{\bL_h^{(r)}}^F \cong \Gal(\FF_{q^n}/\FF_{q^{n_0r}})$, we obtain the following theorem as a direct corollary of Theorem \ref{t:inner_prod}.

\begin{corollary}\label{t:irred}
Let $\theta \from \bT_h(\FF_q) \cong \bW_h^\times(\FF_{q^n}) \to \overline \QQ_\ell^\times$ be any character. Then the virtual $\bL_h^{(r)}(\FF_q)\bG_h^1(\FF_q)$-representation $H_c^*(X_h \cap \bL_h^{(r)}\bG_h^1)[\theta]$ is (up to sign) irreducible if and only if $\theta$ has trivial $\Gal(\FF_{q^n}/\FF_{q^{n_0r}})$-stabilizer.
\end{corollary}

In the special case that $r = n'$, we have $\bL_h^{(n')} = \bT_h$ and using Lemma \ref{l:description} and Definition \ref{d:drinfeld GLn}, we have that $S_h \cap \bT_h \bG_h^1$ is an affine fibration over 
\begin{equation*}
\{x \in \bT_h \bG_h^1 : x^{-1} F(x) \in \bU_h^1 \cap F\bU_h^{-,1}\}.
\end{equation*}
and that
\begin{equation*}
X_h \cap \bT_h \bG_h^1 = \bigsqcup_{t \in \bT_h(\FF_q)} t \cdot X_h^1, \qquad \text{where $X_h^1 = X_h \cap \bG_h^1$.}
\end{equation*}
Here we have
\begin{equation}\label{e:Xh1}
X_h^1 = \{x \in \bG_h^1 : x^{-1} F(x) \in \bU_h^1 \cap F\bU_h^{-,1}\}.
\end{equation}

\begin{corollary}\label{t:irred n'}
Let $\chi \from \bT_h^1(\FF_q) \to \overline \QQ_\ell^\times$ be any character. Then $H_c^*(X_h^1, \overline \QQ_\ell)[\chi]$ is an irreducible representation of $\bG_h^1(\FF_q)$. Moreover, if $\chi,\chi'$ are any two characters of $\bT_h^1(\FF_q)$, then $H_c^*(X_h^1, \overline \QQ_\ell)[\chi] \cong H_c^*(X_h^1, \overline \QQ_\ell)[\chi']$ if and only if $\chi = \chi'$.
\end{corollary}

Corollary \ref{t:irred n'} follows from Corollary \ref{t:irred} (by arguing the relationship between the cohomology of $X_h^1$ and the cohomology of $X_h \cap \bT_h \bG_h^1$), but one can give an alternate proof using \cite[Section 6.1]{Chan_siDL}, which is based on \cite{Lusztig_79}. We do this in Section \ref{s:irred n'}.

\begin{remark}
Recall that specializing Lemma \ref{l:geom parabolic induction} yields that
\begin{equation*}
H_c^*(X_h^{(r)}, \overline \QQ_\ell)[\theta] \cong \Ind_{\bL_h^{(r)}(\FF_q)\bG_h^1(\FF_q)}^{\bG_h(\FF_q)}\big(H_c^*(X_h \cap \bL_h^{(r)} \bG_h^1, \overline \QQ_\ell)[\theta]\big).
\end{equation*}
We note that one needs a separate argument to study the irreducibility of $H_c^*(X_h^{(r)}, \overline \QQ_\ell)[\theta]$. In the case that $r = n'$, this is done in \cite[Theorem 4.1(b)]{CI_loopGLn}. %\hfill $\Diamond$
\end{remark}

\subsubsection{Proof of Theorem \ref{t:inner_prod}}

Recall that by definition
\begin{equation*}
S_h \cap \bL_h^{(r)} \bG_h^1 = \{g \in \bL_h^{(r)} \bG_h^1 : g^{-1} F(g) \in \bU_{h,r}\}, \qquad \text{where $\bU_{h,r} = \bL_h^{(r)} \bU_h^1 \cap \bU_h$.}
\end{equation*}
Consider the variety
\begin{equation*}
\Sigma^{(r)} = \{(x,x', y) \in F(\bU_{h,r}) \times F(\bU_{h,r}) \times \bL_h^{(r)} \bG_h^1 : x F(y) = yx'\}
\end{equation*}
endowed with the $\bT_h(\FF_q) \times \bT_h(\FF_q)$-action given by $(t,t') \from (x,x',y) \mapsto (txt^{-1}, t' x' t'{}^{-1}, tyt'{}^{-1})$. Then we have an isomorphism
\begin{align*}
\bL_h^{(r)}(\FF_q)\bG_h^1(\FF_q) \backslash \big((S_h \cap \bL_h^{(r)} \bG_h^1) \times (S_h \cap \bL_h^{(r)} \bG_h^1)\big) &\to \Sigma^{(r)}, \\ 
(g,g') &\mapsto (g^{-1}F(g), g'{}^{-1} F(g'), g^{-1}g'),
\end{align*}
equivariant with respect to $\bT_h(\FF_q) \times \bT_h(\FF_q)$. To prove Theorem \ref{t:inner_prod}, we need to establish
\begin{equation}\label{e:Sigma goal}
\sum_i (-1)^i \dim H_c^i(\Sigma^{(r)}, \overline \QQ_\ell)_{\theta,\theta'} = \#\{w \in W_{\bL_h^{(r)}}^F : \theta' = \theta \circ \Ad(w)\}.
\end{equation}

The Bruhat decomposition of the reductive quotient $\bG_1$ lifts to a decomposition $\bG_h = \bigsqcup_{w \in W_{\bG_h}} \bG_{h,w}$, where $\bG_{h,w} = \bU_h \bT_h \dot w \bK_{h}^1 \bU_h$ and $\bK_{h}^1 = (\bU_h^-)^1 \cap \dot w^{-1} \bU_h^{-,1} \dot w$ \cite[Lemma 8.6]{CI_MPDL}. This induces the decomposition
\begin{equation*}
\bL_h^{(r)} \bG_h^1 = \bigsqcup_{w \in W_{\bL_h^{(r)}}^F} \bG_{h,w}^{(r)}, \qquad \text{where $\bG_{h,w}^{(r)} = \bG_{h,w} \cap \bL_h^{(r)} \bG_h^1$.}
\end{equation*}
and also the locally closed decomposition 
\begin{equation*}
\Sigma^{(r)} = \bigsqcup_{w \in W_\cO} \Sigma_w^{(r)}, \qquad \text{where $\Sigma_w^{(r)} = \Sigma \cap (F(\bU_{h,r}) \times F(\bU_{h,r}) \times \bG_{h,w}^{(r)})$.}
\end{equation*}
We will calculate \eqref{e:Sigma goal} by analyzing the cohomology of
\begin{align*}
\widehat \Sigma_w^{(r)} = \{(x,x',y_1,\tau,z,y_2) \in F(\bU_{h,r}) \times F(\bU_{h,r}) \times \bU_{h,r} \times {}&{} \bT_h \times \bK_{h}^1 \times \bU_h : \\&x F(y_1 \tau \dot w z y_2) = y_1 \tau \dot w z y_2 x'\}.
\end{align*}
Since $\widehat \Sigma_w^{(r)} \to \Sigma_w^{(r)}, (x,x',y_1,\tau,z,y_2) \mapsto (x,x',y_1 \tau z y_2)$ is a locally trivial fibration, showing \eqref{e:Sigma goal} is equivalent to showing
\begin{equation}\label{e:Sigma hat goal}
\sum_i (-1)^i \dim H_c^i(\widehat \Sigma_w^{(r)}, \overline \QQ_\ell)_{\theta,\theta'} = 
\begin{cases}
1 & \text{if $w \in W_{\bL_h^{(r)}}^F$ and $\theta' = \theta \circ \Ad(w)$,} \\
0 & \text{otherwise.}
\end{cases}
\end{equation}
As in \cite[1.9]{Lusztig_04}, we can simplify the formulation of $\widehat \Sigma_w$ by replacing $x$ by $x F(y_1)$ and replacing $x'$ by $x' F(y_2)^{-1}$. We then obtain
\begin{equation*}
\widehat \Sigma_w^{(r)} = \{(x,y_1, \tau,z,y_2) \in F\bU_{h,r} \times \bU_{h,r} \times \bT_h \times \bK_{h}^1 \times \bU_{h,r} : x F(\tau \dot w z) \in y_1 \tau \dot w z y_2 F \bU_{h,r}\}.
\end{equation*}

\begin{lemma}\label{l:empty criterion}
Assume that there exists some $2 \leq i \leq n$ which satisfies the string of inequalities $[\gamma \dot w \gamma^{-1}(i)] > [\gamma \dot w \gamma^{-1}(i-1)+1] > 1$. Then $\widehat \Sigma_w = \varnothing.$
\end{lemma}

\begin{proof}
By the same argument as in \cite[Lemma 3.4]{CI_loopGLn}, we may assume $h=1$ and come to the statement that $\widehat \Sigma_w = \varnothing$ if there does not exist $(x,y_{12}, y_{21}, \tau) \in F\bU_{1,r} \times (\bU_{1,r} \cap F\bU_{1,r}^-) \times (\bU_{1,r} \cap F\bU_{1,r}^-) \times \bT_1$ such that
\begin{equation*}
\dot w^{-1} \tau y_{12} x F(\dot w) \in y_{21} F(\bU_1 \cap \bL_1^{(r)}).
\end{equation*}
Therefore to prove the lemma, it is enough to analyze the intersection
\begin{equation*}
\big[\dot w^{-1} (\bU_{1,r} \cap F\bU_{1,r}^-) \cdot F\bU_{1,r} F(\dot w)\big] \cap \big[(\bU_{1,r} \cap F\bU_{1,r}^-) \cdot F(\bU_1 \cap \bL_1^{(r)})\big].
\end{equation*}
By construction (see \eqref{e:gamma}, \eqref{e:U}, and write $F_0(g) = b_0 \gamma t_{\kappa,n} \gamma^{-1} \sigma(g) t_{\kappa,n}^{-1} \gamma b_0 \gamma^{-1}$), we have
\begin{align*}
\dot w^{-1} &(\bT_1 \cap (\bU_{1,r} \cap F\bU_{1,r}^-) \cdot F\bU_{1,r}) F(\dot w) \cap ((\bU_{1,r} \cap F\bU_{1,r}^-) \cdot F\bU_{1,r}) \\
&= \gamma^{-1} (\gamma \dot w^{-1} \gamma^{-1})(\bT_1 \cdot (\bU_{{\rm low},1,r} \cap F_0\bU_{{\rm up},1,r}) \cdot F_0\bU_{{\rm low},1,r}) F_0(\gamma \dot w^{-1} \gamma^{-1}) \gamma \\
&\qquad\qquad\qquad\cap \gamma^{-1}((\bU_{{\rm low},1,r} \cap F_0 \bU_{{\rm up},1,r}) \cdot F_0 \bU_{{\rm low},1,r})\gamma.
\end{align*}
Now the desired result holds by \cite[Lemma 3.5]{CI_loopGLn}.
\end{proof}

The rest of the proof now proceeds exactly as in \cite[Section 3.3, 3.4]{CI_loopGLn}, which we summarize now. By \cite[Lemma 3.5]{CI_loopGLn}, if $1 \neq w \in W_{\bL_h^{(r)}}$ is such that $\widehat \Sigma_w \neq \varnothing$, then $\bU_h \cap \dot w^{-1} U_h \dot w$ is centralized by a subtorus of $\bT_h$ which properly contains the center of $\bG_h$. In particular, the group
\begin{equation*}
H_w = \{(t,t') \in \bT_h \times \bT_h : \text{$\dot w^{-1} t^{-1} F(t) \dot w = t'{}^{-1} F(t')$ centralizes $\bK_{h} = \bU_ \cap \dot w^{-1} \bU_h \dot w$}\}
\end{equation*}
has the property that its image under the projections $\pi_1, \pi_2 \from \bT_h \times \bT_h \to \bT_1 \times \bT_1 \to \bT_1$ contains a rank-$1$ regular\footnote{We mean here that this torus is not contained in $\ker(\alpha)$ for any root $\alpha$ of $\bT_1$ the reductive group $\bG_1$. See \cite[Lemma 3.7]{CI_loopGLn}.} torus. Crucially, $H_w$ acts on $\widehat\Sigma_w^{(r)}$ via
\begin{equation*}
(t,t') \from (x,y_1, \tau, z, y_2) \mapsto (F(t) x F(t)^{-1}, F(t) y_1 F(t)^{-1}, t \tau \dot w t'{}^{-1} \dot w^{-1}, t' z t'{}^{-1}, F(t') y_2 F(t')^{-1}),
\end{equation*}
and this action extends the action of $\bT_h(\FF_q) \times \bT_h(\FF_q)$. Then $H_c^*(\widehat \Sigma_w, \overline \QQ_\ell) = H_c^*(\widehat \Sigma_w^{H_{w,{\rm red}}^0}, \overline \QQ_\ell)$ and using \cite[Lemma 3.6]{CI_loopGLn}, we can calculate:
\begin{equation*}
\widehat \Sigma_w^{H_{w,{\rm red}}^0} = \begin{cases}
(\bT_h \dot w)^F & \text{if $F(\dot w) = \dot w$,} \\
\varnothing & \text{otherwise.}
\end{cases}
\end{equation*}
Now \eqref{e:Sigma hat goal} holds for all $w \neq 1$. To obtain \eqref{e:Sigma hat goal} for $w = 1$, we may apply \cite[Section 3.4]{CI_loopGLn} directly. We have now finished the proof of Theorem \ref{t:inner_prod}

\subsubsection{Proof of Corollary \ref{t:irred n'}}\label{s:irred n'}

Consider 
\begin{equation*}
\Sigma^1 = \{(x,x',y) \in (\bU_h^1 \cap F\bU_h^{-,1}) \times (\bU_h^1 \cap F\bU_h^{-,1}) \times \bG_h^1 : x F(y) = yx'\}.
\end{equation*}
Then we have an isomorphism
\begin{equation*}
\bG_h(\FF_q) \backslash \big((X_h \cap \bT_h \bG_h^1) \times (X_h \cap \bT_h \bG_h^1)\big) \to \Sigma^1, \qquad (g,g') \mapsto (g^{-1}F(g), g'{}^{-1}F(g'), g^{-1} g').
\end{equation*}
Since $\bG_h^1$ has an Iwahori factorization, any $y \in \bG_h^1$ can be written uniquely in the form
\begin{align*}
y &= y_1' y_2' y_1'' y_2'', & 
y_1' &\in  \bU_h^1 \cap F^{-1}(\bU_h^1), & y_2' &\in \bU_h^1 \cap F^{-1}(\bU_h^{-,1}), \\
& & y_1'' &\in \bT_h \cdot (\bU_h^{-,1} \cap F^{-1} \bU_h^{-,1}), & y_2'' &\in \bU_h^{-,1} \cap F^{-1} \bU_h^1.
\end{align*}
Then our definition equation becomes
\begin{equation*}
x F(y_1' y_2' y_1'' y_2'') = y_1' y_2' y_1'' y_2'' x'.
\end{equation*}
By \eqref{e:section}, every element of $\bU_h$ can be written uniquely in the form $y_1'{}{-1} x F(y_1')$. We also have $F(y_2''), x' \in \bU_h^1 \cap F \bU_h^{-,1}$ and we can replace $x'$ by $x' F(y_2'')^{-1}$. Therefore $\Sigma^1$ is the set of tuples $(x', y_2', y_1'', y_2'') \in (\bU_h^1 \cap F\bU_h^{-,1}) \times (\bU_h^1 \cap F^{-1}\bU_h^{-,1}) \times (\bT_h \cdot (\bU_h^{-,1} \cap F^{-1} \bU_h^{-,1})) \times (\bU_h^{-,1} \cap F^{-1}\bU_h^1)$ which satisfy
\begin{equation*}
y_1'' y_2'' x' \in y_2'{}^{-1} \bU_h F(y_2') F(y_1'') = \bU_h F(y_2') F(y_1'').
\end{equation*}

Now consider the subgroup
\begin{equation*}
H \colonequals \{(t,t') \in \bT_h \times \bT_h : \text{$t^{-1} F(t) = t'{}^{-1}F(t')$ centralizes $\bT_h \cdot (\bU_h^{-,1} \cap F^{-1} \bU_h^{-,1})$}\}.
\end{equation*}
It is a straightforward check that for any $(t,t') \in H$, the map
\begin{equation*}
(x', y_2', y_1'', y_2'') \mapsto (F(t')^{-1} x' F(t'), t^{-1} y_2' t, t^{-1} y_1'' t', F(t')^{-1} y_2'' F(t'))
\end{equation*}
defines an action of $H$ on $\Sigma^1$. By explicit calculation, one can check that $H$ contains an algebraic torus $\cT$ over $\overline \FF_q$ and that the fixed points of $\Sigma^1$ under $\cT$ is equal to $\bT_h^1(\FF_q)$. We therefore have
\begin{equation*}
\dim H_c^*(\Sigma^1, \overline \QQ_\ell)_{\theta^{-1}, \theta'} = \begin{cases} 1 & \text{if $\chi = \chi'$,} \\
0 & \text{otherwise,} \end{cases}
\end{equation*}
and this completes the proof.

\subsection{Very regular elements}

Recall that we say that an element $g \in \bT_h(\FF_q) \cong \bW_h(\FF_{q^n})^\times$ is \textit{very regular} if its image in $\FF_{q^n}^\times$ has trivial $\Gal(\F/\FF_q)$-stabilizer.

\begin{proposition}\label{t:very_reg}
Let $\theta \from \bT_h(\FF_q) \to \overline \QQ_\ell^\times$ be any character. If $g \in \bT_h(\FF_q)\subset \bL_h^{(r)}(\FF_q) \bG_h^1(\FF_q)$ is a very regular element, then
\begin{equation*}
\Tr(g ; H_c^*(X_h \cap \bL_h^{(r)}\bG_h^1)[\theta]) = \sum_{\gamma \in \Gal(L/k)[n'/r]} \theta^\gamma(x),
\end{equation*}
where $\Gal(L/k)[n'/r]$ is the unique order-$n'/r$ subgroup of $\Gal(L/k)$.
\end{proposition}

\begin{proof}
Let $g \in \bT_h(\FF_q)$ be a very regular element and let $t \in \bT_h(\FF_q)$ be any element. Since the action of $(g,t)$ on $X_h \cap \bL_h^{(r)}\bG_h^1$ is a finite-order automorphism of a separated, finite-type scheme over $\FF_{q^n}$, by the Deligne--Lusztig fixed point formula,
\begin{equation*}
\Tr\left((g,t)^*; H_c^*(X_h \cap \bL_h^{(r)} \bG_h^1)[\theta]\right) = \Tr\left((g_u, t_u)^*; H_c^*((X_h \cap \bL_h^{(r)} \bG_h^1)^{(g_s, t_s)})[\theta]\right),
\end{equation*}
where $g = g_s g_u$ and $t = t_s t_u$ are decompositions such that $g_s, t_s$ is a power of $g,t$ of $p$-power order and $g_u, t_u$ is a power of $g,t$ of prime-to-$p$ order.

Recall from Section \ref{s:explicit} that every element $x$ of $X_h \cap \bL_h^{(r)}\bG_h^1$ is a matrix that is uniquely determined by its first column $(x_1, x_2, \ldots, x_n)$. Furthermore, we have an isomorphism
\begin{equation*}
\bW_h(\FF_{q^n})^\times \to \bT_h(\FF_q), \qquad t \mapsto \diag(t, \sigma^l(t), \sigma^{2l}(t), \ldots, \sigma^{(n-1)l}(t)).
\end{equation*}
Under this identification, for $g,t \in \bT_h(\FF_q)$, the element $gxt \in X_h \cap \bL_h^{(r)}\bG_h^1$ corresponds to the vector $(gt x_1, \sigma^l(g)t x_2, \sigma^{2l}(g)t x_3, \ldots, \sigma^{(n-1)l}(g) t x_n).$ In particular, we see that if $x \in (X_h \cap \bL_h^{(r)}\bG_h^1)^{(g,t)}$, then (for any $i = 1, \ldots, n$) $x_i \neq 0$ implies $t = \sigma^{(i-1)l}(g)^{-1}$. Using the assumption that $g$ is very regular and therefore $g_s$ has trivial $\Gal(L/k)$-stabilizer, this implies that $(X_h \cap \bL_h^{(r)} \bG_h^1)^{(g,t)}$ exactly consists of elements corresponding to vectors with a single nonzero entry $x_i$. Now, if $i \not\equiv 1$ modulo $n_0$, then the corresponding $x$ cannot lie in $X_h$ as then $\det(x) \notin \bW_h(\overline \FF_q)^\times$. On the other hand, if $i \equiv 1$ modulo $n_0$ and $i \not\equiv 1$ modulo $n_0r$, then the corresponding $x$ cannot lie in $\bL_h^{(r)}\bG_h^1$.
%Therefore
%\begin{equation*}
%x \in (X_h \cap \bL_h^{(r)}\bG_h^1)^{(g_s,t_s)} \; \Longleftrightarrow \; 
%\begin{cases}
%\text{$x$ corr.\ to $(0,\ldots, 0, x_i, 0, \ldots, 0)$ for some $i \equiv 1$ (mod $n_0$)}, \\
%\text{and $t = \sigma^{(i-1)l}(g)^{-1}$.}
%\end{cases}
%\end{equation*}
If $x \in X_h \cap \bL_h^{(r)} \bG_h^1$ corresponds to $(0, \ldots, 0, x_i, 0, \ldots, 0)$ for some $i \equiv 1$ modulo $n_0r$, then $x_i$ can be any element of $\bW_h^\times(\FF_{q^n})$. Hence: 
\begin{equation*}
(X_h \cap \bL_h^{(r)} \bG_h^1)^{(g_s,t_s)} = 
\begin{cases}
b_0^i\bT_h(\FF_q) & \text{if $t = \sigma^{(i-1)l}(g)^{-1}$ for some $i \equiv 1$ mod $n_0r$,} \\
\varnothing & \text{otherwise.}
\end{cases}
\end{equation*}
Furthermore, for $g_u, t_u \in \bT_h(\FF_q)$ and $b_0^i x \in (X_h \cap \bL_h^{(r)}\bG_h^1)^{(g_s,t_s)},$
\begin{equation*}
g_u \cdot b_0^i x \cdot t_u = b_0^i (b_0^{-i} g_u b_0^i) x t_u = b_0^i (\sigma^{(i-1)l}(g_u) x t_u).
\end{equation*}

We are now ready to put all the above together. We have
\begin{align*}
\Tr(g ; {}&{}H_c^*(X_h \cap \bL_h^{(r)}\bG_h^1)[\theta])\\
&= \frac{1}{\#\bT_h(\FF_q)} \sum_{t \in \bT_h(\FF_q)} \theta(t)^{-1} \Tr((g,t) ; H_c^*(X_h \cap \bL_h^{(r)} \bG_h^1)) \\
&= \frac{1}{\#\bT_h(\FF_q)} \sum_{t \in \bT_h(\FF_q)} \theta(t)^{-1} \Tr((g_u,t_u) ; H_c^*((X_h \cap \bL_h^{(r)} \bG_h^1)^{(g_s, t_s)})) \\
&= \frac{1}{\#\bT_h(\FF_q)} \sum_{\substack{1 \leq i \leq n \\ i \equiv 1 \!\!\!\!\! \pmod{n_0r}}} \theta(\sigma^{(i-1)l}(g_s)) \sum_{t_u \in \bT_h^1(\FF_q)} \theta(t_u)^{-1} \Tr((g_u, t_u) ; H_c^*(b_0^i\bT_h(\FF_q))) \\
&= \frac{1}{\#\bT_h(\FF_q)} \sum_{\substack{1 \leq i \leq n \\ i \equiv 1 \!\!\!\!\! \pmod{n_0r}}} \theta(\sigma^{(i-1)l}(g_s)) \sum_{t_u \in \bT_h^1(\FF_q)} \theta(t_u)^{-1} \sum_{\theta' \from \bT_h(\FF_q) \to \overline \QQ_\ell} \theta'(\sigma^{(i-1)l}(g_u)) \theta'(t_u) \\
&= \sum_{\substack{1 \leq i \leq n \\ i \equiv 1 \!\!\!\!\! \pmod{n_0r}}} \theta(\sigma^{(i-1)l}(g_s)) \theta(\sigma^{(i-1)l}(g_u)) = \sum_{\gamma \in \Gal(L/k)[n'/r]} \theta^\gamma(g). \qedhere
\end{align*}
\end{proof}

\section{The closed stratum is a maximal variety}\label{s:single_degree}

Recall that $X_h^{(r)}$ is the closure of the $r$th Drinfeld stratum and that the unique closed Drinfeld stratum is the $n'$th Drinfeld stratum 
\begin{equation*}
X_h^{(n')} \colonequals \{x \in \bG_h : x^{-1} \sigma(x) \in \bU_r^1\}.
\end{equation*}
Recall that $X_h^{(n')}$ is a finite disjoint union of copies of $X_h^1 \colonequals X_h^{(n')} \cap \bG_h^1$:
\begin{equation*}
X_h^{(n')} = \bigsqcup_{g \in \bG_1(\FF_q)} [g] \cdot X_h^1,
\end{equation*}
where $[g]$ denotes a coset representative in $\bG_h(\FF_q)$ for $g \in \bG_1(\FF_q) = \bG_h(\FF_q)/\bG_h^1(\FF_q)$. For any character $\theta \from \bT_h(\FF_q) \to \overline \QQ_\ell^\times$, we have an isomorphism of $\bG_h(\FF_q)$-representations
\begin{equation*}
H_c^i(X_h^{(n')}, \overline \QQ_\ell)[\theta] \cong \Ind_{\bT_h(\FF_q) \bG_h^1(\FF_q)}^{\bG_h(\FF_q)}\left(H_c^i(X_h^{(n')} \cap \bT_h \bG_h^1, \overline \QQ_\ell)[\theta]\right), \qquad \text{for all $i \geq 0$}.
\end{equation*}

Let $\chi \colonequals \theta|_{\bT_h^1(\FF_q)}$. As $\bG_h^1(\FF_q)$-representations,
\begin{equation*}
H_c^i\left(X_h^{(n')} \cap \bT_h \bG_h^1, \overline \QQ_\ell\right)[\theta] \cong H_c^i(\Z, \overline \QQ_\ell)[\chi], \qquad \text{for all $i \geq 0$.}
\end{equation*}
The subvariety $\Z \subset X_h$ is stable under the action of $\Gamma_h \colonequals \{(\alpha,\alpha^{-1}) : \alpha \in \bT_h(\FF_q)\} \cdot (\bG_h^1(\FF_q) \times \bT_h^1(\FF_q))$, where the product is viewed as a product of subgroups of $\bG_h(\FF_q) \times \bT_h(\FF_q)$. Observe that $\Gamma_h \cong \FF_{q^n}^\times \ltimes (\bG_h^1(\FF_q) \times \bT_h^1(\FF_q))$ and note that $\Gamma_h \cdot (\{1\} \times \bT_h(\FF_q)) = \bG_h(\FF_q) \times \bT_h(\FF_q)$. Therefore
\begin{equation*}
\Ind_{\Gamma_h}^{\bG_h(\FF_q) \times \bT_h(\FF_q)}(H_c^i(\Z, \overline \QQ_\ell)[\chi]) \cong \bigoplus_{\theta'} H_c^i(X_h \cap \bT_h \bG_h^1)[\theta'],
\end{equation*}
where $\theta'$ ranges over all characters of $\bT_h(\FF_q)$ which restrict to $\chi$ on $\bT_h^1(\FF_q)$. The action of $(\zeta, g, t) \in \FF_{q^n}^\times \ltimes 
(\bG_h^1(\FF_q) \times \bT_h^1(\FF_q)) \cong \Gamma_h$ on $x \in \Z$ is given by
\begin{equation*}
(\zeta, g, t) * x = \zeta (g x t) \zeta^{-1},
\end{equation*}
where we view $\zeta \in \FF_{q^n}^\times$ as an element of $\bW_h(\FF_{q^n})^\times \cong \bT_h(\FF_q)$.

\subsection{The nonvanishing cohomological degree}

Recall from Section \ref{s:howe_fact} that any character $\theta \from \bT_h(\FF_q) \to \overline \QQ_\ell^\times$ has a Howe factorization. For any Howe factorization $\theta = \prod_{i=1}^d \theta_i$ of $\theta$, define a Howe factorization for $\chi \colonequals \theta|_{\bT_h^1(\FF_q)}$ by
\begin{equation*}
\chi = \prod_{i=1}^{d'} \chi_i, \qquad \text{where $\chi_i \colonequals \theta_i|_{\bT_h^1(\FF_q)}$ and $d' \colonequals \begin{cases}
d & \text{if $h_d \geq 2$,} \\
d-1 & \text{if $h_d = 1$.}
\end{cases}$}
\end{equation*}
As in Section \ref{s:howe_fact}, although the characters $\chi_i$ are not uniquely determined, we have two well-defined sequences of integers
\begin{align*}
1 &\equalscolon m_0 \leq m_1 < m_2 < \cdots < m_{d'} \leq m_{d'+1} \leq m_{d+1} \colonequals n \\
h &\equalscolon h_0 \geq h_1 > h_2 > \cdots > h_{d'} > h_{d'+1} = h_{d+1} \colonequals 1
\end{align*}
satisfying the divisibility $m_i \mid m_{i+1}$ for $0 \leq i \leq d$.

We state the main result of this section.

\begin{theorem}\label{t:single_degree}
Let $\chi \from \bT_h^1(\FF_q) \cong \bW_h^1(\FF_{q^n}) \to \overline \QQ_\ell^\times$ be any character. Then
\begin{equation*}
H_c^i(\Z, \overline \QQ_\ell)[\chi] = \begin{cases}
\text{irreducible $\bG_h^1(\FF_q)$-representation} & \text{if $i = r_\chi$,} \\
0 & \text{if $i \neq r_\chi$,}
\end{cases}
\end{equation*}
where
\begin{align*}
r_\chi &= 2(n'-1) + 2 e_\chi + f_\chi \\
e_\chi &= \Big(\frac{n}{m_{d'}} - 1\Big)(h_{d'}-1) - \Big(\frac{n}{\lcm(m_{d'},n_0)} - 1\Big) - (h_0 - h_{d'}) + \sum_{t=0}^{d'-1}\frac{n}{m_t}(h_t - h_{t+1}) \\
f_\chi &= \Big(n-\frac{n}{m_{d'}}\Big) - \Big(n' - \frac{n}{\lcm(m_{d'},n_0)}\Big) + \sum_{t=0}^{d'-1} \Big(\frac{n}{m_t} - \frac{n}{m_{t+1}}\Big)h_{t+1}
\end{align*}
Moreover, $\Fr_{q^n}$ acts on $H_c^{r_\chi}(\Z, \overline \QQ_\ell)$ as multiplication by $(-1)^{i} q^{ni/2}$. 
\end{theorem}

The assertion about the action of $\Fr_{q^n}$ on $H_c^i(\Z, \overline \QQ_\ell)[\theta]$ is equivalent to saying that $\Z$ is a \textit{maximal variety} in the sense of Boyarchenko--Weinstein \cite{BoyarchenkoW_16}; that is, $\#\Z(\FF_{q^n})$ attains its Weil--Deligne bound
\begin{equation*}
\#\Z(\FF_{q^n}) = \sum_{i \geq 0} (-1)^i \Tr(\Fr_{q^n}; H_c^i(\Z, \overline \QQ_\ell)) \leq \sum_{i \geq 0} q^{ni/2}  \dim H_c^i(\Z, \overline \QQ_\ell).
\end{equation*}

For easy reference later, we record the following special case of Theorem \ref{t:single_degree}.

\begin{corollary}\label{c:prounip_degree}
Let $\chi \from \bT_h^1(\FF_q) \cong \bW_h^1(\FF_{q^n}) \to \overline \QQ_\ell^\times$ be any character with trivial $\Gal(L/k)$-stabilizer. Then
\begin{equation*}
H_c^i(\Z, \overline \QQ_\ell)[\chi] =
\begin{cases}
\text{irreducible} & \text{if $i = r_\chi$,} \\
0 & \text{if $i \neq r_\chi$,}
\end{cases}
\end{equation*}
where
\begin{equation*}
r_\chi = n(h-h_1) + h(n-2) + h_{d'} - (n-n') + \sum_{t=1}^{d'-1} \frac{n}{m_t}(h_t - h_{t+1}).
\end{equation*}
\end{corollary}

\begin{proof}
The assumption that $\chi$ has trivial $\Gal(L/k)$-stabilizer is equivalent to the assumption that $m_{d'} = n$. We see then that the formula for $r_\chi$ given in Theorem \ref{t:single_degree} simplifies as follows:
\begin{align*}
r_\chi &= 2(n'-1) + \sum_{t=0}^{d'-1} 2\Big(\frac{n}{m_t} - 1\Big)(h_t - h_{t+1}) \\
&\qquad + \sum_{t=0}^{d'-1}\Big(\Big(\frac{n}{m_t} - \frac{n}{m_{t+1}}\Big)(h_{t+1} - 1) - \Big(\frac{n}{\lcm(m_t, n_0)} - \frac{n}{\lcm(m_{t+1}, n_0)}\Big)\Big) \\
&= 2(n'-1) - 2(h_0 - h_{d'}) - \Big(\frac{n}{m_0} - \frac{n}{m_{d'}}\Big) - \Big(\frac{n}{\lcm(m_0,n_0)} - \frac{n}{\lcm(m_{d'}, n_0)}\Big)\\
&\qquad + \frac{n}{m_0}(2 h_0 - h_1) - \frac{n}{m_{d'}}(h_{d'}) + \sum_{t=1}^{d'-1} \frac{n}{m_t}(h_t - h_{t+1}).
\end{align*}
Using the fact that $h_0 = h$ and $m_0 = 1$ by construction, the above expression simplifies to the one given in the statement of the corollary.
\end{proof}

\subsection{Ramified Witt vectors}

We give a brief summary of ramified Witt vectors, following \cite[Section 3.1]{Chan_siDL}. In this section, we assume $k$ has characteristic $0$. We first define a ``simplified version'' of the ramified Witt ring $\bW$.

\begin{definition}
For any $\FF_q$-algebra $A$, let $W(A)$ be the set $A^\bN$ endowed with the following coordinatewise addition and multiplication rule:
\begin{align*}
[a_i]_{i \geq 0} +_W [b_i]_{i \geq 0} &= [a_i + b_i]_{i \geq 0}, \\
[a_i]_{i \geq 0} *_W [b_i]_{i \geq 0} &= \left[\textstyle \sum\limits_{j = 0}^i a_j^{q^{i-j}} b_{i-j}^{q^i}\right]_{i \geq 0}.
\end{align*}
It is a straightforward check that $W$ is a commutative ring scheme over $\FF_q$. It comes with Frobenius and Verschiebung morphisms $\varphi$ and $V$. 
\end{definition}

The relationship between the ring scheme $W$ and the ring scheme $\bW$ of ramified Witt vectors is captured by the following lemma. The key point here is the notion of ``major contribution'' and ``minor contribution''; this will appear in Lemma \ref{l:det_contr} and (implicitly) in Proposition \ref{p:induct extra}.

\begin{lemma}\label{l:simple Witt n}
Let $A$ be an $\FF_q$-algebra. 
\begin{enumerate}[label=(\alph*)]
\item
For any $[a_1], \ldots, [a_n] \in A^\bN$ where $[a_j] = [a_{j,i}]_{i \geq 0}$,
\begin{equation*}
\prod_{\substack{1 \leq j \leq n \\ \text{w.r.t.\ $\bW$}}} [a_j] = \left(\prod_{\substack{1 \leq j \leq n \\ \text{w.r.t.\ $W$}}} [a_j]\right) +_W [c],
\end{equation*}
where $[c] = [c_i]_{i \geq 0}$ for some $c_i \in A[a_{1,i_1}^{e_1} \cdots a_{n,i_n}^{e_n} : i_1 + \cdots + i_n < i, \, e_1, \ldots, e_n \in \bZ_{\geq 0}].$

\item
For any $[a_1], \ldots, [a_n] \in A^\bN$ where $[a_j] = [a_{j,i}]_{i \geq 0}$,
\begin{equation*}
\sum_{\substack{1 \leq j \leq n \\ \text{w.r.t.\ $\bW$}}} [a_j] = \left(\sum_{\substack{1 \leq j \leq n \\ \text{w.r.t.\ $W$}}} [a_j]\right) +_W [c],
\end{equation*}
where $[c] = [c_i]_{i \geq 0}$ for some $c_i \in A[a_{1,j}, \ldots, a_{n,j} : j < i].$
\end{enumerate}
We call the portion coming from $W$ the ``major contribution'' and $[c]$ the ``minor contribution.''
\end{lemma}

\subsection{Normed indexing sets}\label{s:indexing}

The group $\bG_h^1$ is an affine space of dimension $n^2(h-1)$. To prove Theorem \ref{t:single_degree}, we will need to coordinatize $\bG_h^1$, and we do this here by defining an indexing set $\cA^+$ of triples $(i,j,l)$. Our strategy for approaching Theorem \ref{t:single_degree} is to perform an inductive calculation based on a Howe factorization of the character $\chi \from \bT_h^1(\FF_q) \to \overline \QQ_\ell^\times$. In this section, we will also define a filtration of $\cA^+$ corresponding to the two sequences $\{m_i\}, \{h_i\}$ associated with $\chi$.

The algebraic group $\bG_h^1$ can be described very explicitly: it consists of matrices $(A_{i,j})_{1 \leq i,j \leq n}$ where 
\begin{equation*}
A_{i,j} = \begin{cases}
[1,A_{(i,j,1)}, A_{(i,j,2)}, \ldots, A_{(i,j,h-1)}] \in \bW_h^1 & \text{if $i = j$,} \\
[A_{(i,j,0)}, A_{(i,j,1)}, \ldots, A_{(i,j,h-2)}] \in \bW_{h-1} & \text{if $[i]_{n_0} > [j]_{n_0}$}, \\
[0,A_{(i,j,1)}, A_{(i,j,2)}, \ldots, A_{(i,j,h-1)}] \in \bW_h & \text{if $[i]_{n_0} \leq [j]_{n_0}$ and $i \neq j$.}
\end{cases}
\end{equation*}
Here, we recall that for $x \in \bZ$, we write $[x]_{n_0}$ to denote the unique representative of $x \bZ/n_0 \bZ$ in the set of coset representatives $\{1, \ldots, n_0\}$. We have a well-defined determinant map
\begin{equation*}
\det \from \bG_h^1 \to \bW_h^1.
\end{equation*}
In the way described above, $\bG_h^1$ can be coordinatized by the indexing set
\begin{equation*}
\cA^+ \colonequals \left\{(i,j,l) \in \bZ^{\oplus 3} :
\begin{gathered}
1 \leq i, j \leq n \\
\text{$0 \leq l \leq h-2$ if $[i]_{n_0} > [j]_{n_0}$} \\
\text{$1 \leq l \leq h-1$ if $[i]_{n_0} \leq [j]_{n_0}$}
\end{gathered}\right\}.
\end{equation*}
We also define:
\begin{align*}
\cA &\colonequals \{(i,j,l) \in \cA^+ : i \neq j\}, \\
\cA^- &\colonequals \{(i,j,l) \in \cA : j = 1\}.
\end{align*}
The indexing set $\cA$ corresponds to the elements of $\bG_h^1$ with $1$'s along the diagonal, and $\cA^-$ remembers only the first column of elements of $\bG_h^1$ with $(1,1)$-entry $1$.

\begin{definition}
Define a norm on $\cA^+$:
\begin{align*}
\cA^+ &\to \bR_{\geq 0}, \\
(i,j,l) &\mapsto |(i,j,l)| \colonequals i - j + nl.
\end{align*}
\end{definition}

\begin{definition}
For $\lambda = (i,j,l) \in \cA^+$, define
\begin{equation*}
\lambda^\vee \colonequals (j,i,h-1-l).
\end{equation*}
\end{definition}

The following seemingly innocuous lemma is in some sense the key reason that the indexing sets above allow us to carry over the calculations in \cite[Section 5]{Chan_siDL} from $n'=1$ setting to the present general $n'$ setting with very few modifications.

\begin{lemma}\label{l:det_contr}
Following the conventions as set up above, write $A = (A_{i,j})_{1 \leq i,j \leq n} \in \bG_h^1$, where
\begin{equation*}
A_{i,j} = \begin{cases}
[1,A_{(i,j,1)}, \ldots, A_{(i,j,h-1)}] \in \bW_h^1 & \text{if $i = j$,} \\
[A_{(i,j,0)}, \ldots, A_{(i,j,h-2)}] \in \bW_{h-1} & \text{if $[i]_{n_0} > [j]_{n_0}$}, \\
[0, A_{(i,j,1)}, \ldots, A_{(i,j,h-1)}] \in \bW_h & \text{if $[i]_{n_0} \leq [j]_{n_0}$ and $i \neq j$}.
\end{cases}
\end{equation*}
Assume that for $\lambda_1, \lambda_2 \in \cA^+$, the variables $A_{\lambda_1}$ and $A_{\lambda_2}$ appear in the same monomial in $\det(A) \in \bW_{h'}$ for some $h' \leq h$. 
\begin{enumerate}[label=(\alph*)]
\item
Then $|\lambda_1| + |\lambda_2| \leq n(h'-1)$.
\item
If $|\lambda_1| + |\lambda_2| = n(h'-1)$, then $\lambda_2 = \lambda_1^\vee$, where ${}^\vee$ is taken relative to $h'$.
\end{enumerate}
\end{lemma}

\begin{proof}
By definition,
\begin{equation*}
\det(A) = \sum_{\gamma \in S_n} \prod_{1 \leq i \leq n} A_{i, \gamma(i)} \in \bW_{h'}(\overline \FF_q).
\end{equation*}
Let $l \leq h'-1$. If $K$ has characteristic $p$, then the contributions to the $\varpi^l$-coefficient coming from $\gamma \in S_n$ are of the form
\begin{equation*}
\prod_{i=1}^n A_{(i,\gamma(i), l_i)},
\end{equation*}
where $(l_1, \ldots, l_n)$ is a partition of $l$. Then
\begin{equation}\label{e:equal sum}
\sum_{i=1}^n |(i,\gamma(i), l_i)| = \sum_{i=1}^n i - \gamma(i) + nl_i = \sum_{i=1}^n nl_i = nl \leq n(h'-1).
\end{equation}
If $K$ has characteristic $0$, then the major contributions to the $\varpi^l$-coefficient coming from $\gamma$ are of the form
\begin{equation*}
\prod_{i=1}^n A_{(i,\gamma(i), l_i)}^{e_i}, 
\end{equation*}
where the $e_i$ are some nonnegative integers and where $(l_1, \ldots, l_n)$ is a partition of $l$. Hence
\begin{equation}\label{e:mixed sum}
\sum_{i=1}^n |(i,\gamma(i), l_i)| = nl \leq n(h'-1).
\end{equation}
The minor contributions to the $\varpi^l$-coefficient coming from $\gamma$ are polynomials in $\prod_{i=1}^n A_{(i,\gamma(i), l_i)}^{e_i'}$ where $l_1 + \cdots + l_n < l$ and the $e_i'$ are some nonnegative integers. Hence $\sum_{i=1}^n |(i,\gamma(i), l_i)| < n(h'-1)$.

Suppose now that $\lambda_1 = (i_1, j_1, l_1), \lambda_2 = (i_2, j_2, l_2) \in \cA^+$ are such that $A_{\lambda_1}$ and $A_{\lambda_2}$ contribute to the same monomial in $\det(M) \in \bW_{h'}^1$. Then there exists some $\gamma \in S_n$ such that $\gamma(i_1) = j_1$ and $\gamma(i_2) = j_2$, and by Equations \eqref{e:equal sum} and \eqref{e:mixed sum},
\begin{equation*}
|\lambda_1| + |\lambda_2| \leq n(h'-1).
\end{equation*}
Observe that if $K$ has characteristic $0$ and $\lambda_1$ and $\lambda_2$ occur in a minor contribution, then $|\lambda_1| + |\lambda_2| < n(h')$. This proves (a), and furthermore, we see that if $|\lambda_1| + |\lambda_2| = n(h'-1)$, then the simultaneous contribution of $A_{\lambda_1}$ and $A_{\lambda_2}$ comes from a major contribution. But now (b) follows: since the image of $\bG_h^1$ under the determinant is $\bW_h^1$, if $|\lambda_1| + |\lambda_2| = n(h'-1)$, then necessarily the contribution of $\lambda_1$ and $\lambda_2$ to the $(h'-1)$th coordinate of the determinant must come from a transposition.
\end{proof}

Given two sequences of integers
\begin{align*}
1 &\equalscolon m_0 \leq m_1 < m_2 < \cdots < m_{d'} \leq m_{d'+1} \leq m_{d+1} \colonequals n \\
h &\equalscolon h_0 \geq h_1 > h_2 > \cdots > h_{d'} > h_{d'+1} = h_{d+1} \colonequals 1
\end{align*}
satisfying $m_i \mid m_{i+1}$ for $0 \leq i \leq d$, we can define the following subsets of $\cA$ for $0 \leq s,t \leq d$:
\begin{align*}
\cA_{s,t} &\colonequals \{(i,j,l) \in \cA : i \equiv j \!\!\!\! \pmod{m_s}, \, i \not\equiv j \!\!\!\! \pmod{m_{s+1}}, \, l \leq h_t - 1\}, \\
\cA_{s,t}^- &\colonequals \cA_{s,t} \cap \cA^-.
\end{align*}
%We will also need the following decomposition of $\cA_{s,t}^-$:
%\begin{align*}
%\cI_{s,t} &\colonequals \{(1,j,l) \in \cA_{s,t}^- : |(1,j,l)| > n(h_t - 1)/2\}, \\
%\cJ_{s,t} &\colonequals \{(1,j,l) \in \cA_{s,t}^- : |(1,j,l)| \leq n(h_t - 1)/2\}.
%\end{align*}
%For any real number $\nu \geq 0$, define
%\begin{equation*}
%\cA_{\geq \nu, t} \colonequals \bigsqcup_{s = \lceil \nu \rceil}^r \cA_{s,t}, \quad \cI_{\geq \nu, t} \colonequals \bigsqcup_{s = \lceil \nu \rceil}^r \cI_{s,t}, \quad \cJ_{\geq \nu,t} \colonequals \bigsqcup_{s = \lceil \nu \rceil}^r \cJ_{s,t}, \quad \cA_{\geq \nu,t}^- \colonequals \cA_{\geq \nu, t} \cap \cA^-,
%\end{equation*}
%and observe that
%\begin{equation*}
%\cA_{\geq s,t} = \{(i,j,l) \in \cA : j \equiv i \!\!\!\! \pmod{m_s}, \, 1 \leq l \leq h_t - 1\}.
%\end{equation*}
%
We will need to understand which $\lambda \in \cA$ are such that $x_\lambda$ contributes nontrivially to the determinant. We denote the set of all such $\lambda$ by $\cA^{\min}$. We may describe this explicitly:
\begin{align}
\cA^{\min} &= \{\lambda \in \cA : \lambda^\vee \in \cA\} \label{e:index min} \\ \nonumber
&= \left\{(i,j,l) \in \cA : \begin{gathered} 
\text{$0 \leq l \leq h - 2$ if $[i]_{n_0} > [j]_{n_0}$} \\  
\text{$1 \leq l \leq h - 1$ if $[i]_{n_0} < [j]_{n_0}$} \\
\text{$1 \leq l \leq h - 2$ if $[i]_{n_0} = [j]_{n_0}$}
\end{gathered}\right\}.
\end{align}
For $0 \leq s,t \leq r$, by considering ${}^\vee$ relative to $h_t$, we may similarly define
\begin{align*}
\cA_{s,t}^{\min} 
&\colonequals \{\lambda \in \cA_{s,t} : \lambda^\vee \in \cA_{s,t}\} \\
&= \left\{
(i,j,l) \in \cA_{s,t} : \begin{gathered}
\text{$0 \leq l \leq h_t - 2$ if $[i]_{n_0} > [j]_{n_0}$} \\  
\text{$1 \leq l \leq h_t - 1$ if $[i]_{n_0} < [j]_{n_0}$} \\
\text{$1 \leq l \leq h_t - 2$ if $[i]_{n_0} = [j]_{n_0}$}
\end{gathered}
\right\}.
\end{align*}
Define $\cA_{s,t}^{-,\min} \colonequals \cA^- \cap \cA_{s,t}^{\min} = \cA_{s,t}^- \cap \cA_{s,t}^{\min}$. Define the following decomposition of $\cA_{s,t}^{-,\min}$:
\begin{align*}
\cI_{s,t} &\colonequals \{(i,1,l) \in \cA_{s,t}^{-,\min} : |(i,1,l)| > n(h_t - 1)/2\}, \\
\cJ_{s,t} &\colonequals \{(i,1,l) \in \cA_{s,t}^{-,\min} : |(i,1,l)| \leq n(h_t - 1)/2\}.
\end{align*}
For any real number $\nu$, define
\begin{align*}
\cA_{\geq \nu, t}^{\min} \colonequals \bigsqcup_{s = \lceil \nu \rceil}^r \cA_{s,t}^{\min}, \qquad 
\cA_{\geq \nu, t}^{-,\min} = \cA^- \cap \cA_{\geq \nu, t}^{\min},
\end{align*}
and observe that for $0 \leq s \leq r$ an integer,
\begin{equation*}
\cA_{\geq s,t}^{\min} = \left\{(i,j,l) \in \cA :
\begin{gathered}
\text{$j \equiv i \!\!\!\! \pmod{m_s}$} \\
\text{$0 \leq l \leq h_t - 2$ if $[i]_{n_0} > [j]_{n_0}$} \\  
\text{$1 \leq l \leq h_t - 1$ if $[i]_{n_0} < [j]_{n_0}$} \\
\text{$1 \leq l \leq h_t - 2$ if $[i]_{n_0} = [j]_{n_0}$}
\end{gathered}\right\}.
\end{equation*}
%Define $\cA^{\min} \colonequals \cA \cap \cA^{+, \min}$, $\cA^{-,\min} \colonequals \cA^- \cap \cA^{+, \min}$, $\cI_{s,t}^{\min} \colonequals \cI_{s,t} \cap \cA^{+, \min},$ and $\cJ_{s,t}^{\min} \colonequals \cJ_{s,t} \cap \cA^{+, \min}.$
%\begin{align*}
%\cA^{\min} &\colonequals \cA \cap \cA^{+, \min},  & \cA^{-,\min} &\colonequals \cA^- \cap \cA^{+, \min}, \\
%\cI_{s,t}^{\min} &\colonequals \cI_{s,t} \cap \cA^{-, \min}, &\cJ_{s,t}^{\min} &\colonequals \cJ_{s,t} \cap \cA^{-, \min}.
%\end{align*}

\begin{lm}\label{l:IJ bij}
There is an order-reversing injection $\cI_{s,t} \to \cJ_{s,t}$ that is a bijection if and only if $\cA_{s,t}^{-,\min}$ is even. Explicitly, it is given by
\begin{equation*}
\cI_{s,t} \hookrightarrow \cJ_{s,t}, \qquad (i,1, l) \mapsto ([n-i+2]_n, 1, h_t - 2 -l).
\end{equation*}
Note that $\#\cA_{s,t}^{-,\min}$ is even unless $n$ and $h_t$ are both even.
\end{lm}

\begin{proof}
If $(i,1,l) \in \cA_{s,t}^{-,\min}$, then by definition $i \equiv 1$ modulo $m_s$ and $i \not \equiv 1$ modulo $m_{s+1}$. Thus $[n-i+2]_n \equiv 1$ modulo $m_s$ and $[n-i+2]_n \not \equiv 1$ modulo $m_{s+1}$, which shows that $(i,1,l) \in \cA_{s,t}^{-,\min}$ implies $([n-i+2]_n,1,l) \in \cA_{s,t}^{-,\min}$. Since $i \geq 2$ by assumption, we have $i + [n-i+2]_n = n+2$ and
\begin{equation*}
|(i,1,l)| + |([n-i+2]_n,1,h_t-2-l)| = n(h_t-1).
\end{equation*}
Hence if $(i,1,l) \in \cI_{s,t}$, then $([n-i+2]_n,1,l) \in \cJ_{s,t}$. It is clear that the map is a bijection if and only if $\cJ_{s,t}$ does not contain an element of norm $n(h_t - 1)/2$. Such an element must necessarily be of the form $((n+2/2),1,(h_t-2)/2)$, which is integral if and only if $n$ and $h_t$ are both even.
\end{proof}

\subsection{The cohomology of $\Z$}

The purpose of this section is to establish the following result:

\begin{theorem}\label{t:hom}
For any character $\chi \from \bT_h^1(\FF_q) \to \overline \QQ_\ell^\times$,
\begin{equation*}
\Hom_{\bG_h^1(\FF_q)}\left(\Ind_{\bT_h^1(\FF_q)}^{\bG_h^1(\FF_q)}(\chi), H_c^i(\Z, \overline \QQ_\ell)\right) =
\begin{cases}
\overline \QQ_\ell^{\oplus q^{nf_\chi/2}} \otimes ((-q^{n/2})^{r_\chi})^{\deg} & \text{if $i = r_\chi$,} \\
0 & \text{otherwise.}
\end{cases}
\end{equation*}
Moreover, $\Fr_{q^n}$ acts on $H_c^i(\Z, \overline \QQ_\ell)$ by multiplication by the scalar $(-1)^i q^{ni/2}$.
\end{theorem}

This is a technical calculation which follows the strategy developed in \cite{Chan_siDL} (in particular, see Sections 4 and 5 of \textit{op.\ cit.}). We first rephrase space of homomorphisms in the statement of Theorem \ref{t:hom} in terms of the cohomology of a related variety. Every coset of $\bG_h^1/\bT_h^1$ has a unique coset representative $g$ whose diagonal entries are identically $1$. Over $\FF_q$, we may identify $\bG_h^1/\bT_h^1$ with the affine space $\bA[\cA]$ (the affine space of dimension $\#\cA$ with coordinates indexed by the set $\cA$ of Section \ref{s:indexing}). Then the quotient morphism $\bG_h^1 \to \bG_h^1/\bT_h^1$ has a section given by
\begin{equation*}
s \from \bG_h^1/\bT_h^1 \to \bG_h^1, \qquad (x_{(i,j,l)})_{(i,j,l) \in \cA} \mapsto (x_{i,j})_{i,j=1, \ldots, n},
\end{equation*}
where 
\begin{equation*}
x_{i,j} = 
\begin{cases}
1 \in \bW_h^1 & \text{if $i = j$,} \\
[x_{(i,j,0)}, x_{(i,j,1)}, \ldots, x_{(i,j,h-2)}] \in \bW_{h-1} & \text{if $[i]_{n_0} > [j]_{n_0}$}, \\
[0,x_{(i,j,1)}, x_{(i,j,2)}, \ldots, x_{(i,j,h-1)}] \in \bW_h & \text{if $[i]_{n_0} \leq [j]_{n_0}$ and $i \neq j$.}
\end{cases}
\end{equation*}
As in \cite[Section 5.1.1]{Chan_siDL}, there exists a closed $\FF_{q^n}$-subscheme $\Y$ of $\bG_h^1$ such that $X_h = L_q^{-1}(\Y)$ which satisfies the condition that $\Fr_q^i(\Y) \cap \Fr_q^j(\Y) = \{1\}$ for all $i \neq j$. We are therefore in a setting where we can invoke \cite[Proposition 4.1.1]{Chan_siDL}.

Define
\begin{equation*}
\beta \from (\bG_h^1/\bT_h^1) \times \bT_h^1 \to \bG_h^1, \qquad (x,g) \mapsto s(\Fr_q(x)) \cdot g \cdot s(x).
\end{equation*}
The affine $\FF_{q^n}$-scheme $\beta^{-1}(\Y) \subset (\bG_h^1/\bT_h^1) \times \bT_h^1$ comes with two maps:
\begin{equation*}
\pr_1 \from \beta^{-1}(\Y) \to \bG_h^1/\bT_h^1 = \bA[\cA], \qquad \pr_2 \from \beta^{-1}(\Y) \to \bT_h^1.
\end{equation*}
Recall from \cite[Lemma 4.1.2]{Chan_siDL} that since the Lang morphism $L_q$ is surjective,
\begin{equation}\label{e:beta}
(x,g) \in \beta^{-1}(\Y) \qquad \Longleftrightarrow \qquad s(x) \cdot y \in X_h,
\end{equation}
where $y \in \bT_h^1$ is any element such that $L_q(y) = g$. 

\begin{proposition}\label{p:coh_beta}
For any character $\chi \from \bT_h^1(\FF_q) \cong \bW_h^1(\F) \to \overline \QQ_\ell^\times$, let $\Loc_\chi$ denote the corresponding $\overline \QQ_\ell$-local system on $\bW_h^1$. For $i \geq 0$, we have $\Fr_{q^n}$-compatible isomorphisms
\begin{equation*}
\Hom_{\bG_h^1(\FF_q)}\Big(\Ind_{\bT_h^1(\FF_q)}^{\bG_h^1(\FF_q)}(\chi), H_c^i(X_h, \overline \QQ_\ell)\Big) \cong H_c^i(\bA[\cA^-], P^* \Loc_\chi),
\end{equation*}
where $P \from \bA[\cA^-] \to \bW_h^1$ is the morphism $(x_{(i,1,l)})_{(i,1,l) \in \cA^-} \mapsto L_q(\det(g_b^{\red}(1,x_2,\ldots, x_n)))^{-1}$ for $x_i \colonequals [x_{(i,1,0)}, x_{(i,1,1)}, \ldots, x_{(i,1,h-1)}]$.
\end{proposition}

\begin{proof}
By \cite[Proposition 4.1.1]{Chan_siDL},
\begin{equation*}
\Hom_{\bG_h^1(\FF_q)}\Big(\Ind_{\bT_h^1(\FF_q)}^{\bG_h^1(\FF_q)}(\chi), H_c^i(X_h, \overline \QQ_\ell)\Big) \cong H_c^i(\beta^{-1}(\Y), \pr_2^* \cF_\chi),
\end{equation*}
where $\cF_\chi$ is the rank-$1$ local system on $\bT_h^1$ corresponding to $\chi$. By the same proof as \cite[Lemma 5.1.1]{Chan_siDL}, $\beta^{-1}(\Y)$ is the graph of the morphism $P_0 \from \bA[\cA] \to \bW_h^1$ given by $x \mapsto L_q(\det(s(x)))^{-1}$. Furthermore, as morphisms on $\beta^{-1}(\Y)$, we have $\pr_2 = i \circ P_0 \circ \pr_1$, where $i \from \bW_h^1 \to \bT_h^1, x \mapsto \diag(x,1, \ldots, 1)$. Therefore, as sheaves on $\pr_1(\beta^{-1}(\Y))$, we have $\pr_2^* \cF_\chi = P_0^* i^* \cF_\chi = P_0^* \Loc_\chi$, so 
\begin{equation*}
H_c^i(\beta^{-1}(\Y), \pr_2^* \cF_\chi) = H_c^i(\pr_1(\beta^{-1}(\Y)), P_0^* \Loc_\chi).
\end{equation*}

Next we claim that the projection $\bA[\cA] \to \bA[\cA^-]$ induces an isomorphism $\pr_1(\beta^{-1}(\Y)) \to \bA[\cA^-]$. Injectivity is clear: using \eqref{e:beta}, we know that $x \in \pr_1(\beta^{-1}(\Y))$ if $s(x) \cdot y \in \Z$ for some $y \in \bT_h^1$. Since $s(x) \cdot y$ is uniquely determined by its first column, then $s(x)$ is uniquely determined by its first column, which is precisely the projection of $x$ to $\bA[\cA^-]$. To see surjectivity, we need to show that for any $x \in \bA[\cA^-](\overline \FF_q)$, there exists a $y \in \bT_h^1(\overline \FF_q)$ such that $g_b^{\red}(x) \cdot y \in \Z$. Pick any $y = \diag(y_1, \sigma(y_1), \ldots, \sigma(y_1)) \in \bT_h^1(\overline \FF_q)$ such that $\det(y) = \det(g_b^{\red}(x))^{-1}$. Then $g_b^{\red}(x) \cdot y \in X_h$ since $g_b^{\red}(x) \cdot y = g_b^{\red}(xy_1)$ and $\det(g_b^{\red}(x) \cdot y) = 1 \in \bW_h^1(\FF_q)$. Under the isomorphism $\pr_1(\beta^{-1}(\Y)) \cong \bA[\cA^-]$, the sheaf $P_0^* \Loc_\chi$ is identified with $P^* \Loc_\chi$, and the proposition now follows.
\end{proof}

Note that the last paragraph of the above proof is a simpler and more conceptual proof of \cite[Lemma 5.1.6]{Chan_siDL}. To calculate $H_c^i(\bA[\cA^-], P^* \Loc_\chi)$, we will use an inductive argument on affine fibrations that relies on iteratively applying the next two propositions:

\begin{proposition}\label{p:induct factor}
For $0 \leq t \leq d'$, we have $\Fr_{q^n}$-compatible isomorphisms
\begin{equation*}
H_c^i(\bA[\cA_{\geq t,t}^{-,\min}], P^* \Loc_{\chi_{\geq t}}) \cong H_c^i(\bA[\cA_{\geq t,t+1}^{-,\min}], P^* \Loc_{\chi_{\geq t+1}})[2 e_t] \otimes ((-q^{n/2})^{2e_t})^{\deg},
\end{equation*}
where $e_t = \#(\cA_{\geq t,t}^{-,\min} \smallsetminus \cA_{\geq t,t+1}^{-,\min})$.
\end{proposition}

\begin{proof}
The proof is the same as the proof of \cite[Proposition 5.3.1]{Chan_siDL}. We give a sketch here. By definition, $\chi_{\geq t} = \chi_t \cdot \chi_{\geq t+1}$ and $\chi_t$ factors through the norm map $\bW_{h_t}^1(\FF_{q^n}) \to \bW_{h_t}^1(\FF_{q^{m_t}})$. Let $\pr \from \bW_{h_t}^1 \to \bW_{h_{t+1}}^1$. Since $P \from \bA[\cA_{\geq t,t}^{-,\min}] \to \bW_{h_t}^1$ factors through $L_{q^{m_t}}$, this implies that
\begin{equation*}
P^* \Loc_{\chi_{\geq t}} = P^* \Loc_{\chi_t} \otimes P^* \pr^* \Loc_{\chi_{\geq t+1}} = \overline \QQ_\ell \boxtimes P^* \Loc_{\chi_{\geq t+1}},
\end{equation*}
where $\overline \QQ_\ell$ is the constant sheaf on $\bA[\cA_{\geq t,t}^{-,\min} \smallsetminus \cA_{\geq t,t+1}^{-,\min}]$ and $P^* \Loc_{\chi_{\geq t+1}}$ is the pullback along $P \from \bA[\cA_{\geq t,t+1}^{-,\min}] \to \bW_{h_{t+1}}^1$. The conclusion then follows from the K\"unneth formula.
\end{proof}

\begin{proposition}\label{p:induct extra}
For $0 \leq t \leq d'-1$, we have $\Fr_{q^n}$-compatible isomorphisms
\begin{equation*}
H_c^i(\bA[\cA_{\geq t,t+1}^{-,\min}], P^* \Loc_{\chi_{\geq t+1}}) \cong H_c^i(\bA[\cA_{\geq t+1,t+1}^{-,\min}], P^* \Loc_{\chi_{\geq t+1}})^{\oplus q^{nf_t/2}}[f_t] \otimes ((-q^{n/2})^{f_t})^{\deg},
\end{equation*}
where $f_t = \#(\cA_{\geq t,t+1}^{-,\min} \smallsetminus \cA_{\geq t+1,t+1}^{-,\min}) = \#\cA_{t,t+1}^{-,\min}$.
\end{proposition}

\begin{proof}
By replacing \cite[Lemmas 3.2.3, 3.2.6]{Chan_siDL} with Lemmas \ref{l:det_contr}, \ref{l:IJ bij}, the proof of \cite[Proposition 5.3.2]{Chan_siDL} applies. (The proof is quite technical; simpler incarnations of this idea have appeared in \cite{Boyarchenko_12}, \cite{Chan_DLI}, \cite{Chan_DLII}.)
\end{proof}

\begin{proof}[Proof of Theorem \ref{t:hom}]
By Proposition \ref{p:coh_beta}, we need to calculate $H_c^i(\bA[\cA^-], P^* \Loc_\chi)$. Since $P(\bA[\cA^- \smallsetminus \cA^{-,\min}]) = \{1\} \in \bW_h^1$ and $\#(\cA^- \smallsetminus \cA^{-,\min}) = n'-1$, we see that
\begin{equation*}
H_c^i(\bA[\cA^-], P^* \Loc_\chi) = H_c^i(\bA[\cA^{-,\min}], P^* \Loc_\chi)[2(n'-1)] \otimes ((-q^{n/2})^{2(n'-1)})^{\deg}.
\end{equation*}
Using Propositions \ref{p:induct factor} and \ref{p:induct extra} iteratively, we have
\begin{align*}
H_c^i&(\bA[\cA^{-,\min}], P^* \Loc_\chi) \\
&= H_c^i(\bA[\cA_{\geq 0,0}^{-,\min}], P^* \Loc_{\chi_{\geq 0}}) && \text{(by def)} \\
&\cong H_c^i(\bA[\cA_{\geq 0,1}^{-,\min}], P^* \Loc_{\geq 1})[2e_0] \otimes \Big((-q^{n/2})^{2e_0}\Big)^{\deg} && \text{(Prop \ref{p:induct factor})} \\
&\cong H_c^i(\bA[\cA_{\geq 1,1}^{-,\min}], P^* \Loc_{\geq 1})^{\oplus q^{nf_0/2}}[f_0 + 2 e_0] \otimes \Big((-q^{n/2})^{f_0 + 2e_0}\Big)^{\deg} && \text{(Prop \ref{p:induct extra})} \\
&\cong H_c^i(\bA[\cA_{\geq 1,2}^{-,\min}], P^* \Loc_{\geq 2})^{\oplus q^{nf_0/2}}[f_0 + 2(e_0 + e_1)] \otimes \Big((-q^{n/2})^{f_0 + 2(e_0 + e_1)}\Big)^{\deg} && \text{(Prop \ref{p:induct factor})}
\end{align*}
and so forth until
\begin{equation*}
\cong H_c^i(\bA[\cA_{\geq d',d'+1}^{-,\min}], P^* \Loc_{\chi_{\geq d'+1}})^{\oplus q^{nf_\chi/2}}[f_\chi + 2 e_\chi] \otimes \Big((-q^{n/2})^{f_\chi + 2 e_\chi}\Big),
\end{equation*}
where 
\begin{equation*}
f_\chi \colonequals f_0 + f_1 + \cdots + f_{d'-1}, \qquad e_\chi \colonequals e_0 + e_1 + \cdots + e_{d'}.
\end{equation*}
Since $\cA_{\geq d',d'+1} = \varnothing$, now we have shown
\begin{equation}\label{e:min to pt}
H_c^i(\bA[\cA^{-,\min}], P^* \Loc_\chi) \cong H_c^i(*, \overline \QQ_\ell)^{\oplus q^{n f_\chi/2}}[f_\chi + 2 e_\chi] \otimes \Big((-q^{n/2})^{f_\chi + 2 e_\chi}\Big)^{\deg}.
\end{equation}
Set $r_\chi \colonequals 2(n'-1) + f_\chi + 2 e_\chi$. By Proposition \ref{p:coh_beta}, we now have
\begin{equation*}
\Hom_{\bG_h^1(\FF_q)}\Big(\Ind_{\bT_h^1(\FF_q)}^{\bG_h^1(\FF_q)}(\chi), H_c^i(X_h, \overline \QQ_\ell)\Big) \cong 
\begin{cases}
\overline \QQ_\ell^{q^{nf_\chi/2}} & \text{if $i = r_\chi$,} \\
0 & \text{otherwise.}
\end{cases}
\end{equation*}
Moreover, since $\Fr_{q^n}$ acts trivially on $H_c^i(\bA[\cA_{\geq d',d'+1}^{-,\min}], P^* \Loc_{\chi_{\geq d'+1}}) = H^0(*, \overline \QQ_\ell)$, then $\Fr_{q^n}$ acts by multiplication by $(-1)^{r_\chi} q^{nr_\chi/2}$ on the above space of homomorphisms. 

To finish the proof of Theorem \ref{t:hom}, we need only calculate $e_\chi$, $f_\chi$, $r_\chi$. Unwinding the definitions of indexing sets given in Section \ref{s:indexing}, we have, for $0 \leq t \leq d'$,
\begin{align*}
\cA_{\geq t,t}^{-,\min} &=
\left\{
(i,1,l) \in \bZ^{\oplus 3} :
\begin{gathered}
\text{$2 \leq i \leq n$, $i \equiv 1 \!\!\!\! \pmod{m_t}$} \\
\text{$0 \leq l \leq h_t - 2$ if $[i]_{n_0} \neq 1$} \\
\text{$1 \leq l \leq h_t - 2$ if $[i]_{n_0} = 1$}
\end{gathered}
\right\}, \\
\cA_{\geq t, t+1}^{-,\min} &=
\left\{
(i,1,l) \in \bZ^{\oplus 3} :
\begin{gathered}
\text{$2 \leq i \leq n$, $i \equiv 1 \!\!\!\! \pmod{m_t}$} \\
\text{$0 \leq l \leq h_{t+1} - 2$ if $[i]_{n_0} \neq 1$} \\
\text{$1 \leq l \leq h_{t+1} - 2$ if $[i]_{n_0} = 1$}
\end{gathered}
\right\}.
\end{align*}
Therefore, we have
\begin{align*}
e_t &= \Big(\frac{n}{m_t} - 1\Big)(h_t - h_{t+1}) \qquad \text{if $0 \leq t \leq d'-1$}, \\
e_{d'} &= \Big(\frac{n}{m_{d'}} - 1\Big)(h_{d'} - 1) - \Big(\frac{n}{\lcm(m_{d'},n_0)} - 1\Big).
\end{align*}
For $0 \leq t \leq d'-1$, we have
\begin{equation*}
\cA_{t,t+1}^{-,\min} =
\left\{
(i,1,l) \in \bZ^{\oplus 3} :
\begin{gathered}
\text{$2 \leq i \leq n$, $i \equiv 1 \!\!\!\! \pmod{m_{t}}$, $i \not\equiv 1 \!\!\!\! \pmod{m_{t+1}}$} \\
\text{$0 \leq l \leq h_{t+1} - 2$ if $[i]_{n_0} \neq 1$} \\
\text{$1 \leq l \leq h_{t+1} - 2$ if $[i]_{n_0} = 1$}
\end{gathered}
\right\}
\end{equation*}
so that
\begin{equation*}
f_t = \Big(\frac{n}{m_t} - \frac{n}{m_{t+1}}\Big)(h_{t+1}-1) - \Big(\frac{n}{\lcm(m_t,n_0)} - \frac{n}{\lcm(m_{t+1},n_0)}\Big). \qedhere
\end{equation*}
\end{proof}

\subsection{The nonvanishing cohomological degree}

In this section, we use the results of the preceding sections to finish the proof of Theorem \ref{t:single_degree}. Observe that from Theorem \ref{t:hom} together with Corollary \ref{t:irred}, we have the following:

\begin{corollary}\label{c:s_chi}
Let $\pi$ be an irreducible constituent of $H_c^r(Z_h^1, \overline \QQ_\ell)$ for some $r$. Then
\begin{equation*}
\Hom_{\bG_h^1(\FF_q)}\left(\pi, H_c^i(\Z, \overline \QQ_\ell)\right) = 0 \qquad \text{for all $i \neq r$.}
\end{equation*}
In particular, for any $\chi \from \bT_h^1(\FF_q) \to \overline \QQ_\ell^\times$, there exists a positive integer $s_\chi$ such that
\begin{equation*}
H_c^i(\Z, \overline \QQ_\ell)[\chi] =
\begin{cases}
\text{irreducible} & \text{if $i = s_\chi$,} \\
0 & \text{if $i \neq s_\chi$.}
\end{cases}
\end{equation*}
\end{corollary}

\begin{proof}
This is the same as the proof of \cite[Corollary 5.1.3]{Chan_siDL}. The irreducible $\bG_h^1(\FF_q)$-representation $\pi \subset H_c^r(\Z, \overline \QQ_\ell)$ is a summand of $\Ind_{\bT_h^1(\FF_q)}^{\bG_h^1(\FF_q)}(\chi')$ for some $\chi'$. Hence 
\begin{equation*}
\Hom_{\bG_h^1(\FF_q)}\left(\Ind_{\bT_h^1(\FF_q)}^{\bG_h^1(\FF_q)}(\chi'), H_c^r(\Z, \overline \QQ_\ell)\right) \neq 0.
\end{equation*}
Theorem \ref{t:hom} implies that $r = r_{\chi'}$ and that there are no $\bG_h^1(\FF_q)$-equivariant homomorphisms from $\pi$ to $H_c^i(\Z, \overline \QQ_\ell)$ for $i \neq r_{\chi'}$. This proves the first assertion.

To see the second assertion, first recall from Corollary \ref{t:irred} that $H_c^*(\Z, \overline \QQ_\ell)[\chi]$ is (up to sign) an irreducible $\bG_h^1(\FF_q)$-representation. Therefore, we may apply the above argument to $H_c^*(\Z, \overline \QQ_\ell)[\chi]$ and we see that if $H_c^*(\Z, \overline \QQ_\ell)[\chi]$ is a summand of $\Ind_{\bT_h^1(\FF_q)}^{\bG_h^1(\FF_q)}(\chi')$, then
\begin{equation*}
H_c^i(\Z, \overline \QQ_\ell)[\chi] = \begin{cases}
\text{irreducible} & \text{if $i = r_{\chi'}$,} \\
0 & \text{otherwise.}
\end{cases}
\end{equation*}
Since the number $r_{\chi'}$ only depends on $\chi$, we final assertion of the corollary holds taking $s_\chi = r_{\chi'}$.
\end{proof}

We see now that the upshot of Theorem \ref{t:hom} is that we already know that $H_c^i(\Z, \overline \QQ_\ell)[\chi]$ is concentrated in a single degree $s_\chi$. However, it would be much more satisfying---for many reasons, computational, conceptual, idealogical---if we could pinpoint this nonvanishing cohomological degree. Taking a hint from the proof of Corollary \ref{c:s_chi}, one strategy to prove that $s_\chi = r_\chi$ is to prove that $H_c^{s_\chi}(\Z, \overline \QQ_\ell)[\chi]$ is a summand of $\Ind_{\bT_h^1(\FF_q)}^{\bG_h^1(\FF_q)}(\chi)$. This is our next result.

\begin{theorem}\label{t:r=s}
For any $\chi \from \bT_h^1(\FF_q) \to \overline \QQ_\ell^\times$,
\begin{equation*}
\Hom_{\bG_h^1(\FF_q)}\left(\Ind_{\bT_h^1(\FF_q)}^{\bG_h^1(\FF_q)}(\chi), H_c^{s_\chi}(Z_h^1, \overline \QQ_\ell)[\chi]\right) \neq 0.
\end{equation*}
In particular, $s_\chi = r_\chi$.
\end{theorem}

The proof of Theorem \ref{t:r=s} is essentially the same proof as \cite[Theorem 6.2.4]{Chan_siDL}. By Frobenius reciprocity, it is enough to show
\begin{equation}\label{e:T_hom}
\Hom_{\bT_h^1(\FF_q)}\left(\chi, H_c^{s_\chi}(\Z, \overline \QQ_\ell)[\theta]\right) \neq 0.
\end{equation}
We will sometimes write $\bT_h^1 = \bT_{h,n,q}^1$ and $\bG_h^1 = \bG_{h,n,q}^1$, $\Z = X_{h,n,q}^1$, $g_b^{n,q}$, and $s_\chi = s_\chi^{n,q}$ to emphasize the dependence on $n,q$. It is clear that once \eqref{e:T_hom} is established, then by Theorem \ref{t:hom}, it follows that $s_\chi = r_\chi$. For notational convenience, we write $H_c^i(X)$ to mean $H_c^i(X, \overline \QQ_\ell)$. We first establish a few lemmas. 

\begin{lemma}\label{l:0+ reg}
For any $\zeta \in \FF_{q^n}^\times$ with trivial $\Gal(\FF_{q^n}/\FF_q)$-stabilizer and any $g \in \bT_h^1(\FF_q)$,
\begin{equation*}
\Tr\Big((\zeta, 1, g) ; H_c^{s_\chi}(\Z)[\chi]\Big) = (-1)^{s_\chi} \chi(g).
\end{equation*}
\end{lemma}

\begin{proof}
Recall that the action of $(\zeta, 1, 1) \in \Gamma_h$ is given by conjugation. Observe that if $x \in (\Z)^{(\zeta, 1, 1)}$, then $x = g_b(v_1, 0, \ldots, 0)$. Furthermore, this forces $v_1 \in \bW_h^1(\FF_{q^n})$. Therefore $(\Z)^{(\zeta, 1, 1)} = \bT_h^1(\FF_q)$. By the Deligne--Lusztig fixed point formula,
\begin{align*}
\Tr\Big((\zeta, g, 1)^* ; H_c^*(\Z)[\chi]\Big)
&= \frac{1}{\#\bT_h^1(\FF_q)} \sum_{t \in \bT_h^1(\FF_q)} \chi(t)^{-1} \Tr\Big((\zeta, g, t)^* ; H_c^*(\Z)\Big) \\
&= \frac{1}{\#\bT_h^1(\FF_q)} \sum_{t \in \bT_h^1(\FF_q)} \chi(t)^{-1} \Tr\Big((1, g, t)^* ; H_c^*((\Z)^{(\zeta, 1, 1)}) \Big) \\
&= \frac{1}{\#\bT_h^1(\FF_q)} \sum_{t \in \bT_h^1(\FF_q)} \chi(t)^{-1} \Tr\Big((1,g,t)^* ; H_c^*(\bT_h^1(\FF_q))\Big) \\
&= \frac{1}{\#\bT_h^1(\FF_q)} \sum_{t \in \bT_h^1(\FF_q)} \chi(t)^{-1} \sum_{\chi' \from \bT_h^1(\FF_q) \to \overline \QQ_\ell^\times} \chi'(g) \chi'(t) = \chi(g).
\end{align*}
The conclusion of the lemma now follows from Corollary \ref{c:s_chi}.
\end{proof}

\begin{lemma}\label{l:induction}
Let $p_0$ be a prime dividing $n$. For any $\zeta \in \FF_{q^{p_0}}^\times \smallsetminus \FF_q^\times$ and any $g \in \bT_h^1(\FF_q)$,
\begin{equation*}
(-1)^{s_\chi^{n,q}} \Tr\Big((\zeta, 1, g); H_c^{s_\chi^{n,q}}(X_{h,n,q}^1)[\chi]\Big) = (-1)^{s_\chi^{n/p_0, q^{p_0}}}\Tr\Big((1,1,g) ; H_c^{s_\chi^{n/p_0,q^{p_0}}}(X_{h,n/p_0, q^{p_0}}^1)[\chi]\Big).
\end{equation*}
\end{lemma}

\begin{proof}
Recall that the action of $(\zeta, 1, 1) \in \Gamma_h$ is given by conjugation. Observe that if $x \in (\Z)^{(\zeta, 1, 1)}$, then $x = g_b(v_1, \ldots, v_n)$ where $v_i = 0$ for all $i \not\equiv 1$ modulo $p_0$. The map
\begin{align*}
f \from (X_{h,n,q}^1)^{(\zeta, 1, 1)} &\to X_{h,n/p_0,q^{p_0}}^1 \\
g_b^{n,q}(v_1, v_2, \ldots, v_n) &\mapsto g_b^{n/p_0, q^{p_0}}(v_1, v_{p_0+1}, v_{2p_0+1}, \ldots, v_{n-p_0+1})
\end{align*}
defines an isomorphism equivariant under the action of $\bT_{h,n,q}^1(\FF_q) \times \bT_{h,n,q}^1(\FF_q) \cong \bT_{h,n/p_0,q^{p_0}}^1(\FF_q) \times \bT_{h,n/p_0,q^{p_0}}^1(\FF_q)$. (Note that the determinant condition on the image can be seen by observing that the rows and columns of $x \colonequals g_b^{n,q}(v_1, \ldots, v_n)$ can be rearranged so that the matrix becomes block-diagonal of the form $\diag(f(x), \sigma^l(f(x)), \ldots, \sigma^{[l(p_0-1)]_n}(f(x)))$. Hence the determinant of $x$ is fixed by $\sigma$ if and only if the determinant of $f(x)$ is fixed by $\sigma^{p_0}$.)

By the Deligne--Lusztig fixed-point formula,
\begin{equation*}
\Tr\Big((\zeta, g, t)^*; H_c^*(X_{h,n,q}^1)\Big) = \Tr\Big((1,g, t)^* ; H_c^*(X_{h,n,q}^1)^{(\zeta, 1, 1)}\Big),
\end{equation*}
so that
\begin{align*}
\Tr\Big((\zeta, g, 1)^* ; H_c^*(X_{h,n,q}^1)[\chi]\Big)
&= \frac{1}{\#\bT_h^1(\FF_q)} \sum_{t \in \bT_h^1(\FF_q)} \chi(t)^{-1} \Tr\Big((\zeta, g, t)^* ; H_c^*(X_{h,n,q}^1)\Big) \\
&= \frac{1}{\# \bT_h^1(\FF_q)} \sum_{t \in \bT_h^1(\FF_q)} \chi(t)^{-1} \Tr\Big((1, g, t)^* ; H_c^*((X_{h,n,q}^1)^{(\zeta, 1, 1)})[\chi]\Big) \\
&= \frac{1}{\#\bT_h^1(\FF_q)} \sum_{t \in \bT_h^1(\FF_q)} \chi(t)^{-1} \Tr\Big((1, g, t)^* ; H_c^*(X_{h,n/p_0,q^{p_0}}^1)\Big) \\
&= \Tr\Big((1, g,1)^* ; H_c^*(X_{h,n/p_0,q^{p_0}}^1)[\chi]\Big).
\end{align*}
The conclusion of the lemma now holds by Corollary \ref{c:s_chi}.
\end{proof}

\begin{lemma}\label{l:chi tilde}
Let $\chi \from \bT_h^1(\FF_q) \to \overline \QQ_\ell^\times$. Assume that we are in one of the following cases:
\begin{enumerate}[label=(\arabic*)]
\item
$n > 1$ is odd and $p_0$ is a prime divisor of $n$.
\item
$n > 1$ is even and $p_0 = 2$.
\end{enumerate}
Fix a $\zeta \in \FF_{q^{p_0}}^\times$ such that $\langle \zeta \rangle = \FF_{q^{p_0}}^\times$ and consider the extension of $\chi$ defined by
\begin{equation*}
\widetilde \chi \from \FF_{q^{p_0}}^\times \times \bT_h^1(\FF_q) \to \overline \QQ_\ell^\times, \qquad (\zeta^i, g) \mapsto \begin{cases}
\chi(g) & \text{if $q$ is even,} \\
((-1)^{s_\chi^{n,q}+s_\chi^{n/p_0,q^{p_0}}})^i \cdot \chi(g) & \text{if $q$ is odd.}
\end{cases}
\end{equation*}
Then
\begin{equation*}
\sum_{x \in \FF_{q^{p_0}}^\times \smallsetminus \FF_q} \widetilde \chi(x, 1)^{-1} \neq 0.
\end{equation*}
\end{lemma}

\begin{proof}
This is the same proof as \cite[Lemma 6.2.6]{Chan_siDL}.
\end{proof}

\begin{proof}[Proof of Theorem \ref{t:r=s}]
The proof is exactly as in \cite[Theorem 6.2.4]{Chan_siDL}. We give a sketch here. Since $X_{h,1,q}^1 = \bT_h^1(\FF_q)$ and hence for any $\chi \from \bT_h^1(\FF_q) \to \overline \QQ_\ell^\times$, we have 
\begin{equation*}
H_c^{s_\chi^{1,q}}(X_{h,1,q}^1)[\chi] = H_c^0(\bT_h^1(\FF_q))[\chi] = \chi,
\end{equation*}
so Equation \eqref{e:T_hom} holds for $n = 1$ and $q$ arbitrary. We induct on the number of prime divisors of $n$: assume that for a fixed integer $l \geq 0$, Equation \eqref{e:T_hom} holds for any $\prod_{i=1}^l p_i$ and arbitrary $q$, where the $p_i$ are (possibly non-distinct) primes. We will show that Equation \eqref{e:T_hom} holds for any $\prod_{i=0}^l p_i$ and arbitrary $q$. If $n$ is even, let $p_0 = 2$; otherwise, $p_0$ can be taken to be anything. Let $\widetilde \chi$ be as in Lemma \ref{l:chi tilde}. Then
\begin{align*}
&\sum_{(x,g) \in \FF_{q^{p_0}}^\times \times \bT_h^1(\FF_q)} \widetilde \chi(x,g)^{-1} \Tr\Big((x,1,g) ; H_c^{s_\chi^{n,q}}(X_{h,n,q}^1)[\chi]\Big) \\
&= \#(\FF_q^\times \times \bT_h^1(\FF_q)) \cdot \dim \Hom_{\FF_q^\times \times \bT_h^1(\FF_q)}\Big(\widetilde \chi, H_c^{s_\chi^{n,q}}(X_{h,n,q}^1)[\chi]\Big) \\
&\quad+ \sum_{\substack{(x,g) \in \FF_{q^{p_0}}^\times \times \bT_h^1(\FF_q) \\ x \in \FF_{q^{p_0}}^\times \smallsetminus \FF_q^\times}} \widetilde \chi(x,g)^{-1} \cdot (-1)^{s_\chi^{n,q}+s_\chi^{n/p_0,q^{p_0}}} \cdot  \Tr\Big((1,1,g); H_c^{s_\chi^{n/p_0,q^{p_0}}}(X_{h,n/p_0,q^{p_0}}^1)[\chi]\Big).
\end{align*}
By the inductive hypothesis together with Lemma \ref{l:chi tilde}, the second summand is a nonzero number, and hence necessarily either the left-hand side is positive or the first summand is positive. In either case, Equation \eqref{e:T_hom} must hold.
\end{proof}

For the reader's benefit, we summarize the discussion of this section to prove Theorem \ref{t:single_degree}.

\begin{proof}[Proof of Theorem \ref{t:single_degree}]
By Corollary \ref{t:irred}, we know that $H_c^*(\Z, \overline \QQ_\ell)[\chi]$ is (up to sign) an irreducible $\bG_h^1(\FF_q)$-representation. By Theorem \ref{t:hom}, for any character $\chi'$,
\begin{equation*}
\Hom_{\bG_h^1(\FF_q)}\Big(\Ind_{\bT_h^1(\FF_q)}^{\bG_h^1(\FF_q)}(\chi'), H_c^i(\Z, \overline \QQ_\ell)\Big) \neq 0 \qquad \Longleftrightarrow \qquad i = r_{\chi'}.
\end{equation*}
As explained in Corollary \ref{c:s_chi}, this implies that if $H_c^*(\Z, \overline \QQ_\ell)[\chi]$ is a summand of $\Ind_{\bT_h^1(\FF_q)}^{\bG_h^1(\FF_q)}(\chi')$ for some $\chi'$, then
\begin{equation*}
H_c^i(\Z, \overline \QQ_\ell)[\chi] \neq 0 \qquad \Longleftrightarrow \qquad i = r_{\chi'} \equalscolon s_\chi.
\end{equation*}
By Theorem \ref{t:r=s}, we see that in fact we can take $\chi' = \chi$, and therefore the nonvanishing cohomological degree of $H_c^i(\Z, \overline \QQ_\ell)[\chi]$ is in fact $i = r_\chi$. The final assertion about the action of $\Fr_{q^n}$ on $H_c^{r_\chi}(\Z, \overline \QQ_\ell)[\theta] = (-1)^{r_\chi} H_c^*(\Z, \overline \QQ_\ell)[\theta]$ now follows from Theorem \ref{t:hom}.
\end{proof}

\subsection{Dimension formula}

We use Theorem \ref{t:single_degree} to give an explicit dimension formula for the $\bG_h^1(\FF_q)$-representation $H_c^*(\Z, \overline \QQ_\ell)[\chi]$.

\begin{corollary}\label{c:dimension}
If $\chi \from \bT_h^1(\FF_q) \cong \bW_h^1(\FF_{q^n}) \to \overline \QQ_\ell^\times$  is any character, then
\begin{equation*}
\dim H_c^{r_\chi}(\Z, \overline \QQ_\ell)[\chi] = q^{(n^2 - n)(h-1) - nr_\chi/2}.
\end{equation*}
In particular, if $\chi$ has trivial $\Gal(L/k)$-stabilizer, then
\begin{equation*}
\log_q(\dim H_c^{r_\chi}(\Z, \overline \QQ_\ell)[\chi]) = \frac{n}{2}\textstyle\Big(n(h_1-1)-(h_{d'}-1)-(n'-1) - \sum\limits_{t=1}^{d'-1} \frac{n}{m_t}(h_t - h_{t+1})\Big).
\end{equation*}
\end{corollary}

\begin{proof}
By applying \cite[Lemma 2.12]{Boyarchenko_12} to calculate the character of $H_c^{r_\chi}(\Z, \overline \QQ_\ell)[\chi]$ at the identity, we have
\begin{equation*}
\dim H_c^{r_\chi}(\Z, \overline \QQ_\ell)[\chi] = \frac{(-1)^{r_\chi}}{\lambda \cdot \#\bT_h^1(\FF_q)} \sum_{t \in \bT_h^1(\FF_q)} \chi(t) \cdot \#S_{1, t},
\end{equation*}
where $S_{1,t} = \{x \in \Z(\overline \FF_q) : \sigma(\Fr_{q^n}(x)) = x \cdot t\}$ and $\lambda$ is the scalar by which $\Fr_{q^n}$ acts on $H_c^{r_\chi}(\Z, \overline \QQ_\ell)[\chi]$. Suppose that $x \in S_{1,t}$. Then by the same argument as \cite[Lemma 9.3]{CI_ADLV}, $\det(b\sigma(g_b(x))) = t \cdot \det(b) \det(g_b(x))$, which then forces $t = 1$. By construction, $S_{1,1} = \bG_h^1(\FF_q)$, so therefore
\begin{equation*}
\dim H_c^{r_\chi}(\Z, \overline \QQ_\ell)[\chi] = \frac{\#\bG_h^1(\FF_q)}{q^{nr_\chi/2} \cdot \#\bT_h^1(\FF_q)} = q^{(n^2-n)(h-1) - nr_\chi/2},
\end{equation*}
where we also use the fact that $\lambda = (-1)^{r_\chi} q^{nr_\chi/2}$ from Theorem \ref{t:single_degree}. The assertion in the case that $\chi$ has trivial $\Gal(L/k)$-stabilizer follows from Corollary \ref{c:prounip_degree}.
\end{proof}
 
\section{Conjectures}

\subsection{Concentration in a single degree}

Recall that from Corollary \ref{t:irred}, we know that if $\theta \from \bT_h(\FF_q) \cong \bW_h^\times(\FF_{q^n}) \to \overline \QQ_\ell^\times$ is a character with trivial $\Gal(\FF_{q^n}/\FF_{q^{n_0r}})$-stabilizer, then the alternating sum $H_c^*(X_h \cap \bL_h^{(r)}\bG_h^1, \overline \QQ_\ell)[\theta]$ is (up to sign) an irreducible $\bL_h^{(r)}(\FF_q)\bG_h^1(\FF_q)$-representation. We conjecture that in fact these cohomology groups should be concentrated in a single degree.

\begin{conjecture}\label{c:single_degree}
Let $r \mid n'$ and let $\theta \from \bT_h(\FF_q) \cong \bW_h^\times(\FF_{q^n}) \to \overline \QQ_\ell^\times$ be a character with trivial $\Gal(\FF_{q^n}/\FF_{q^{n_0r}})$-stabilizer. Then there exists an integer $i_{\theta,r}$ such that
\begin{equation*}
H_c^i(X_h \cap \bL_h^{(r)}\bG_h^1, \overline \QQ_\ell)[\theta] \neq 0 \qquad \Longleftrightarrow \qquad i = i_{\theta,r}.
\end{equation*}
\end{conjecture}

In this paper, we proved this conjecture in the case $r = n'$ and in fact pinpointed the nonvanishing cohomological degree $i_{\theta,n'}$ (Theorem \ref{t:single_degree}). We expect that a similar formula for $i_{\theta,r}$ should be obtainable, where the methods in this paper can be used to reduce the determination of $i_{\theta,r}$ to a ``depth-zero'' setting. The hypotheses of Conjecture \ref{c:single_degree} should be equivalent to saying that the consequent depth-zero input comes from the $\theta_0$-isotypic part of the cohomology of a classical Deligne--Lusztig variety (of dimension $r-1$) for the twisted Levi $\bL_{1,r}$ in $\bG_1$, where $\theta_0$ is a character of $\bT_1(\FF_q) \cong \FF_{q^n}^\times$ in general position.

\subsection{Relation to loop Deligne--Lusztig varieties}

The varieties $X_h$ are closely related to a conjectural construction of Deligne--Lusztig varieties for $p$-adic groups initiated by Lusztig \cite{Lusztig_79}. We call these sets \textit{loop Deligne--Lusztig varieties}, although the algebro-geometric structure is still unknown in general. 

In \cite{CI_ADLV}, we studied this question for a certain class of these sets attached to inner forms of $\GL_n$. We prove (see also \cite[Proposition 2.6]{CI_loopGLn}) that the fpqc-sheafification $X$ of the presheaf on category $\Perf_{\overline \FF_q}$ of perfect $\FF_q$-schemes
\begin{equation*}
X \colon R \mapsto \{x \in LG(R) : x^{-1} F(x) \in LU(R)\}/L(U \cap F^{-1}U)
\end{equation*}
is representable by a perfect $\overline \FF_q$-scheme and that $X$ is the perfection of 
\begin{equation*}
\bigsqcup_{g \in G(k)/G_{x,0}(\cO_k)} g \cdot \varprojlim_h X_h.
\end{equation*}
We see that an intermediate step to understanding the cohomology of loop Deligne--Lusztig is to calculate the cohomology of $X_h$. 

However, for various reasons, it is often easier to calculate the cohomology of the Drinfeld stratification. For example, in \cite{CI_loopGLn}, to prove cuspidality of $H_*(X, \overline \QQ_\ell)[\theta]$ for a broad class of characters $\theta \from T(k) \to \overline \QQ_\ell^\times$, we  calculate the formal degree of this representation, which we achieve by calculating the dimension of $H_c^*(X_h^{(n')}, \overline \QQ_\ell)[\theta]$ from the Frobenius eigenvalues (see Corollary \ref{c:dimension}). In this setting, we can prove a comparison formula between the cohomology of $X_h^{(n')}$ and the cohomology of $X_h$ (see Section \ref{s:p>n closed}). 

We conjecture the following comparison theorem between the cohomology of $X_h$ and its Drinfeld stratification. In Section \ref{s:evidence}, we present evidence supporting the truth of this conjecture.

\begin{conjecture}\label{c:Xh}
Let $r \mid n'$ and let $\theta \from \bT_h(\FF_q) \cong \bW_h^\times(\FF_{q^n}) \to \overline \QQ_\ell^\times$ be a character with trivial $\Gal(L/k)$-stabilizer. Let $\chi \colonequals \theta|_{\bW_h^1(\FF_{q^n})}$ and assume that the stabilizer of $\chi$ in $\Gal(L/k)$ is equal to the unique index-$n_0r$ subgroup. 
Then we have an isomorphism of virtual $\bG_h(\FF_q)$-representations
\begin{equation*}
H_c^*(X_h, \overline \QQ_\ell)[\theta] \cong H_c^*(X_h^{(r)}, \overline \QQ_\ell)[\theta].
\end{equation*}
\end{conjecture}

Combining Conjectures \ref{c:single_degree} and \ref{c:Xh} with Corollary \ref{t:irred}, the above conjecture asserts that as elements of the Grothendieck group of $\bG_h(\FF_q)$,
\begin{align*}
H_c^*(X_h, \overline \QQ_\ell)[\theta] 
&= (-1)^{i_{\theta,r}} H_c^{i_{\theta,r}}(X_h^{(r)}, \overline \QQ_\ell)[\theta] \\
&= (-1)^{i_{\theta,r}} \Ind_{\bL_h^{(r)}(\FF_q)\bG_h^1(\FF_q)}^{\bG_h(\FF_q)}\Big(H_c^{i_{\theta,r}}(X_h \cap \bL_h^{(r)} \bG_h^1, \overline \QQ_\ell)[\theta]\Big).
\end{align*}

\subsubsection{Evidence}\label{s:evidence}

At present, we can prove Conjecture \ref{c:Xh} in some special cases. We discuss these various cases, their context, and the ideas involved in the proof. 

\vspace{5pt}
\paragraph{}

The most degenerate setting of Conjecture \ref{c:Xh} is when $G$ is a division algebra over $k$. Then $n' = 1$ and so the closed Drinfeld stratum $X_h^{(n')} = X_h^{(1)}$ is the only Drinfeld stratum. Additionally, we have that $X_h^{(n')}$ is a disjoint union of $\#\bG_h(\FF_q)$ copies of $X_h^1 \colonequals X_h \cap \bG_h^1$. In \cite{Chan_siDL}, all the technical calculations happen at the level of $X_h^1$ (though in different notation in \textit{op.\ cit.}), and using the new methods developed there, one knows nearly everything about the representations $H_c^i(X_h^1, \overline \QQ_\ell)[\chi]$ for arbitrary characters $\chi \from \bT_h^1(\FF_q) \to \overline \QQ_\ell^\times$. However, the expected generalization of these techniques extend not to $H_c^i(X_h, \overline \QQ_\ell)[\chi]$, but to $H_c^i(X_h^{(r)}, \overline \QQ_\ell)[\chi]$---hence one is really forced to work on the stratum in order to approach $X_h$ (at least with the current state of technology).

\vspace{5pt}
\paragraph{}\label{s:p>n closed}

Now let $G$ be any inner form of $\GL_n$ (as it has been this entire paper, outside Section \ref{s:drinfeld}). We are close to establishing Conjecture \ref{c:Xh} when $\chi = \theta|_{\bW_h^1(\FF_{q^n})}$ has trivial $\Gal(L/k)$-stabilizer. In this case, Conjecture \ref{c:Xh} says that $H_c^*(X_h, \overline \QQ_\ell)[\theta] \cong H_c^*(X_h^{(n')}, \overline \QQ_\ell)[\theta]$ as virtual $\bG_h(\FF_q)$-representations. In \cite[Theorem 4.1]{CI_loopGLn}, we prove this isomorphism holds under the additional assumption that $p > n$. The idea here is to use a highly nontrivial generalization of a method of Lusztig to calculate the inner product $\big\langle H_c^*(X_h, \overline \QQ_\ell)[\theta], H_c^*(X_h^{(n')}, \overline \QQ_\ell)[\theta]\big\rangle$ in the space of conjugation-invariant functions on $\bG_h(\FF_q)$.

\vspace{5pt}
\paragraph{}

In Appendix \ref{s:fibers}, we present a possible geometric approach to Conjecture \ref{c:Xh} which has its roots in the $\GL_2$ setting of the proof of \cite[Theorem 3.5]{Ivanov_15_ADLV_GL2_unram}. The idea is to study the fibers of the natural projection\footnote{When $G = \GL_n$, then this is literally what we do in Appendix \ref{s:fibers}. When $G$ is a nonsplit inner form of $\GL_n$, in order to get a shape analogous to the split case, we work with an auxiliary scheme which is an affine fibration over $X_h$.} $\pi \from X_h \to X_{h-1}$. We can show that the behavior of $\pi^{-1}(x)$ depends \textit{only} on the location of $x$ relative to the Drinfeld stratification of $X_h$: If $r$ is the smallest divisor of $n'$ such that $x \in X_h^{(r)}$ (i.e.\ $x$ is in the $r$th Drinfeld stratum $X_{h,r}$ of $X_h$), then there exists a morphism
\begin{equation*}
\pi^{-1}(x) \to \bigsqcup_{\bW_h^{h-1}(\FF_{q^{n_0 r}})} \bA^{n-1}
\end{equation*}
which is a composition of isomorphisms and purely inseparable morphisms. Moreover, the action of $\ker(\bW_h^{h-1}(\FF_{q^n}) \to \bW_h^{h-1}(\FF_{q^{n_0r}}))$ on $\pi^{-1}(x)$ fixes the set of connected components. The crucial point here is that the fibers of the natural map
\begin{equation*}
X_{h,r}/\ker(\bW_h^{h-1}(\FF_{q^n}) \to \bW_h^{h-1}(\FF_{q^{n/(n_0r)}})) \to X_{h-1,r}
\end{equation*}
are again isomorphic to $\bigsqcup_{\bW_h^{h-1}(\FF_{q^{n_0r}})}\bA^{n-1}$ and therefore $\ker(\bW_h^{h-1}(\FF_{q^n}) \to \bW_h^{h-1}(\FF_{q^{n/(n_0r)}}))$ acts trivially on the cohomology of $X_{h,r}$:
\begin{equation*}
H_c^*(X_{h,r}, \overline \QQ_\ell) \cong H_c^*(X_{h,r}, \overline \QQ_\ell)^{\ker(\bW_h^{h-1}(\FF_{q^n}) \to \bW_h^{h-1}(\FF_{q^{n/(n_0r)}}))}.
\end{equation*}
Using open/closed decompositions of $X_h$ via Drinfeld strata, we have that if $\theta$ is trivial on $\ker(\bW_h^{h-1}(\FF_{q^n}) \to \bW_h^{h-1}(\FF_{q^{n/(n_0r)}}))$, then 
\begin{equation*}
H_c^*(X_h, \overline \QQ_\ell)[\theta] \cong H_c^*(X_h^{(r)}, \overline \QQ_\ell)[\theta]
\end{equation*}
as virtual $\bG_h(\FF_q)$-representations. It seems reasonable to guess that if one can generalize Appendix \ref{s:fibers} to study the fibers of $X_h \to X_1$, then one could establish Conjecture \ref{c:Xh} using a similar reasoning as above.

\appendix

\section{The geometry of the fibers of projection maps}\label{s:fibers}

In this section, we study the fibers of the projection maps $X_h \to X_{h-1}$. This is a technical computation which we perform by first using the isomorphism $X_h \cong X_h(b,b_\cox)$ for a particular choice of $b$ which we call the \textit{special representative}. This is the first time in this paper that we see the convenience of having the alternative presentations of $X_h$ discussed in Sections \ref{s:Xhbw} and \ref{s:drinfeld b,w}.

\subsection{The special representative}

We first recall the content of Section \ref{s:drinfeld b,w} in the context of a particular representative of the $\sigma$-conjugacy class corresponding to the fixed integer $\kappa$.

\begin{definition}\label{d:special}
The \textit{special representative} $b_\xp$ attached to $\kappa$ is the block-diagonal matrix of size $n \times n$ with $(n_0 \times n_0)$-blocks of the form $\left(\begin{matrix} 0 & \varpi \\ 1_{n_0-1} & 0 \end{matrix}\right)^\kappa$.
\end{definition}

By \cite[Lemma 5.6]{CI_ADLV}, there exists a $g_0 \in G_{x,0}(\cO_{\breve k})$ such that $b_\xp = g_0 b_\cox \sigma(g_0)^{-1}$. Observe further that since $b_\xp, b_\cox$ are $\sigma$-fixed and $b_\xp^n = b_\cox^n = \varpi^{kn}$, 
\begin{equation*}
\sigma^n(g_0) = g_0.
\end{equation*}
Therefore $b_\xp$ satisfies the conditions of Lemma \ref{l:g0 independence}. Recall from Section \ref{s:drinfeld b,w} that we have
\begin{equation}\label{e:Xh}
X_h \cong X_h(b_\xp,b_\cox) \cong \{v \in \sL_h : \sigma(\det g_{b_\xp}(v)) = (-1)^{n-1} \det g_{b_\xp}(v) \in \bW_h^\times\},
\end{equation}
where 
\begin{align*}
\sL_h &= (\bW_h \oplus (V \bW_{h-1})^{\oplus n_0 - 1})^{\oplus n'} \subset \bW_h^{\oplus n} \\
g_{b_\xp}(v) &= \big( v_1 \, \big| \, v_2 \, \big| \, v_3 \, \big| \, \cdots \, \big| \, v_n\big) \\
\text{where } v_i &= \varpi^{\lfloor (i-1)k_0/n_0 \rfloor} \cdot (b_\xp \sigma)^{i-1}(v) \text{ for $1 \leq i \leq n-1$.} 
\end{align*}
In this section, we will work with 
\begin{equation}\label{e:Xh+}
X_h^+ \colonequals \{v \in \sL_h^+ : \sigma(\det g_{b_\xp(v)}) = \det g_{b_{\xp}}(v) \in \bW_h^\times\}
\end{equation}
where $\sL_h^+$ is now the subquotient of $\bW_{h+1}^{\oplus n}$ 
\begin{equation*}
\sL_h^+ \colonequals (\bW_h \oplus (V \bW_h)^{\oplus n_0 -1})^{\oplus n'},
\end{equation*}
and $g_{b_\xp}(v)$ is defined as before. Note that \eqref{e:Xh} differs from \eqref{e:Xh+} in that the former takes place in $G_{x,0}/G_{x,(h-1)+}$ and the latter takes place in $G_{x,0}/G_{x,h}$. A straightforward computation shows that the defining equation of $X_h^+$ does not depend on the quotient $\sL_h^+/\sL_h = \bA^{n-n'}$.

Observe that $\det g_{b_\xp}(\zeta v) = \Nm(\zeta) \cdot \det g_{b_\xp}(\zeta v)$ where $\Nm(\zeta) = \zeta \cdot \sigma(\zeta) \cdot \sigma^2(\zeta) \cdots \sigma^{n-1}(\zeta)$. Picking any $\zeta$ such that $\sigma(\Nm(\zeta)) = (-1)^{n-1} \Nm(\zeta)$ allows us to undo the $(-1)^{n-1}$ factor in the defining equation in \eqref{e:Xh}. In particular, this means
\begin{equation*}
H_c^i(X_h^+, \overline \QQ_\ell) = H_c^{i + 2(n-n')}(X_h, \overline \QQ_\ell), \qquad \text{for all $i \geq 0$.}
\end{equation*}
For each divisor $r \mid n'$, we define the $r$th Drinfeld stratum $X_{h,r}^+$ of $X_h^+$ to be the preimage of $X_{h,r}$ under the natural surjection $X_h^+ \to X_h$.

\subsection{Fibers of $X_{h,r}^+ \to X_{h-1,r}^+$}

For notational convenience, we write $b = b_\xp.$ We may identify $\sL_h^+ = \bA^{n(h-1)}$ with coordinates $x = \{x_{i,j}\}_{1 \leq i \leq n, \, 0 \leq j \leq h-1}$ which we typically write as $x = (\widetilde x, x_{1,h-1}, x_{2,h-1}, \ldots, x_{n,h-1}) \in \sL_{h-1}^+ \times \bA^n$; here, an element $v = (v_1, \ldots, v_n) \in \sL_h^+$ is such that $v_i = [x_{i,0}, x_{i,1}, \ldots, x_{i,n}]$ if $i \equiv 1 \pmod{n_0}$ and $v_i = [0, x_{i,0}, x_{i,1}, \ldots, x_{i,n}]$ if $i \not\equiv 1 \pmod{n_0}$. 

In this section, fix a divisor $r \mid n'$. From the definitions, $X_{h,r}^+$ can be viewed as the subvariety of $X_{h-1,r}^+ \times \bA^n$ cut out by the equation 
\begin{equation*}
0 = P_0(x)^q - P_0(x),
\end{equation*}
where $P_0$ is the coefficient of $\varpi^{h-1}$ in the expression $\det g_b^{\red}(v)$. Let $c$ denote the polynomial consisting of the terms of $P_0(x)$ which only depend on $\widetilde x$. An explicit calculation shows that there exists a polynomial $P_1$ in $x$ such that
\begin{equation}\label{e:c}
P_0(x) = c(\widetilde x) + \sum_{i=0}^{n_0-1} P_1(x)^{q^i}.
\end{equation}
Therefore $X_{h,r}^+$ is the subvariety of $X_{h-1,r}^+ \times \bA^n$ cut out by
\begin{equation*}
P_1(x)^{q^{n_0}} - P_1(x) = c(\widetilde x) - c(\widetilde x)^q.
\end{equation*}
One can calculate $P_1$ explicitly (see \cite[Proposition 7.5]{CI_ADLV}):

\begin{lemma}\label{lm:polynomial_P_describing_the_fiber_arbitrary_kappa}
Explicitly, the polynomial $P_1$ is
\begin{equation*}
P_1(x) = \sum_{1 \leq i,j \leq n^{\prime}} m_{ji} x_{1 + n_0(i-1),h-1}^{q^{(j-1)n_0}},
\end{equation*}
where $m \colonequals (m_{ji})_{j,i}$ is the adjoint matrix of $\overline{g_b}(\bar{x})$ and $\bar x$ denotes the image of $x$ in $\overline{V} = \sL_0/\sL_0^{(1)}$. Explicitly, $m \cdot \overline{g_b}(\bar{x}) = \det\overline{g_b}(\bar{x})$ and the $(j,i)$th entry of $m$ is $(-1)^{i+j}$ times the determinant of the $(n^{\prime}-1)\times (n^{\prime}-1)$ matrix obtained from $\overline{g_b}(\bar{x})$ by deleting the $i$th row and $j$th column.
\end{lemma}

The main result of this section is:

\begin{proposition}\label{p:Mr}
There exists an $X_{h-1,r}^+$-morphism
\begin{equation*}
M_r \from X_{h-1,r}^+ \times \bA^n \to X_{h-1,r}^+ \times \bA^n
\end{equation*}
(the left $\bA^n$ in terms of the coordinates $\{x_{i,h-1}\}_{i=1}^n$ and the right $\bA^n$ in terms of new coordinates $\{z_i\}_{i=1}^n$) satisfying the following properties:
\begin{enumerate}[label=(\roman*)]
\item
$M_r$ is a composition of $X_{h-1,r}^+$-isomorphisms and purely inseparable $X_{h-1,r}^+$-morphisms.
\item
$M_r(X_{h,r}^+)$ is the closed subscheme defined by the equation
\begin{equation*}
z_1^{q^{n_0 r}} - z_1 = c(\widetilde x) - c(\widetilde x)^q,
\end{equation*}
where $c$ is as in \eqref{e:c}.
\item
$M_r$ is $\bW_h^{h-1}(\FF_{q^n})$-equivariant after equipping the left $X_{h-1,r}^+ \times \bA^n$ with the $\bW_h^{h-1}(\FF_{q^n})$-action 
\begin{equation*}
1 + \varpi^{h-1} a \from x_{i,h-1} \mapsto x_{i,h-1} + x_{i,0} a, \qquad \text{for all $1 \leq i \leq n$,}
\end{equation*}
and the right $X_{h-1,r}^+ \times \bA^n$ with the $\bW_h^{h-1}(\FF_{q^n})$-action
\begin{equation*}
1 + \varpi^{h-1} a \from z_i \mapsto
\begin{cases}
z_1 + \Tr_{\FF_{q^n}/\FF_{q^{n_0r}}}(a) & \text{if $i = 1$,} \\
z_2 + a & \text{if $r \neq n'$ and $i = 2$,} \\
z_i & \text{otherwise.}
\end{cases}
\end{equation*}
\end{enumerate}
\end{proposition}

In the rest of this section we prove Proposition \ref{p:Mr}. To simplify the notation we will first establish the proposition in the case $\kappa = 0$ (i.e.\ $G = \GL_n$), and at the end generalize it to all $\kappa$. The first part of the proof of Proposition \ref{p:Mr} is given by the lemma below. Before stating it, we establish some notation. For an ordered basis $\sB$ of $V$ and $v \in V$, let $v_{\sB}$ denote the coordinate vector of $v$ in the basis $\sB$. For two ordered bases $\sB, \sC = \{c_i\}_{i=1}^n$ of $V$, let $M_{\sB,\sC}$ denote the base change matrix between them, that is, the $i$th column vector of $M_{\sB,\sC}$ is $c_{i, \sB}$. It is clear that 
\begin{itemize}
 \item $M_{\sC,\sB} = M_{\sB,\sC}^{-1}$,
 \item for any $v\in V$, $M_{\sB,\sC} v_{\sC} = v_{\sB}$,
 \item for a third ordered basis $\sD$ of $V$, one has $M_{\sB,\sC} M_{\sC,\sD} = M_{\sB,\sD}$.
\end{itemize}
For a linear map $f \colon V \rightarrow V$, let $M_{\sB,\sC}(f)$ denote the matrix representation of $f$; that is, $M_{\sB,\sC}(f)\cdot v_{\sC} = f(v)_{\sB}$. In $V$ we have the two ordered bases: 
\begin{align*}
\sE   &:= \text{ the standard basis of $V$, arising from the basis $\{e_i\}$ of the lattice $\sL_0$}, \\
\sB_x &:= \{ \sigma_b^{i-1}(x) \}_{i=1}^n, \text{attached to the given $x \in X_0^+$.}
\end{align*}
We identify $V$ with $\overline{\mathbb F}_q^n$ via the standard basis $\sE$ and write $v = v_{\sE}$ for all $v \in V$.

\begin{lm}\label{lm:linear_change_of_variables_fibers}
Assume $\kappa = 0$. There exists an $X_{h-1,r}^+$-isomorphism $X_{h-1,r}^+ \times \bA^n \stackrel{\sim}{\rightarrow} X_{h-1, r}^+ \times \bA^n$ given by a linear change of variables $x_{i,h-1} \rightsquigarrow x_{i,h-1}^{\prime}$, such that $P_1$ in the new coordinates $x_{i,h-1}^{\prime}$ takes the form
\[
P_1 = x_{1,h-1}^{\prime} + x_{1,h-1}^{\prime,q} + \dots + x_{1,h-1}^{\prime, q^{n-1}} + \sum_{j=0}^s \sum_{\lambda = i_j + 1}^{i_{j+1}} x_{s+2-j,h-1}^{\prime, q^{\lambda}},
\]
and the action of $1 + \varpi^{h-1} a \in W_h^{h-1}(\FF_{q^n})$ on the coordinates $x_{i,h-1}^{\prime}$ is given by
\begin{equation}\label{eq:Wh_action_on_intermediate_coordinates}
x_{i,h-1}^{\prime} \mapsto \begin{cases} x_{1,h-1}^{\prime} + a & \text{if $i=1$,} \\ 
x_{i,h-1} & \text{if $i \geq 2$.} \end{cases}
\end{equation}
\end{lm}

\begin{proof}[Proof of Lemma \ref{lm:linear_change_of_variables_fibers}]
We have to find a morphism $C := (c_{ij}) \colon X_{h-1,r}^+ \rightarrow \GL(V) = \GL_{n,\FF_q}$ (this identification uses the standard basis $\sE$ of $V$) such that the corresponding linear change of coordinates
\begin{equation}\label{eq:general_coordinate_change}
x_{i,h-1} = c_{i,1} x_{1,h-1}^{\prime} + c_{i,2} x_{2,h-1}^{\prime} + \dots + c_{i,n} x_{n,h-1}^{\prime}, \text{ for all $1 \leq i \leq n$}. 
\end{equation}
brings $P_1$ to the requested form. Moreover, it suffices to do this fiber-wise by first determining $C(\tilde x)$ for any point $\tilde x \in X_{h-1,r}^+$ and then seeing that $\tilde x \mapsto C(\tilde x)$ is in fact an algebraic morphism.

Fix $\tilde{x} \in X_{h-1,r}^+$ with image $x \in X_1^+$, and write $C$ instead of $C(\tilde{x})$ to simplify notation. Let $C_i$ denote the $i$th column of $C$. Our coordinate change replaces $P_1$ by the polynomial (after dividing by the irrelevant non-zero constant $\det g_b(x) \in \mathbb{F}_q^{\times}$) 
\small
\begin{align} \label{eq:poly_P_basechanged_in_general}
P_1 &= x_{1,h-1}^{\prime} (m_1 \cdot C_1) + x_{1,h-1}^{\prime,q} (m_2 \cdot \sigma_b(C_1)) + x_{1,h-1}^{\prime, q^2} (m_3 \cdot \sigma_b^2(C_1)) + \dots + x_{1,h-1}^{\prime, q^{n-1}} (m_n \cdot \sigma_b^{n-1}(C_1)) \nonumber \\
  &+ x_{2,h-1}^{\prime} (m_1 \cdot C_2) + x_{2,h-1}^{\prime,q} (m_2 \cdot \sigma_b(C_2)) + x_{2,h-1}^{\prime,q^2} (m_3 \cdot \sigma_b^2(C_2)) + \dots + x_{2,h-1}^{\prime,q^{n-1}} (m_n \cdot \sigma_b^{n-1}(C_2)) \nonumber \\
  &+ \cdots + \nonumber \\
  &+ x_{n,h-1}^{\prime} (m_1 \cdot C_n) + x_{n,h-1}^{\prime,q} (m_2 \cdot \sigma_b(C_n)) + x_{n,h-1}^{\prime,q^2} (m_3 \cdot \sigma_b^2(C_n)) + \dots + x_{n,h-1}^{\prime,q^{n-1}} (m_n \cdot \sigma_b^{n-1}(C_n)) \nonumber \\
\end{align}
\normalsize
in the indeterminates $\{x_{i,h-1}^{\prime}\}_{i=1}^n$. Here, we write $m_i$ to mean the $i$th row of the matrix $m$ (adjoint to $g_b(x)$) from Lemma \ref{lm:polynomial_P_describing_the_fiber_arbitrary_kappa}. For $z \in V$, we put 
\begin{equation}\label{eq:def_m_ast}
m\ast z = \sum_{i=1}^{n} (m_i \cdot (b\sigma)^{i-1}(z)) e_i. 
\end{equation}
%$m \ast z := \begin{pmatrix} m_1 \cdot z \\ m_2 \cdot \sigma_b(z) \\ \dots \\ m_n \cdot \sigma_b^{n-1}(z) \end{pmatrix}$. 
The intermediate goal is to describe the map $m \ast \colon V \rightarrow V$ in terms of a coordinate matrix. Of course, $m\ast$ is not linear, but its composition with the projection on the $i$th component (corresponding to the $i$th standard basis vector) is $\sigma^{i-1}$-linear. Thus we instead will describe the linear map $(m\ast)^{\prime} \colon V \rightarrow V$, which is the composition of $m\ast$ and the map $\sum_i v_i e_i \mapsto \sum_i \sigma^{-(i-1)}(v_i) e_i$.  This is done by the following lemma.

\begin{lm}\label{lm:explicit_construction_of_M_BxEmast}
Assume $\kappa = 0$. We have 
\[
M_{\sE, \sB_x}((m\ast)^{\prime}) = 
\begin{pmatrix} 
1 & 0 & 0 & \cdots & 0 & 0 \\ 
1 & 0 & 0 & \cdots & 0 & \sigma^{-1}(y_1) \\ 
\vdots & \vdots & \vdots & \text{\reflectbox{$\ddots$}} & \sigma^{-2}(y_2) & \ast \\ 
1 & 0 & 0 & \text{\reflectbox{$\ddots$}} & * & \vdots \\ 
1 & 0 & \sigma^{-(n-2)}(y_{n-2}) & \text{\reflectbox{$\ddots$}} & \vdots & \ast \\ 
1 & \sigma^{-(n-1)}(y_{n-1}) & * & \cdots & * & \ast \end{pmatrix}
\]
where the $y_i$'s are defined by the equation
\begin{equation*}
(b\sigma)^n(\mathfrak v) = v + \sum_{i=1}^{n-1} y_i (b\sigma)^i(\mathfrak v).
\end{equation*}
More precisely, if $\mu_{i,j}$ denotes the $(i,j)$th entry of $\det(g_b(\bar{x}))^{-1} M_{\sE, \sB_x}((m\ast)^{\prime})$, then for $1 \leq i,j \leq n$ we have
\[ 
\mu_{i,j} = \begin{cases} 1 & \text{if $j = 1$} \\ 0 & \text{if $i+j \leq n+1$ and $j > 1$} \\ \sigma^{-(i-1)}(y_{i-1}) & \text{if $i + j = n+2$} \\ \mu_{i-1,j} + \sigma^{-(i-1)}(y_{i-1}) \sigma^{n-(i-1)}(\mu_{n,j+i-(n+1)}) & \text{if $i+j \geq n+3$ and $i \geq 3$}. \end{cases}
\]
In particular, if $i+j \geq n+3$ and $y_{i-1} = 0$, then $\mu_{i,j} = \mu_{i-1,j}$.
\end{lm}

\begin{proof}[Proof of Lemma \ref{lm:explicit_construction_of_M_BxEmast}]
Let $z = \sum_{i=1}^n z_i (b\sigma)^{i-1}(x)$ be a generic element of $V$, written in $\sB_x$-coordinates, that is $z_{\sB_x}$ is the $n$-tuple $(z_i)_{i=1}^n$. 
The $(i,j)$th entry of $M_{\sE, \sB_x}((m\ast)^{\prime})$ is equal to $\sigma^{-(i-1)}$ applied to the coefficient of $\sigma^{i-1}(z_j)$ in the  $i$th entry of $(b\sigma)^{i-1}(z)_{\sB_x}$ ($=$ the $i$th entry of $m\ast z$). 

The coordinate matrix of the $\sigma$-linear operator $b\sigma \colon V \rightarrow V$ in the basis $\sB_x$,
\[ 
M_{\sB_x,\sB_x}(b\sigma) =  
\begin{pmatrix}
0 & 0 & \cdots & 0 & 1 \\
1 & 0 & \cdots & 0 & y_1 \\
0 & 1 & \ddots & \vdots & y_2 \\
\vdots & \ddots & \ddots & 0 & \vdots \\
0 & \cdots & 0 & 1 & y_{n-1}
\end{pmatrix}.
\]
That is, for any $z \in V$,
\begin{equation}\label{eq:sigma_lin_coord_change_for_a_vector}
b\sigma(z)_{\sB_x} = M_{\sB_x,\sB_x}(b\sigma) \cdot \sigma(z_{\sB_x}),
\end{equation}
where the last $\sigma$ is applied entry-wise. Explicitly, the first entry of $b\sigma(z)_{\sB_x}$ is $\sigma(z_n)$, and for $2 \leq i \leq n$ the $i$th entry of $b\sigma(z)_{\sB_x}$ is $\sigma(z_{i-1}) + y_{i-1}\sigma(z_n)$. This allows to iteratively compute $b\sigma^i(z)$ for all $i$, which we do to finish the proof. 

First, we see that $z_1$ can occur in the $n$th (i.e.\ last) entry of $(b\sigma)^{\lambda - 1}(z)_{\sB_x}$ only if $\lambda \geq n$; hence its contribution to the $i$th entry of $(b\sigma)^{i-1}(z)_{\sB_x}$ for $i \leq n$ is simply $\sigma^{i-1}(z_1)$. This shows that the first column of $M_{\sE, \sB_x}((m\ast)^{\prime})$ consists of $1$'s. Assume now $j \geq 2$. Then there is a smallest (if any) $i_0$ , such that $z_j$ occurs in the $i_0$th entry of $(b\sigma)^{i_0-1}(z)_{\sB_x}$. Note that as $j\geq 2$, one has $i_0 \geq 2$. Then $z_j$ must have been occurred in the $n$th entry of $(b\sigma)^{i_0 - 2}(z)_{\sB_x}$. As $z_j$ occurs in $z_{\sB_x}$ in exactly the $j$th entry, and it needs $(n-j)$ times to apply $b\sigma$ to get it to the $n$th entry, we must have $i_0 - 2 \geq n - j$. This shows that the $(i,j)$th entry of $M_{\sE, \sB_x}((m\ast)^{\prime})$ is $0$, unless $i \geq n+2-j$. The same consideration shows that if $i = n+2-j$, then $\sigma^{i-1}(z_j)$ has the coefficient $y_{i-1}$ in $\sigma_b^{i - 1}(z)_{\sB_x}$. This gives the entries of $M_{\sE, \sB_x}((m\ast)^{\prime})$ on the diagonal $i = n+2-j$. It remains to compute the entries below it, so assume $i > n + 2 - j$. Again, by the characterization of the entries of $M_{\sE, \sB_x}((m\ast)^{\prime})$ in the beginning of the proof and by the explicit description of how $\sigma_b$ acts (in the $\sB_x$-coordinates), it is clear that the $(i,j)$th entry of $M_{\sE, \sB_x}((m\ast)^{\prime})$ is just the sum of the $(i-1,j)$th entry and $\sigma^{-(i-1)}(y_{i-1})\sigma^{n-(i-1)}((n,j-1)\text{th entry})$. This finishes the proof of Lemma \ref{lm:explicit_construction_of_M_BxEmast}.
\end{proof}

Now we continue the proof of Lemma \ref{lm:linear_change_of_variables_fibers}. Let $\sC$ denote the ordered basis of $V$ consisting of columns $C_1, C_2, \dots, C_n$ of $C$. We have $M_{\sB_x,\sC} = (\det g_b(x))^{-1} m \cdot C$. In particular, to give the invertible matrix $C$ it is equivalent to give the invertible matrix $M_{\sB_x,\sC}$. But the $i$th column of $M_{\sB_x,\sC}$ is the coordinate vector of $C_i$ in the basis $\sB_x$, i.e., what we denoted $C_{i,\sB_x}$. We now show that one can find an invertible $M_{\sB_x,\sC}$, such that for its columns $C_{i,\sB_x}$ we have
\begin{align}
\nonumber m \ast C_{1,\sB_x} &= \textstyle\sum\limits_{\lambda=1}^n e_{\lambda} &\\
\label{eq:what_we_want_in_the_image_of_m_ast} m \ast C_{s+2 - j,\sB_x} &= \textstyle\sum\limits_{\lambda = i_j + 1}^{i_j} e_{\lambda} \quad &\text{for $s \geq j \geq 0$,} \\
\nonumber  m \ast C_{j^{\prime},\sB_x}   &= 0 \quad &\text{if $j^{\prime}>s+2$}.
\end{align}
Taking into account equation \eqref{eq:poly_P_basechanged_in_general} and the definition of $m \ast $ in \eqref{eq:def_m_ast}, this (plus the fact that $x \mapsto M_{\sB_x,\sC}$ will in fact an algebraic morphism) finishes the proof of Lemma \ref{lm:linear_change_of_variables_fibers}, except for the claim regarding the $W_h^{h-1}(\FF_{q^n})$-action.

To find $M_{\sB_x,\sC}$ satisfying \eqref{eq:what_we_want_in_the_image_of_m_ast}, first observe that by Lemma \ref{lm:explicit_construction_of_M_BxEmast}, there is some invertible matrix $S$ depending on $\tilde{x} \in X_{h-1,r}^+$ (in fact, only on its image $x \in X_1^+$), such that $M_{\sE ,\sB_x}((m\ast)^{\prime}) \cdot S$ has the following form: its first column consists of $1$'s; its $i$th column is $0$, unless $i = n + 1 - i_j$ for some $s \geq j \geq 0$; for $s \geq j \geq 0$, the $\lambda$th entry of its $(n + 1 - i_j)$th column is $1$ if $i_j +1 \leq \lambda \leq i_{j+1}$ (we put $i_{s+1} := n$ here) and zero otherwise. (To show this, use the general shape of $M_{\sE ,\sB_x}((m\ast)^{\prime})$ provided by Lemma \ref{lm:explicit_construction_of_M_BxEmast}, and then consecutively apply row operations to it and use the last statement of Lemma \ref{lm:explicit_construction_of_M_BxEmast}). Moreover, it is also clear from Lemma \ref{lm:explicit_construction_of_M_BxEmast} that $S$ will be upper triangular with the upper left entry $= 1$.

Secondly, let $T$ be a matrix 
%\[
%\begin{pmatrix}
%1 & 0 & 0 & \dots & 0 \\  
%0 & \ast & \ast & \dots & \ast \\
%&&\dots&& \\
%0 & \ast & \ast & \dots & \ast \\
%0 & 1 & 0 & \dots & 0 \\
%0 & \ast & \ast & \dots & \ast \\
%&&\dots&&
%\end{pmatrix}
%\]
such that: the first row has $1$ in the first position and zeros otherwise; all except for the first entry of the first column are $0$; for $s \geq j \geq 0$, the $(n + 1 - i_j)$th row has $1$ in the $(s + 2 - j)$th position and $0$'s otherwise; the remaining rows can be chosen arbitrarily. Obviously, $T$ can be chosen to be a permutation matrix with entries only $0$ or $1$, and in particular invertible and independent of $x$. Finally, put $M_{\sB_x,\sC} := S \cdot T$. Explicitly the columns of the matrix 
\begin{equation}\label{eq:some_auxil_matrix_A}
M_{\sE ,\sB_x}((m\ast)^{\prime}) \cdot M_{\sB_x,\sC} = (M_{\sE ,\sB_x}((m\ast)^{\prime}) \cdot S) \cdot T
\end{equation}
are as follows: the first column consist of $1$'s; for $s \geq j \geq 0$, the the $\lambda$th entry of the $(s + 2 - j)$th column is $1$ if $i_j +1 \leq \lambda \leq i_{j+1}$, and zero otherwise; all other columns consist of $0$'s. On the other side, the $j$th column of of $M_{\sE ,\sB_x}((m\ast)^{\prime}) \cdot M_{\sB_x,\sC}$ is precisely $m \ast C_{j,\sB_x}$ (up to the unessential $\sigma^{-\ast}$-twist in each entry). This justifies \eqref{eq:what_we_want_in_the_image_of_m_ast}.

\label{loc:action}
The action of $1 + \varpi^h a \in W_h^{h-1}(\FF_{q^n})$ on the coordinates $x_{i,h}$ is given by $(x_{i,h})_{i=1}^n \mapsto (x_{i,h} + ax_{i,0})_{i=1}^n$. We determine the action $1 + \varpi^h a$ in the coordinates $x^{\prime}_{i,h}$. Indeed, let $C^{-1} = (d_{i,j})_{1\leq i,j\leq n}$. Then $1 +\varpi^h a$ acts on $x_{i,h}^{\prime}$ by
\begin{align*}
x_{i,h}^{\prime} = \sum_{j=1}^n d_{i,j}x_{j,h} \mapsto \sum_{j=1}^n d_{i,j} (x_{j,h} + ax_{j,0}) = x_{i,h}^{\prime} + a\sum_{j=1}^n d_{i,j} x_{j,0}.
\end{align*}
Organizing the $x_{i,h}$ for $1 \leq i \leq n$ in one (column) vector, we can rewrite this as 
\[
1 +\varpi^h a \colon (x_{i,h}^{\prime})_{i=1}^n \mapsto (x_{i,h}^{\prime})_{i=1}^n + a C^{-1} \cdot x.
\]
We determine $C^{-1} \cdot x$. As $M_{\sB_x,\sC} = \det(g_b(x))^{-1} m C = g_b(x)^{-1}C$ (as $\det(g_b(x))^{-1} m = g_b(x)^{-1}$), we have $C^{-1} = M_{\sB_x,\sC}^{-1}g_b(x)^{-1}$. But $x$ is the first column of $g_b(x)$, thus 
\[ 
C^{-1} \cdot x = M_{\sB_x,\sC}^{-1}g_b(x)^{-1}\cdot x = M_{\sB_x,\sC}^{-1} \cdot (1, 0, \ldots, 0)^\trans, 
\]
so $C^{-1} \cdot x$ is the first column of $M_{\sB_x,\sC}^{-1} = (ST)^{-1} = T^{-1} S^{-1}$. But $S$ is upper triangular with upper left entry $=1$, so the first column of $M_{\sB_x,\sC}^{-1}$ is the first column of $T^{-1}$, which is $(1, 0, \ldots, 0)^\trans$. This finishes the proof of the lemma.
\end{proof}

The second part of the proof is given by the following lemma.

\begin{lm}\label{lm:non_linear_changes_of_variables}
Assume $\kappa = 0$. There exists a $X_{h-1,r}^+$-morphism $X_{h-1,r}^+ \times \bA^n \to X_{h-1,r}^+ \times \bA^n$ such that if $\{z_i\}$ denotes the coordinates on $\bA^n$ on the target $\bA^n$, then the image of $X_{h,r}^+$ in $X_{h-1,r}^+ \times \bA^n$ and the action of $\bW_h^{h-1}(\FF_{q^n})$ on $z_i$ are given by Proposition \ref{p:Mr}(ii),(iii). Moreover, such a morphism is given by the composition of the change-of-variables $x_{i,h}'$ and purely inseparable morphisms of the form $x_{i,h-1}' \mapsto x_{h-1}^{\prime,q^{-j}}$ for appropriate $i,j$.
\end{lm}

\begin{proof} If $r = n$, this is literally Lemma \ref{lm:linear_change_of_variables_fibers}. Assume $r<n$. First, for $s \geq j \geq 0$, replace $x_{s+2-j}^{\prime}$ by $x_{s+2-j}^{\prime, q^{i_j + 1}}$. Then, by applying a series of iterated changes of variables of the form $x_c^{\prime} =: x_c^{\prime} + x_d^{\prime, q^{\lambda}}$ for appropriate $2 \leq c,d \leq s+2$ and $\lambda$ (essentially following the Euclidean algorithm to find the gcd of the integers $(i_{j+1} - i_j)$ (this gcd is equal to $r$)), we transform $P_1$ from Lemma \ref{lm:linear_change_of_variables_fibers} to the form 
\[
P_1 = \sum_{i=0}^{n-1} x_{1,h}^{\prime, q^i} + \sum_{i=0}^{r-1} x_{2,h}^{\prime,q^i}.
\]
As these operations does not involve $x_{1,h}^{\prime}$, the formulas \eqref{eq:Wh_action_on_intermediate_coordinates} remain true. Now make the change of variables given by $z_1 := x_{2,h}^{\prime} + \sum_{j=0}^{\frac{n}{r} - 1} x_{1,h}^{\prime,q^{rj}}$ and $z_2 := x_{1,h-1}^{\prime}$. In this coordinates, $P_1 = \sum_{i=0}^{r-1} z_1^{q^i}$ and the action is as claimed.
\end{proof}

We are now ready to complete the proof of Proposition \ref{p:Mr}.

\begin{proof}[Proof of Proposition \ref{p:Mr}]
Combining Lemmas \ref{lm:linear_change_of_variables_fibers} and \ref{lm:non_linear_changes_of_variables} we obtain Proposition \ref{p:Mr} in the case $\kappa = 0$. Now let $\kappa$ be arbitrary. It is clear that the proof of Lemma \ref{lm:linear_change_of_variables_fibers} can be applied to this more general situation. One then obtains the same statement, with the only difference being that now our change of variables does not affect the variables $x_{i,h-1}$ for $i \not\equiv 1 \mod n_0$ (these are exactly the variables which do not show up in $P_1$). That is, the right-hand side $X_{h-1,r}^+ \times \bA^n$ will have the coordinates $\{x_{i,h-1}^{\prime} \colon i\equiv 1 \mod n_0, 1\leq i\leq n \} \cup \{x_{i,h-1} \colon i \not\equiv 1 \mod n_0, 1\leq i\leq n  \}$ and the polynomial defining $X_{h,r}^+$ as a relative $X_{h-1,r}^+$ hypersurface in $X_{h-1,r}^+ \times \bA^n$ is 
\[
P_1 = x_{1,h-1}^{\prime} + x_{1,h-1}^{\prime,q^{n_0}} + \dots + x_{1,h-1}^{\prime, q^{n_0(n^{\prime}-1)}} + \sum_{j=0}^s \sum_{\lambda = i_j + 1}^{i_{j+1}} x_{s+2-j,h-1}^{\prime, q^{n_0\lambda}},
\]
and the $\bW_h^{h-1}(\FF_{q^n})$-action is given by
\[
1 + \varpi^{h-1} a \colon x_{i,h-1}^{\prime} \mapsto \begin{cases} x_{1,h-1}^{\prime} + a & \text{if $i=1$} \\ x_{i,h-1}^{\prime} & \text{if $i \equiv 1 \mod n_0$ and $i > 1$} \\ x_{i,h-1} + x_{i,0} a & \text{if $i \not\equiv 1 \mod n_0$.} \end{cases}
\]
We now apply the change of variables replacing $x_{i,h-1}$ by $x_{i,h-1}' \colonequals x_{h-1} - x_{i,0} x_{1,h-1}'$ for all $i \not\equiv 1 \mod n_0$. This exactly gives us Lemma \ref{lm:linear_change_of_variables_fibers} for arbitrary $\kappa$ (the only difference being the $q^{n_0}$-powers occurring in $P_1$). Now Lemma \ref{lm:non_linear_changes_of_variables} can be applied as in the case $\kappa = 0$, and this finishes the proof of Proposition \ref{p:Mr}.
\end{proof}

\bibliography{bib_ADLV_CC}{}
\bibliographystyle{alpha}

\end{document}